\documentclass[11pt,twoside]{article}
\usepackage{times}
\usepackage{amsmath,amssymb,amsthm}
\usepackage{enumerate}
\usepackage{cite}
\usepackage{graphicx}
\usepackage{booktabs}
\usepackage{makecell}
\usepackage{epstopdf}
\usepackage{xcolor}

\allowdisplaybreaks

\newcommand{\R}{\mathbb R}
\newcommand{\N}{\mathbb N}

\newcommand{\p}{\partial}
\newcommand{\ve}{\varepsilon}
\newcommand{\f}{\frac}

\newcommand{\al}{\alpha}

\newcommand{\ds}{\displaystyle}

\allowdisplaybreaks

\theoremstyle{plain}
\newtheorem{theorem}{Theorem}[section]
\newtheorem{proposition}{Proposition}[section]
\newtheorem{lemma}[theorem]{Lemma}

\theoremstyle{definition}

\theoremstyle{remark}
\newtheorem{remark}{Remark}[section]

\numberwithin{equation}{section}

\title{Global existence of small data solutions to 3-D semilinear Euler-Poisson-Darboux equations}

\author{Li Qianqian,\quad Yin Huicheng\footnote{Li Qianqian (\texttt{214597007@qq.com}) and Yin Huicheng
    (\texttt{huicheng@} \texttt{nju.edu.cn}, \texttt{05407@njnu.edu.cn}) are supported by the NSFC
    (No.~12331007, No.~12571237).}\vspace{0.5cm}\\
\small
School of Mathematical Sciences, Nanjing Normal University, Nanjing 210023, China.\\
}
\vspace{0.5cm}

\begin{document}

\date{}

\maketitle
\thispagestyle{empty}

\begin{abstract}
There is an interesting open question:  for $n$-D ($n\ge 1$) semilinear Euler-Poisson-Darboux equation
$\p_t^2u-\Delta u+\f{\mu}{t}\p_tu=|u|^p$, where $t\ge 1$,  $p>1$ and $\mu>0$, the global small data weak solution $u$ will exist
when $p>p_{crit}(n,\mu)=\max\{p_s(n+\mu), p_f(n)\}$ with the Strauss exponent $p_{s}(n+\mu)=\f{n+\mu+1+\sqrt{(n+\mu)^2+10(n+\mu)-7}}{2(n+\mu-1)}$
and the Fujita exponent $p_f(n)=1+\f{2}{n}$. The blowup of weak solution $u$ has been shown when $1<p\le p_{crit}(n,\mu)$
meanwhile this open question has been solved for $n=1,2$.
In the present paper, we focus on this open question for $n=3$ and establish the global existence of small data solution $u$ for $\mu\geq\f{14}{5}$ (equivalent to $p_{crit}(3,\mu)=p_f(3)=\f53$) and $p>\max\{\f53, 1+\f{2}{\mu}\}$.
\end{abstract}

\noindent
\textbf{Keywords.} Euler-Poisson-Darboux equation, critical exponent, global existence, Bessel function,

\qquad \quad vector field, Klainerman-Sobolev inequality

\vskip 0.1 true cm

\noindent
\textbf{2010 Mathematical Subject Classification.} 35L70, 35L65, 35L67

\tableofcontents

\section{Introduction}
Consider  the $n$-D ($n\ge1$) semilinear Euler-Poisson-Darboux equation
\begin{equation}\label{equ:eff1-0}
\left\{ \enspace
\begin{aligned}
&\square u+\f{\mu}{t}\,\p_tu=|u|^p, \,\qquad t\geq1,\, x\in\Bbb R^n,\\
&u(1,x)=\ve u_0(x), \quad \partial_{t} u(1,x)=\ve u_1(x),
\end{aligned}
\right.
\end{equation}
where $\mu>0$,  $p>1$,
$\p_i=\p_{x_i}$ ($i=1, ..., n$), $\square=\partial_t^2-\Delta$ with $\Delta=\p_1^2+\cdot\cdot\cdot+\p_n^2$, $u_0, u_1\in
C_0^{\infty}(\R^n)$ with $\operatorname{supp} u_k\in B(0, 1)$  ($k=0, 1$), and $\ve>0$ is sufficiently small.
For the related physical backgrounds of the linear operator $\square +\f{\mu}{t}\,\p_t$ in \eqref{equ:eff1-0},
one can see \cite{DA-0}, \cite{LWY} and  references therein.
Denote by $p_{crit}(n,\mu)=\max\{p_s(n+\mu), p_f(n)\}$,
where  $p_s(z)=\f{z+1+\sqrt{z^2+10z-7}}{2(z-1)}$ ($z>1$)
is the Strauss exponent which comes from the positive root of the quadratic algebraic equation
$(z-1)p^2-(z+1)p-2=0$, and $p_f(n)=1+\f{2}{n}$ is the Fujita exponent.
It is pointed out that the Strauss exponent $p_s(n)=\f{n+1+\sqrt{n^2+10n-7}}{2(n-1)}$
originates from the  critical exponent problem of  global existence/blowup for the small data solution $v$ to $n$-D ($n\ge 2$)
semilinear wave equation
$\square v=|v|^p$ with $p>1$ (see \cite{Strauss}), meanwhile, the Fujita exponent $p_f(n)=1+\f{2}{n}$ stems from
the related $n$-D ($n\ge 2$) semilinear parabolic equation  $\partial_t v-\Delta v=|v|^p$ with $p>1$ (see \cite{Fuj}).
In terms of \cite{Rei1} and \cite{Imai}, there is an interesting open question as follows:

{\bf Open question.} {\it For $\mu>0$, when $p>p_{crit}(n,\mu)$, the small data weak solution $u$
of \eqref{equ:eff1-0} exists globally; otherwise, the solution $u$ may blow up in finite time when $1<p\le p_{crit}(n,\mu)$.
}

It is easy to know that $p_{crit}(n,\mu)=p_s(n+\mu)$ for $0<\mu\le\bar \mu(n)$ and
$p_{crit}(n,\mu)=p_f(n)$ for $\mu\ge\bar\mu(n)$ with $\bar\mu(n)=\frac{n^2+n+2}{n+2}$.
So far, the blowup results of the Open question have been systematically studied for $1<p\leq p_{crit}(n, \mu)$(see \cite{FLT}, \cite{IS},
\cite{LTW}, \cite{PR-0}, \cite{PR}, \cite{TL2} and \cite{W1}). On the other hand, the Open question has been solved for $n=1,2$ (see \cite{DA-0}, \cite{LWY},
\cite{HLY}, \cite{LY}, \cite{LY-1}, \cite{DA} and \cite{Rei1}).
However, there has been only a few  results concerning the global existence in Open question for
$n\ge 3$.

Note that for the 3-D problem
\begin{equation}\label{equ:eff1}
\left\{ \enspace
\begin{aligned}
&\square u+\f{\mu}{t}\,\p_tu=|u|^p, \,\qquad t\geq1,\, x\in\Bbb R^3,\\
&u(1,x)=\ve u_0(x), \quad \partial_{t} u(1,x)=\ve u_1(x),
\end{aligned}
\right.
\end{equation}
the corresponding open question can be stated as follows:

{\bf Open question (A).} {\it For problem \eqref{equ:eff1},

{\bf (A1)} when $\mu\ge \f{14}{5}$ and $p>p_f(3)=\f{5}{3}$, there exists a global solution $u$;

{\bf (A2)} when $0<\mu<\f{14}{5}$ and $p>p_s(3+\mu)$, the small solution $u$ exists globally.
}

In the present paper, we focus on Open problem {\bf (A1)}.
Our main result can be stated as follows.
\begin{theorem}\label{YH-1}
Let $\mu\geq\f{14}{5}$. If
$p>\max\{\f{5}{3}, 1+\f{2}{\mu}\}$, then
there exists a small constant $\varepsilon_0>0$ such that for $\ve<\ve_0$,
\eqref{equ:eff1} admits a unique global solution $u\in C\left([1, \infty) ; H^{2} (\mathbb R^3) \right) \cap C^{1}\left([1, \infty) ; H^{1}(\mathbb R^3) \right)\cap C^{2}\left([1, \infty) ; L^{2}(\mathbb R^3) \right)$.
\end{theorem}

\begin{remark}\label{re-0}
{\it Note that when $\mu\geq3$, $\max\{\f{5}{3}, 1+\f{2}{\mu}\}=p_f(3)$; when $\f{14}{5}\leq\mu<3$, $\max\{1+\f{2}{\mu}, \f{5}{3}\}=1+\f{2}{\mu}$.
Therefore, for $\mu\ge 3$, the Open question (A1) is solved. However, for $\f{14}{5}\leq\mu<3$, there still exists a gap
of $p_{crit}(3,\mu)-(1+\f{2}{\mu})=p_f(3)-(1+\f{2}{\mu})=\f{2(3-\mu)}{3\mu}\le\f{1}{21}$ due to our technical reasons
(the requirement  for the convergence of the integral $\int_{1}^{t} \tau^{\f{\mu}{2}+(-\f{\mu}{2})p}\mathrm{d}\tau<\infty$ when $p>1+\f{2}{\mu}$,
one can see \eqref{crit} of Section 5 below).}
\end{remark}

\begin{remark}\label{re-3}
{\it For problem \eqref{equ:eff1-0}, when $\mu\in (0,1)\cup (1,2)$ and $p>p_s(n+\mu)$,
as pointed in  \cite[Remark 1.14]{HLY}, the global existence of small data solution $u$ to problem \eqref{equ:eff1-0} has been obtained.
In our forthcoming work, we will study the Open problem {\bf (A2)} and the remaining case of $\f{14}{5}<\mu<3$ in the Open problem {\bf (A1)}
by establishing some delicate space-time weighted Strichartz estimates.}
\end{remark}

For clarity, when $p>p_{crit}(n,\mu)$, the known global existence results of problem \eqref{equ:eff1-0} can be
roughly listed in the following table.
\begin{table}[htbp]
\centering
\footnotesize
\renewcommand{\arraystretch}{1.8}
\begin{tabular}{cccc}
\toprule

References & Space dimension $n$ & Damping coefficient $\mu$ & Nonlinearity exponent $p$ \\
\midrule

\cite{DA}
&
\makecell{$n=1$\\ $n=2$\\ $n\ge3$}
&
\makecell{
$\mu\ge\bar{\mu}(1)=\dfrac{4}{3}$\\
$\mu\ge3$\\
$\mu\ge n+2$
}
&
\makecell{
$p>p_f(1)$\\
$p>p_f(2)$\\
$p_f(n)<p\le\dfrac{n}{n-2}$
}
\\
\midrule

\cite{DA-0}
&
$n=1$
&
$\mu>1$
&
$p>p_{crit}(1,\mu)$
\\
\midrule

\cite{Rei1}
&
\makecell{$n=2$\\$n=3$ (radial symmetric)}
&
$\mu=2$
&
$p>p_s(5)$
\\
\midrule

\cite{HLY}, \cite{LWY}-\cite{LY-1}
&
$n=2$
&
$1<\mu<3$
&
$p>p_{crit}(2,\mu)$
\\

\midrule

\cite{LZ}
&
\makecell{$n=3$\\ (radial symmetric)}
&
$\mu\in[\frac{3}{2},2)$
&
$p_s(3+\mu)<p\le2$
\\
\midrule

\cite{HLY}
&
\makecell{$n=3$}
&
$0<\mu<2, \mu\not=1$
&
$p>p_s(3+\mu)$
\\

\midrule

\cite{P1}
&
\makecell{$n\ge4$\\ (even dimensions, radial symmetric)}
&
$\mu=2$
&
$p_s(n+2)<p<p_f\left(\frac{n+1}{2}\right)$
\\
\midrule

\cite{HLY},\cite{HL}
&
$n\ge4$
&
$0<\mu<2$
&
$p>p_s(n+\mu)$
\\
\midrule

\cite{Rei2}
&
\makecell{$n\ge5$\\ (odd dimensions, radial symmetric)}
&
$\mu=2$
&
$p_s(n+\mu)<p<\min\left\{2,\frac{n+1}{n-3}\right\}$
\\

\bottomrule
\end{tabular}
\end{table}

On the other hand, the unsolved cases in the Open question before our result Theorem \ref{YH-1} are also given in the following table.

\begin{table}[htbp]
\centering
\footnotesize
\renewcommand{\arraystretch}{1.8}
\begin{tabular}{cccc}
\toprule

Space dimension $n$ & Damping coefficient $\mu$ & Nonlinearity exponent $p$ \\
\midrule

$n=3$
&
$\mu=1$\quad{or $\mu\in[2,\frac{14}{5}]$}
&
$p>p_{s}(3,\mu)$
\\

\midrule

$n=3$
&
$\mu\ge \frac{14}{5}$
&
$p>p_f(3)=\f53$
\\
\midrule

$n\ge 4$
&
$2\le\mu<\bar\mu(n)=\f{n^2+n+2}{n+2}$
&
$p>p_s(n+\mu)$
\\

\midrule

$n\ge4$
&
$\mu\ge\bar\mu(n)$
&
$p>p_f(n)$
\\

\bottomrule
\end{tabular}
\end{table}

To prove Theorem \ref{YH-1}, as in \cite{LY} for treating the 2-D semilinear Euler-Poisson-Darboux equations,
we will utilize the vector field method together with the delicate time-decay analysis for the
3-D  linear Euler-Poisson-Darboux equation $\Box v+\f{\mu}{t}\p_tv=F$ with $(v,\p_tv)(1,x)=(v_0,v_1)(x)$.
The vector field method comes from the studies on the long time existence or blowup of classical solutions to
the $n-$D ($n\ge 2$)  quasilinear wave equation
$\square \psi+\ds\sum_{l,j,k=0}^ng_{lj}^k\p_{lj}^2\psi\p_k\psi=0$ (see \cite{KS}-\cite{KP}).
Define the vector field $Z\in\{\p, L_0, L_j, \Omega_{kl}:  1\le j\le n, 1\le k<l\le n\}$,
where $\p=(\partial_0, \partial_1, \ldots, \partial_n)$ $=(\partial_t, \partial_{x_1}, \ldots, \partial_{x_n})$, $L_0=t \partial_t+x_1 \partial_1+\cdots+x_n \partial_n$,
$L_j=t\partial_j+x_j \partial_t$ and $\Omega_{k l}=x_k \partial_l-x_l \partial_k$.
Due to $[\square, \p]=0,\quad [\square, L_0]=2\square,\quad [\square, L_j]=0, \quad [\square, \Omega_{k l}]=0$,
then $Z^I\psi$ ($|I|\le N$ and $N\in\Bbb N$) satisfy
\begin{equation}\label{YSon-0}
\square Z^I\psi+\sum_{l,j,k=0}^n\sum_{|I'|+|I''|\le|I|}C_{lj}^k\p_{lj}^2Z^{I'}\psi\p_kZ^{I''}\psi=0.
\end{equation}
By the standard energy estimates for the nonlinear wave equations, $\|Z^I\p \psi\|_{L_x^2(\mathbb R^n)}$ can be obtained
in terms of the initial data.
It follows from the crucial Klainerman-Sobolev inequality
\begin{equation}\label{YSon-1}
|\p \psi(t,x)|\le \f{C}{(1+|t-r|)^{\f12}(1+t)^{\f{n-1}{2}}}\ds\sum_{|J|\le [\f{n}{2}]+1}\|Z^J\p \psi(t,x)\|_{L_x^2(\mathbb R^n)}
\end{equation}
that the sharp time-decay and spacetime-decay rates of $\p\psi$ are established, where $r=\sqrt{x_1^2+\cdot\cdot\cdot+x_n^2}$.
For more detailed introductions and more extensive applications of the vector field methods, one can see \cite{H}, \cite{KS}-\cite{KP} and \cite{Li-Chen}-\cite{LX}.
However, as illustrated in \cite{LY},  the commutator $[\square+\f{\mu}{t}\,\p_t, Z]$ can not be expressed as an efficient linear combination of
$\square+\f{\mu}{t}\,\p_t$ and $Z$. This leads to that the vector field method in nonlinear wave equation can not
be applied to problem \eqref{equ:eff1-0} directly.

Motivated by \cite{LX}, where the authors treated the lower bounds
of the lifespan of classical solutions to the Cauchy problem of the fully $n-$D ($n\ge 3$) nonlinear wave equation
$\square u=F(u,\p u, \p^2 u)$ with
$F(u,\p u, \p^2 u)=O(|(u, \p u, \p^2 u)|^{1+\lambda})$ ($\lambda\in\Bbb N$),  we will introduce the following vector field norms
for 3-D problem \eqref{equ:eff1}
\begin{equation}\label{norm}
\text{$\|v\|_{Z, s,(p, q)}=\sum_{|\alpha| \leq s}\left\|Z^\alpha v\right\|_{(p, q)}$ and
$\|v\|_{Z, s, p}=\|v\|_{Z, s,(p, p)}$ with $\|v\|_{(p,q)}=\|v(r \omega) r^{\frac{2}{p}}\|_{L^p([0,+\infty); L^q(S^2))}$},
\end{equation}
where $s\in\N_0, p,q\in[1,\infty], \omega=\f{x}{r}\in S^2$.
From \eqref{norm}, for $1 \le p < +\infty$ and $q = \infty$, one has
\begin{equation}\label{norm2}
\begin{aligned}
\|v\|_{(p, \infty)}
= \big( \int_{0}^{\infty} \big\|v(r\omega)\big\|_{L^{\infty}(S^{2})}^{p} r^{2} \mathrm{d}r \big)^{\frac{1}{p}}
= \big( \int_{0}^{\infty} \big( \operatorname*{ess\,sup}_{|\omega|=1} |v(r\omega)| \big)^{p} r^{2} \mathrm{d}r \big)^{\frac{1}{p}} ;
\end{aligned}
\end{equation}
for $p = q = \infty$,
\begin{equation}\label{norm3}
\|v\|_{(\infty, \infty)}
= \operatorname*{ess\,sup}_{\substack{0\le r<\infty \\ |\omega|=1}} |v(r\omega)|
= \operatorname*{ess\,sup}_{x\in\mathbb{R}^3} |v(x)| .
\end{equation}

However, compared with the 2-D case in \cite{LY}, there are some different techniques in the present paper.

\vskip 0.1 true cm

${\bf \bullet}$  For $\mu\ge 3$, when $\rho=-\frac{\mu-1}{2}\le -1$ is a negative integer,
the corresponding Hankel functions of order $\rho$ should be replaced by
the Bessel functions of the second kind. In this case, some related estimates will be carried out
by utilizing the basic properties of the Bessel functions of the second kind. Note that in \cite{LY}, due to $2<\mu<3$,
then $\rho=-\frac{\mu-1}{2}\in (-1, -\f12)\not\in\Bbb Z$ and the resulting Hankel functions can be applied directly.

\vskip 0.1 true cm

${\bf \bullet}$ For the 2-D case in \cite{LY}, the following Sobolev embedding theorem on $S^1$ plays a key role
(see (5.11) of \cite{LY})
\begin{equation}\label{L11-0}
\begin{aligned}
\|u(\tau, \cdot)\|_{(2, \infty)} \lesssim \|u\|_{Z, 1, 2}.
\end{aligned}
\end{equation}
However, \eqref{L11-0} is not applicable for $\Bbb S^2$ and $x\in\Bbb R^3$. Thanks to some delicate observations, we can obtain
the following Sobolev imbedding (see \eqref{Q11} of Section 2 below)
\begin{equation}\label{Q20-0}
\|u(\tau, \cdot)\|_{(2, \infty)}\lesssim\|u(\tau, \cdot)\|_{Z, 1, 2}+\|\p u(\tau, \cdot)\|_{Z, 1, 2},
\end{equation}
which will be crucial in proving Theorem \ref{YH-1}.

On the other hand, by the Bessel function and Fourier analysis tools, the explicit expression of
solution mapping to the 3-D problem \eqref{equ:eff1} can be written as
\begin{equation}\label{Q2-0}
\begin{aligned}
\mathcal{N} u(t, x)&=\ve\Psi_{0}(t, 1, D)u_{0}(x)+\ve\Psi_{1}(t, 1, D) u_{1}(x)
+\int_{1}^{t} \Psi_{1}(t, \tau, D)|u(\tau, x)|^{p} \mathrm{d} \tau,
\end{aligned}
\end{equation}
where the symbols of the pseudo-differential operators $\Psi_{0}(t, 1, D)$ and $\Psi_{1}(t, 1, D)$
are given in \eqref{equ:q9-3} of Section \ref{sec2-1} below.
Introducing the function space $X(T)$
$$
X(T)=\{u(t,x): \|u\|_{X(T)}<\infty, \operatorname{supp} u(x, t)\subset\{|x|\lesssim t\}\}
$$
with the norm
\begin{equation}\label{Q1-0}
\|u\|_{X(T)}=\begin{cases}
\sup _{t \in[1, T]}\big(t^{\f{\mu}{2}}\|u\|_{Z, 1, 2}+t^{\f{\mu}{2}}\|\p u\|_{Z, 1, 2}\big),&\f{14}{5}\leq\mu<3,\\
\sup _{t \in[1, T]}\big(t^{\f32-\f{3}{1+\delta}}\|u\|_{Z, 1, 2}+t^{\f32-\f{3}{1+\delta}}\|\p u\|_{Z, 1, 2}\big),&\mu\geq3,
\end{cases}
\end{equation}
where $T>1$ is any fixed number, $\delta>0$ is any fixed small constant.
By some involved estimates on $\|u\|_{X(T)}$, we can show that the mapping $\mathcal{N}$ in \eqref{Q2-0}
is contractible in a closed subspace of $X(T)$. Therefore, it follows from the fixed point principle that the 3-D
problem \eqref{equ:eff1} admits a global small solution $u\in C\left([1, \infty) ; H^{2}\right) \cap C^{1}\left([1, \infty) ; H^{1}\right)\cap C^{2}\left([1, \infty) ; L^{2}\right)$. It is worth pointing out that the gap in the case
$\f{14}{5}<\mu<3$ comes from the integrability condition of the time-integral in the Duhamel term (see \eqref{crit} below), which requires
$p>1+\f{2}{\mu}$.

\vskip 0.3 true cm

\noindent\textbf{Notations.}

\vskip 0.1 true cm

(1) For nonnegative quantities \(f\) and \(g\), \(f\lesssim g\) means \(f\le Cg\) for some generic positive constant \(C\) independent of $\ve$,
 and \(f\sim g\) means \(g\lesssim f\lesssim g\).

(2)
$
\partial=(\partial_t,\partial_1,\partial_2,\partial_3),
L_0=t\partial_t+x_1\partial_1+x_2\partial_2+x_3\partial_3,
L_j=t\partial_j+x_j\partial_t~(j=1,2,3), \text{and}\,
\Omega_{lk}=x_l\partial_k-x_k\partial_l~ (1\leq l<k\leq3)$,
$Z\in\{\partial, L_0, L_j, \Omega_{lk}: 1\le j\le 3, 1\leq l<k\leq3\}$.

(3) For \(\xi\in\mathbb R^3\) and \(t\ge\tau\ge1\), set
\[
D_1=\{\xi: |\xi|\ge1\},\qquad D_2=\{\xi: |\xi|\le1\le t|\xi|\},\qquad D_3=\{\xi: t|\xi|\le1\}
\]
and
\[
D'_1=\{\xi: \tau|\xi|\ge1\},\qquad D'_2=\{\xi: \tau|\xi|\le1\le t|\xi|\},\qquad D'_3=\{\xi: t|\xi|\le1\}.
\]

This paper is organized as follows. In Section 2, at first, by utilizing
the Bessel functions of the first kind and the second kind, and Hankel functions,
we give out the explicit representation of the solution $v$ to the linear homogeneous equation $\square v +\f{\mu}{t}\,\p_tv=0$
with $(v,\p_tv)(\tau,x)=(v_0(x), v_1(x))$. Meanwhile, some crucial properties of Bessel functions and Hankel functions
are listed. Secondly, in terms of the vector field norm in \eqref{norm}, a useful Sobolev imbedding on $\Bbb S^2$
is derived.
In Section 3, we establish a series of time-decay estimates for the solution $v$ to $\square v+\f{\mu}{t}\,\p_tv=0$
under the actions of the vector field $Z\in\{\p, L_0, L_j, \Omega_{lk}: 1\le j\le 3, 1\le l<k\le 3\}$.
In Section 4, several time-decay estimates are derived for the linear inhomogeneous
equation $\square w +\f{\mu}{t}\,\p_tw=F$ with vanishing initial data.
In Section 5, based on  the estimates in Sections 3-4,  Duhamel's principle and the contraction mapping principle,
the  proof of Theorem  \ref{YH-1} can be completed.

\section{ Preliminaries}\label{sec2-1}

\subsection{Explicit solution of 3-D homogeneous equation $\square v+\f{\mu}{t}\,\p_tv=0$}\label{sec2-1-1}

Let $v$ solve the following 3-D problem
\begin{equation}\label{equ:q1}
\left\{ \enspace
\begin{aligned}
&\square v+\f{\mu}{t}\,\p_tv=0, &&
t\geq \tau \geq1,\\
&v(\tau,x)=v_0(x), \quad \partial_{t} v(\tau,x)=v_1(x), &&x\in\R^3,
\end{aligned}
\right.
\end{equation}
where $\mu\geq \f{14}{5}$ and $v_k(x)\in C_0^{\infty}(B(0,1))$ $(k=0, 1)$.
Taking the Fourier transformation of $v(t,x)$ for the spatial variable $x$, one has from \eqref{equ:q1} that
\begin{equation}\label{equ:q2}
\left\{ \enspace
\begin{aligned}
&\partial_t^2 \hat{v} +|\xi|^{2} \hat{v} +\f{\mu}{t}\,\p_t\hat{v}=0, &&
t\geq \tau \geq 1,\\
&\hat{v}(\tau,\xi)=\hat{v}_0(\xi), \quad \partial_{t} \hat{v}(\tau,\xi)=\hat{v}_1(\xi), &&\xi\in\R^3.
\end{aligned}
\right.
\end{equation}
Assume that $\rho=-\f{\mu-1}{2}\in(-\infty, -\f{9}{10}]$ is  an integer, namely, $\mu\geq \f{14}{5}$ is  an odd integer. Then, by   \cite[Theorem 2.1]{Wirth-3} and \cite[Lemma 2.1]{LY}, we can write the solution $\hat{v}(t, \xi)$ of
\eqref{equ:q2} as
\begin{equation}\label{equ:q4}
\hat{v}(t, \xi)=\Psi_{0}(t, \tau, \xi) \hat{v}_{0}(\xi)+\Psi_{1}(t, \tau, \xi) \hat{v}_{1}(\xi),
\end{equation}
where the multipliers $\Psi_{j}(t, \tau, \xi)(j=0,1)$ and their time derivatives are given by
\begin{equation}\label{equ:q6}
\Psi_{j}(t, \tau, \xi)=(-1)^{j}\frac{i \pi}{4}|\xi|^{1-j}  \frac{t^{\rho}}{\tau^{\rho-1}}\left|\begin{array}{ll}
H_{\rho}^{+}(t|\xi|) & H_{\rho-1+j}^{+}(\tau|\xi|) \\
H_{\rho}^{-}(t|\xi|) & H_{\rho-1+j}^{-}(\tau|\xi|)
\end{array}\right|,
\end{equation}
\begin{equation}\label{equ:q9}
\partial_{t} \Psi_{j}(t, \tau, \xi)=(-1)^{j} \frac{i \pi}{4}|\xi|^{2-j} \frac{t^{\rho}}{\tau^{\rho-1}}\left|\begin{array}{ll}
H_{\rho-1}^{+}(t|\xi|) & H_{\rho-1+j}^{+}(\tau|\xi|)  \\
H_{\rho-1}^{-}(t|\xi|) & H_{\rho-1+j}^{-}(\tau|\xi|)
\end{array}\right|,
\end{equation}
and
\begin{equation}\label{equ:q9-1}
\begin{aligned}
&\partial_{t}^2 \Psi_{j}(t, \tau, \xi)=(-1)^{j} \frac{i \pi}{4}|\xi|^{3-j} \frac{t^{\rho}}{\tau^{\rho-1}}\left|\begin{array}{ll}
H_{\rho-2}^{+}(t|\xi|) & H_{\rho-1+j}^{+}(\tau|\xi|)  \\
H_{\rho-2}^{-}(t|\xi|) & H_{\rho-1+j}^{-}(\tau|\xi|)
\end{array}\right|+t^{-1}\p_t\Psi_{j}(t, \tau, \xi)
\end{aligned}
\end{equation}
with $i=\sqrt{-1}$.
Here,  $H_{\nu}^{\pm}(z)$ stand for the Hankel functions of order $\nu$ ($\nu\in\R$), namely,
\begin{equation}\label{equ:q9-2}
H_\nu^{\pm}(z) = J_\nu(z) \pm i Y_\nu(z),
\end{equation}
where $J_\nu(z)$ and $Y_\nu(z)$ are the Bessel functions of the first and second kinds, respectively.
For $\gamma\in \mathbb{C} \setminus \Bbb{Z}$, one has the following definitions
\begin{equation*}
\begin{aligned}
J_{\gamma}(z)&=(\f{z}{2})^{\gamma}\sum_{k=0}^{\infty}(-1)^k\f{z^{2k}}{4^kk!\Gamma(\gamma+k+1)},\\
Y_{\nu}(z)&=\f{J_{\nu}(z)cos(\nu\pi)-J_{-\nu}(z)}{sin(\nu\pi)}.
\end{aligned}
\end{equation*}
On the other hand, for $n\in \N$, it holds that from \cite[Section 10]{OLBC}, $$J_{-n}(z)=(-1)^nJ_{n}(z)=(-1)^n(\f{z}{2})^{n}\sum_{k=0}^{\infty}(-1)^k\f{z^{2k}}{4^kk!n(n+k+1)}$$ and
$$
Y_{-n}(z)=(-1)^nY_n(z) = \left.\frac{(-1)^n}{\pi}\frac{\partial J_\nu(z)}{\partial\nu}\right|_{\nu=n}
+ \left.\frac{1}{\pi}\frac{\partial J_\nu(z)}{\partial\nu}\right|_{\nu=-n}.
$$
Therefore, it follows from \eqref{equ:q6}-\eqref{equ:q9-1} and \eqref{equ:q9-2} that
\begin{equation}\label{equ:q9-3}
\begin{aligned}
\Psi_j(t,\tau,\xi)
&= (-1)^{j}\frac{i \pi}{4}|\xi|^{1-j}  \frac{t^{\rho}}{\tau^{\rho-1}}
\begin{vmatrix}
J_{\rho}(t|\xi|) + iY_{\rho}(t|\xi|) & J_{\rho-1+j}(\tau|\xi|) + iY_{\rho-1+j}(\tau|\xi|) \\
J_{\rho}(t|\xi|) - iY_{\rho}(t|\xi|) & J_{\rho-1+j}(\tau|\xi|) - iY_{\rho-1+j}(\tau|\xi|)
\end{vmatrix} \\
&= (-1)^{j+1}\frac{\pi}{2}|\xi|^{1-j}  \frac{t^{\rho}}{\tau^{\rho-1}}
\begin{vmatrix}
J_{\rho-1+j}(\tau|\xi|) & J_{\rho}(t|\xi|) \\
Y_{\rho-1+j}(\tau|\xi|) & Y_{\rho}(t|\xi|)
\end{vmatrix},
\end{aligned}
\end{equation}
\begin{equation}\label{equ:q9-4}
\partial_{t} \Psi_{j}(t, \tau, \xi)=(-1)^{j+1} \frac{\pi}{2}|\xi|^{2-j} \frac{t^{\rho}}{\tau^{\rho-1}}\left|\begin{array}{ll}
J_{\rho-1+j}(\tau|\xi|) & J_{\rho-1}(t|\xi|)  \\
Y_{\rho-1+j}(\tau|\xi|) & Y_{\rho-1}(t|\xi|)
\end{array}\right|
\end{equation}
and
\begin{equation}\label{equ:q9-5}
\partial_{t}^2 \Psi_{j}(t, \tau, \xi)=(-1)^{j+1} \frac{\pi}{2}|\xi|^{3-j} \frac{t^{\rho}}{\tau^{\rho-1}}\left|\begin{array}{ll}
J_{\rho-1+j}(\tau|\xi|) & J_{\rho-2}(t|\xi|)  \\
Y_{\rho-1+j}(\tau|\xi|) & Y_{\rho-2}(t|\xi|)
\end{array}\right|+t^{-1}\p_t\Psi_{j}(t, \tau, \xi),
\end{equation}
where $\rho\in(-\infty, -\f{9}{10}]$ is  an integer.

When $\rho=-\f{\mu-1}{2}\in(-\infty, -\f{9}{10}]$ is not an integer, that is,  $\mu\geq \f{14}{5}$ is not an odd integer,
by \cite[Section 2]{LY} we have that for $k,j=0,1$,
\begin{equation}\label{equ:q9-5'}
\partial_t^k \Psi_j(t,\tau,\xi)
=
\frac{\pi}{2}\csc(\rho\pi)
|\xi|^{1+k-j}
\frac{t^\rho}{\tau^{\rho-1}}
\begin{vmatrix}
J_{-(\rho-1+j)}(\tau|\xi|) & J_{-(\rho-k)}(t|\xi|) \\
(-1)^{1+k-j}J_{\rho-1+j}(\tau|\xi|) & J_{\rho-k}(t|\xi|)
\end{vmatrix}
\end{equation}
and
\begin{equation}\label{l-3}
\begin{aligned}
\p_t^2\Psi_{j}(t, \tau, \xi)
& =\frac{\pi}{2}|\xi|^{3-j} \frac{t^{\rho}}{\tau^{\rho-1}}\csc(\rho\pi)\left|\begin{array}{ll}
J_{-(\rho-1+j)}(\tau|\xi|) & J_{-(\rho-2)}(t|\xi|) \\
(-1)^{3-j}J_{\rho-1+j}(\tau|\xi|) & J_{\rho-2}(t|\xi|)
\end{array}\right|+ t^{-1}\p_t\Psi_{j}(t, \tau, \xi).
\end{aligned}
\end{equation}

In terms of \cite[Section 10]{OLBC}, we now recall the following asymptotic properties and recurrence relations for
$J_{\nu}(z)$, $H_{\nu}^{\pm}(z)$ and $Y_{n}(z)$ $(z>0)$.

\begin{lemma}\label{lem1}
It holds that

(i) for large $z\geq K > 0$,
\begin{equation}\label{equ:q14}
|J_{\nu}(z)|\lesssim z^{-\f{1}{2}};
\end{equation}

\quad for small $z$ with $0<z\leq c<1$,
\begin{equation}\label{equ:q15}
|J_{\nu}(z)|\lesssim z^{\nu}.
\end{equation}

(ii) for large $z\geq K > 0$,
\begin{equation}\label{equ:q12}
|H_{\nu}^{\pm}(z)|\lesssim z^{-\f{1}{2}};
\end{equation}

\quad for small $z$ with $0<z\leq c<1$,
\begin{equation}\label{equ:q13}
\left|H_{\nu}^{ \pm}(z)\right| \lesssim \begin{cases}z^{-|\nu|}, & \nu \neq 0, \\
-\ln z, & \nu=0.
\end{cases}
\end{equation}

(iii) for large $z\geq K > 0$,
\begin{equation}\label{equ:q14-1}
|Y_{n}(z)|\lesssim z^{-\f{1}{2}};
\end{equation}

\quad for small $z$ with $0<z\leq c<1$,
\begin{equation}\label{equ:q15-1}
|Y_{n}(z)| \lesssim \begin{cases}z^{-n}, & n \in \N, \\
-\ln z, & n=0.
\end{cases}
\end{equation}
\end{lemma}

\begin{lemma}\label{lem1-1}
$J_{\nu}(z)$, $H_{\nu}^{\pm}(z)$ and $Y_{n}(z)$ $(z>0)$ fulfill the following recurrence relations:
\begin{equation}\label{equ:q15-11}
\begin{aligned}
J_{\nu}'(z)
&=J_{\nu-1}(z)-\f{\nu}{z}J_{\nu}(z),\quad
(H_{\nu}^{\pm})'(z)
=H_{\nu-1}^{\pm}(z)-\f{\nu}{z}H_{\nu}^{\pm}(z),\\
Y_{n}'(z)
&=Y_{n-1}(z)-\f{\nu}{z}Y_{n}(z).
\end{aligned}
\end{equation}

\end{lemma}

\subsection{A useful Sobolev imbedding}\label{sec2-1-2}

In this subsection, we establish a useful Sobolev imbedding under the norms \eqref{norm}-\eqref{norm3}.

\begin{lemma}\label{YHC-05}
It holds that
\begin{equation}\label{Q10}
\left\|u(\tau, \cdot)\right\|_{(b, \infty)}\lesssim \|u(\tau, \cdot)\|_{Z, 1,2}+\|\p u(\tau, \cdot)\|_{Z, 1,2}, \quad 2\le b\le +\infty.
\end{equation}
\end{lemma}

\begin{proof}
By interpolation, it is enough to  prove \eqref{Q10} only for $b=2$ and $b=\infty$.

For $b=2$,  it follows from the Sobolev embedding theorem on $S^2$ that
\begin{equation}\label{Q11}
\begin{aligned}
\|u\|^2_{(2, \infty)}&=\int_0^\infty  \operatorname*{ess\,sup}_{|\omega|=1}|u(r\omega)|^2r^2\mathrm{d}r\\
&\lesssim\int_0^\infty\|u(r, \cdot)\|^2_{H^2(S^2)}r^2\mathrm{d}r\\
&=\int_0^\infty\int_{S^2}\left(|u|^2+|\nabla_{\omega}u|^2+|\nabla^2_{\omega}u|^2\right)\mathrm{d}\sigma r^2\mathrm{d} r\\
&\lesssim \int_{\R^3}|u|^2\mathrm{d}x+\sum_{1\leq k<j\leq3}\int_{\R^3}\f{|\Omega_{kj}u|^2}{|x|^2}\mathrm{d}x+\sum_{m<j, k<l}\int_{\R^3}\f{|\Omega_{mj}\Omega_{kl}u|^2}{|x|^4}\mathrm{d}x,
\end{aligned}
\end{equation}
where $r=|x|$ and $\mathrm{d} x=r^2\mathrm{d}r\mathrm{d}\sigma$.

We next treat each term in \eqref{Q11}. In terms of the definition of $\|u\|_{Z, 1, 2}$, one has
\begin{equation}\label{Q12}
\int_{\R^3}|u|^2\mathrm{d}x \leq \|u\|^2_{Z, 1, 2}.
\end{equation}
By Hardy's inequality and direct compuation, we arrive at
\begin{equation}\label{Q13}
\begin{aligned}
\int_{\R^3}\f{|\Omega_{mj}u|^2}{|x|^2}\mathrm{d}x&\lesssim \int_{\R^3}|\nabla(\Omega_{mj} u)|^2\mathrm{d} x
\lesssim \int_{\R^3}|\Omega_{mj}(\nabla u)|^2\mathrm{d} x+\int_{\R^3}|\nabla u|^2\mathrm{d} x
\lesssim \|\p u\|^2_{Z, 1, 2}.
\end{aligned}
\end{equation}
Note that for $1\leq m<j\leq3, 1\leq k<l\leq3$, it holds that
\begin{equation}\label{Q15}
\begin{aligned}
\Omega_{mj}\Omega_{kl}u&=(x_m\p_j-x_j\p_m)(x_k\p_l-x_l\p_k)u\\
&=x_m\p_j(x_k\p_lu)-x_m\p_j(x_l\p_ku)-x_j\p_m(x_k\p_lu)
+x_j\p_m(x_l\p_ku)
\end{aligned}
\end{equation}
and
\begin{equation}\label{Q16}
\begin{aligned}
x_m\p_j(x_k\p_lu)=x_m\delta_{jk}\p_lu+x_mx_k\p_j\p_lu,
\end{aligned}
\end{equation}
where  $\delta_{jk}$ denotes the Kronecker delta.

By \eqref{Q15} and \eqref{Q16}, $\Omega_{mj}\Omega_{kl}u$ can be written as
\begin{equation}\label{Q17}
\Omega_{mj}\Omega_{kl}u=\sum_{|\al|=2}c_\al(x)\p_{\al}u+\sum_{|\beta|=2}d_\beta(x)\p_{\beta}u,
\end{equation}
where $c_\al(x)$ and $d_\beta(x)$ satisfy
\begin{equation}\label{Q18}
|c_\al(x)|\lesssim r^2, \quad |d_\beta(x)|\lesssim r.
\end{equation}

Due to \eqref{Q17}-\eqref{Q18} and Hardy's inequality, one has
\begin{equation}\label{Q19}
\begin{aligned}
\int_{\R^3}\f{|\Omega_{mj}\Omega_{kl}u|^2}{|x|^4}\mathrm{d} x
&\lesssim  \int_{\R^3} \f{r^4|\nabla^2u(x)|^2}{r^4} \mathrm{d} x+\int_{\R^3} \f{r^2|\nabla u(x)|^2}{r^4} \mathrm{d} x\\
&=  \int_{\R^3} |\nabla^2u|^2 \mathrm{d} x+\sum_{k=1}^3\int_{\R^3} \f{|\p_k u|^2}{r^2} \mathrm{d} x\\
&\lesssim  \int_{\R^3} |\nabla^2u|^2 \mathrm{d} x+\sum_{k=1}^3\int_{\R^3} |\nabla(\p_ku)|^2\mathrm{d} x\\
&\lesssim \|\p u\|_{Z, 1, 2}^2.
\end{aligned}
\end{equation}
Substituting \eqref{Q12}-\eqref{Q13} and \eqref{Q19} into \eqref{Q11}  yields
\begin{equation}\label{Q20}
\|u(\tau, \cdot)\|_{(2, \infty)}\lesssim\|u(\tau, \cdot)\|_{Z, 1, 2}+\|\p u(\tau, \cdot)\|_{Z, 1, 2}.
\end{equation}
On the other hand, for $b=\infty$, it follows from the Klainerman-Sobolev inequality in \cite{KS} that
\begin{equation}\label{Q21}
\|u(\tau, \cdot)\|_{\infty} \lesssim \|u(\tau, \cdot)\|_{Z, 1,2}+\|\p u(\tau, \cdot)\|_{Z, 1,2}.
\end{equation}
Interpolating between \eqref{Q20} and \eqref{Q21} yields
\begin{equation}\label{Q22}
\begin{aligned}
\|u\|^{b}_{(b, \infty)}&=\int_0^\infty  (\operatorname*{ess\,sup}_{|\omega|=1}|u(r\omega)|)^{b}r^2\mathrm{d}r\\
&=\int_0^\infty  (\operatorname*{ess\,sup}_{|\omega|=1}|u(r\omega)|)^{b-2}(\operatorname*{ess\,sup}_{|\omega|=1}|u(r\omega)|)^{2}r^2\mathrm{d}r\\
&\leq \operatorname*{ess\,sup}_{|\omega|=1, 0\leq r<\infty}|u(r\omega)|^{b-2}\cdot\|u\|^2_{(2, \infty)}\\
&\leq \|u\|_{\infty}^{b-2}\|u\|^2_{(2, \infty)}\lesssim(\|u(\tau, \cdot)\|_{Z, 1, 2}+\|\p u(\tau, \cdot)\|_{Z, 1, 2})^{b}.
\end{aligned}
\end{equation}
Therefore, \eqref{Q10} is shown.

\end{proof}

\section{Time-decay estimates of solutions to 3-D homogeneous equation $\square v+\f{\mu}{t}\,\p_tv=0$}\label{sec3}

Let $v$ solve
\begin{equation}\label{equ:q16}
\left\{ \enspace
\begin{aligned}
&\square v +\f{\mu}{t}\,\p_tv=0, \quad
t\geq 1, \,x\in\R^3\\
&v(1,x)=v_0(x), \quad \partial_{t} v(1,x)=v_1(x),
\end{aligned}
\right.
\end{equation}
where $\mu\geq \f{14}{5}$ and $v_k(x)\in C_0^{\infty}(B(0, 1))$ $(k=0, 1)$. Note
that the expressions of $\Psi_j(t,\tau,\xi)$ $(j=0,1)$ in Section 2 depend on the parameter $\mu$. At first,
we consider the odd integer case of $\mu$.
\begin{lemma}\label{lem2}
Assume that $\mu\geq \f{14}{5}$ (i.e. $\mu\ge3$) is an odd integer and  $\varepsilon_1\in(0, 1)$ is any small fixed constant.
Then the solution
$v$ of \eqref{equ:q16} admits the following decay estimate:
\begin{equation}\label{equ:q17}
\|v(t, \cdot)\|_{Z, 1, 2} \lesssim t^{-\f{\mu}{2}}\left\|v_0\right\|_{Z, 1, 2}
+t^{-\f{\mu}{2}}\left\|v_1\right\|_{Z, 1,(\f{6}{5}, 2)}
+t^{\f{3}{2}-\f{3}{1+\varepsilon_1}}\left\|v_0+v_1\right\|_{Z, 1,(1+\varepsilon_1, 2)}.
\end{equation}
\end{lemma}
\begin{proof}
By the representation
\begin{equation}\label{formula}
\hat{v}(t, \xi)=\Psi_{0}(t, 1, \xi) \hat{v}_{0}(\xi)+\Psi_{1}(t, 1, \xi) \hat{v}_{1}(\xi)
\end{equation}
and Lemma \ref{lem1}, we decompose the frequency space $\xi\in\Bbb R^3$ into the following three zones
\begin{equation}\label{equ:q19}
D_1=\{\xi: |\xi| \leq 1 \leq t|\xi|\}, \quad D_2=\{\xi: t|\xi|\geq|\xi|\ge 1\}, \quad D_3=\{\xi: |\xi|\leq t|\xi| \leq 1\}.
\end{equation}
Note that
\begin{equation}\label{YHC-01}
\|v\|_{Z, 1, 2} \leq \|\hat{v}\|_{L^2}+\|\hat{\p v}\|_{L^2}+\|\widehat{L_0v}\|_{L^2}
+\sum_{j=1}^3\|\widehat{L_jv}\|_{L^2}+\sum_{1\leq k<j\leq3}\|\widehat{\Omega_{kj}v}\|_{L^2}.
\end{equation}
We next divide the proof procedure of \eqref{equ:q17} into five parts, which just correspond to the treatment of five terms
on the right-hand side of \eqref{YHC-01}.
\vskip 0.2 true cm

{\bf Part 1.  The treatment  of $\|v(t, \cdot)\|_{L^2(\R^3)}$ }

\vskip 0.2 true cm

We shall prove
\begin{equation}\label{equ:q18}
\|v(t, \cdot)\|_{L^2(\R^3)} \lesssim t^{-\f{\mu}{2}}\left\|v_0\right\|_{L^2(\R^3)}
+t^{-\f{\mu}{2}}\left\|v_1\right\|_{(\f{6}{5}, 2)}
+t^{\f{3}{2}-\f{3}{1+\varepsilon_1}}\left\|v_0+v_1\right\|_{(1+\varepsilon_1, 2)}.
\end{equation}

\vskip 0.2 true cm

{\bf Part 1.1.  The analysis  in zone $D_1$}

\vskip 0.2 true cm

Since $\rho=-\f{\mu-1}{2}\in(-\infty, -1]$ is an integer, it follows from \eqref{equ:q9-3} and Lemma \ref{lem1} that for $j=0,1$,
\begin{equation}\label{e:q1}
\begin{aligned}
|\Psi_j(t,\tau,\xi)|
&\lesssim\left| |\xi|^{1-j} t^{\rho}
\begin{vmatrix}
(-1)^{1-\rho-j}J_{1-\rho-j}(|\xi|) & (-1)^{-\rho}J_{-\rho}(t|\xi|) \\
(-1)^{1-\rho-j}Y_{1-\rho-j}(|\xi|) & (-1)^{-\rho}Y_{-\rho}(t|\xi|)
\end{vmatrix} \right|\\
&\lesssim |\xi|^{1-j} t^{\rho}\left[|\xi|^{1-\rho-j}(t|\xi|)^{-\f12}+(t|\xi|)^{-\f12}|\xi|^{\rho-1+j}\right]\\
&\lesssim |\xi|^{1-j} t^{\rho}(t|\xi|)^{-\f12}|\xi|^{\rho-1+j}\\
&\lesssim t^{-\f{\mu}{2}}|\xi|^{-\f{\mu}{2}}.
\end{aligned}
\end{equation}
For $\mu\geq3$ and $t|\xi|\geq 1$, set
$x=ty$ and $\eta=t\xi$, one then has
\begin{equation}\label{e:q2}
\begin{aligned}
\big\|\hat{v}(t, \cdot)\big\|_{L^2(D_1)} & \leq\big\|(t|\xi|)^{-\f{\mu}{2}}(\hat{v}_{0}+\hat{v}_{1})\big\|_{L^{2}(D_1)} \\& \lesssim\big\|(1+t|\xi|)^{-\f32}(\hat{v}_{0}+\hat{v}_{1})\big\|_{L^{2}(D_1)} \\
& \lesssim t^{\f32}\big(\int_{R^3} (\frac{1}{1+|\eta|})^3\big(\int_{R^3} e^{-i y \cdot \eta} (v_0+v_1)(ty) \mathrm{d} y\big)^2 \mathrm{d} \eta\big)^{\frac{1}{2}}\\
& \lesssim t^{\f32}\|(v_0+v_1)(ty)\|_{H^{-\f32}(\R^3)}.
\end{aligned}
\end{equation}
By the following dual characterization of $H^{-\f32}(\R^3)$, $$\|V_1\|_{H^{-\f32}}=\sup _{\substack{v \in H^\f32 \\ v \neq 0}} \frac{\left|\int_{\mathbb{R}^3} V_1(y) v(y) \mathrm{d} y\right|}{\|v\|_{H^{\f32}}}\qquad \text{for}\,\, V_1(y)=(v_0+v_1)(ty),$$
together with H\"{o}lder's inequality and Sobolev imbedding theorem, we obtain
\begin{equation}\label{e:q3}
\begin{aligned}
\|V_1 v\|_{L^1(\R^3)} &\lesssim\|V_1\|_{(q, 2)}\|v\|_{\left(q^{\prime}, 2\right)} \lesssim\|V_1\|_{(q, 2)}\|v\|_{q^{\prime}} \lesssim\|V_1\|_{(q, 2)}\|v\|_{H^{\f32}}\\
&\lesssim t^{-\f{3}{q}}\|v_0+v_1\|_{(q, 2)}\|v\|_{H^{\f32}},
\end{aligned}
\end{equation}
where $q=1+\varepsilon_1$  and $\f{1}{q}+\f{1}{q'}=1$. Thus,
$$
\|V_1\|_{H^{-\f32}}\lesssim t^{-\f{3}{1+\varepsilon_1}}\|v_0+v_1\|_{(1+\varepsilon_1, 2)}.
$$
From this and \eqref{e:q2}, one arrives at
\begin{equation}\label{e:q4}
\left\|\hat{v}(t, \cdot)\right\|_{L^2(D_1)}\lesssim t^{\f32-\f{3}{1+\varepsilon_1}}\|v_0+v_1\|_{(1+\varepsilon_1, 2)}.
\end{equation}

\vskip 0.2 true cm

{\bf Part 1.2.  The analysis  in zone $D_2$}

\vskip 0.2 true cm

From \eqref{equ:q6} with $\rho=-\f{\mu-1}{2}$ and \eqref{equ:q12}, it holds that
\begin{equation}\label{e:q5}
\begin{aligned}
& \left|\Psi_0(t, 1,\xi)\right| \lesssim|\xi| t^\rho|\xi|^{-\frac{1}{2}}(t|\xi|)^{-\frac{1}{2}}=t^{-\frac{\mu}{2}}, \\
& \left|\Psi_1(t, 1, \xi)\right|  \lesssim t^\rho|\xi|^{-\frac{1}{2}}(t|\xi|)^{-\frac{1}{2}}=t^{-\frac{\mu}{2}}|\xi|^{-1}.
\end{aligned}
\end{equation}
Then
\begin{equation}\label{e:q6}
\begin{aligned}
\left\|\hat{v}(t, \cdot)\right\|_{L^2(D_2)}
& \lesssim t^{-\frac{\mu}{2}}\left\|v_{0}\right\|_{L^{2}}+t^{-\f{\mu}{2}} \|(1+|\xi|^2)^{-\f{1}{2}}\hat{v}_{1}\|_{L^{2}}
= t^{-\frac{\mu}{2}}\left\|v_{0}\right\|_{L^{2}}+ A.
\end{aligned}
\end{equation}
We now treat the term $A$ in \eqref{e:q6}. By the transformation of $x=ty$ and $\xi=\f{\eta}{t}$, one has that for $\mu\geq 3$ and $t\geq1$,
\begin{equation}\label{e:q7}
\begin{aligned}
A & =t^{-\f{\mu}{2}} \left(\int_{R^3} \frac{1}{1+|\xi|^2}\big(\int_{R^3} e^{-i x \cdot \xi} v_1(x) \mathrm{d} x\big)^2 \mathrm{d}\xi\right)^{\frac{1}{2}}\\
& =t^{-\f{\mu}{2}+\f52}\left(\int_{R^3} \frac{1}{t^2+|\eta|^2}\big(\int_{R^3} e^{-i y \cdot \eta} v_1(ty) \mathrm{d} y\big)^2 \mathrm{d} \eta\right)^{\frac{1}{2}}\\
& \leq t^{-\f{\mu}{2}+\f52}\left(\int_{R^3} \frac{1}{1+|\eta|^2}\big(\int_{R^3} e^{-i y \cdot \eta} v_1(ty) \mathrm{d} y\big)^2 \mathrm{d} \eta\right)^{\frac{1}{2}}\\
& \lesssim t^{-\f{\mu}{2}+\f52}\|v_1(ty)\|_{H^{-1}(\R^3)}.
\end{aligned}
\end{equation}
Due to $\|V_1\|_{H^{-1}}=\sup _{v \in H^1, v \neq 0} \frac{\left|\int_{\mathbb{R}^3} V_1(y) v(y) \mathrm{d} y\right|}{\|v\|_{H^1}}$
with $V_1(y)=v_1(ty)$, then it follows from H\"{o}lder's inequality and Sobolev imbedding theorem that
\begin{equation}\label{e:q8}
\|V_1 v\|_{L^1(\R^3)} \lesssim\|V_1\|_{(\f65, 2)}\|v\|_{\left(6, 2\right)} \lesssim\|V_1\|_{(\f65, 2)}\|v\|_{6} \lesssim\|V_1\|_{(\f65, 2)}\|v\|_{H^1}\lesssim t^{-\f{5}{2}}\|v_1\|_{(\f65, 2)}\|v\|_{H^1}.
\end{equation}
Hence,
$$
\|V_1\|_{H^{-1}}\lesssim t^{-\f{5}{2}}\|v_1\|_{(\f65, 2)}.
$$
This, together with \eqref{e:q7}, yields
$$
I\lesssim t^{-\f{\mu}{2}}\|v_1\|_{(\f65, 2)},
$$
and then
\begin{equation}\label{e:q9}
\left\|\hat{v}(t, \cdot)\right\|_{L^2(D_2)}\lesssim t^{-\frac{\mu}{2}}\left\|v_{0}\right\|_{L^{2}}+ t^{-\f{\mu}{2}}\|v_1\|_{(\f65, 2)}.
\end{equation}

\vskip 0.2 true cm

{\bf Part 1.3.  The analysis  in zone $D_3$}

\vskip 0.2 true cm

Since $\rho\in(-\infty,-1]$ is an integer, we deduce from \eqref{equ:q9-3}, \eqref{equ:q15} and \eqref{equ:q15-1} that for $j=0,1$,
\begin{equation}\label{e:q10}
\begin{aligned}
|\Psi_j(t,\tau,\xi)|
&\lesssim\left| |\xi|^{1-j} t^{\rho}
\begin{vmatrix}
(-1)^{1-\rho-j}J_{1-\rho-j}(|\xi|) & (-1)^{-\rho}J_{-\rho}(t|\xi|) \\
(-1)^{1-\rho-j}Y_{1-\rho-j}(|\xi|) & (-1)^{-\rho}Y_{-\rho}(t|\xi|)
\end{vmatrix} \right|\\
&\lesssim t^{2\rho}|\xi|^{2-2j}+1
\lesssim 1.
\end{aligned}
\end{equation}
Then for $|\xi|\leq t|\xi| \leq 1$, as in \eqref{e:q2}, on can obtain
\begin{equation}\label{e:q11}
\begin{aligned}
\left\|\hat{v}(t, \cdot)\right\|_{L^2(D_3)} & \leq\left\|(1+t|\xi|)^{\f32}\times(1+t|\xi|)^{-\f32}(\hat{v}_{0}+\hat{v}_{1})\right\|_{L^{2}(D_3)} \\
&\lesssim\left\|(1+t|\xi|)^{-\f32}(\hat{v}_{0}+\hat{v}_{1})\right\|_{L^{2}}\\
& \lesssim t^{\f32}\|(v_0+v_1)(ty)\|_{H^{-\f32}}\\
&\lesssim t^{\f32-\f{3}{1+\varepsilon_1}}\|v_0+v_1\|_{(1+\varepsilon_1, 2)}.
\end{aligned}
\end{equation}
Hence, \eqref{e:q18} follows directly from \eqref{e:q4}, \eqref{e:q9}, \eqref{e:q11} and Parseval's identity.

\vskip 0.2 true cm

{\bf Part 2.  The treatment  of $\|\p v(t, \cdot)\|_{L^2(\R^3)}$ }

\vskip 0.2 true cm

We next derive the following estimate
\begin{equation}\label{e:q12}
\|\p v(t, \cdot)\|_{L^2(\R^3)} \lesssim t^{-\f{\mu}{2}}\left\|v_0\right\|_{Z, 1, 2}+t^{-\f{\mu}{2}}\left\|v_1\right\|_{Z, 1,(\f65, 2)}+t^{\f32-\f{3}{1+\varepsilon_1}}\left\|v_0+v_1\right\|_{Z, 1,(1+\varepsilon_1, 2)}.
\end{equation}
To prove \eqref{e:q12}, we still
decompose the space $\R_{\xi}^3$ into three zones as in \eqref{equ:q19}.

\vskip 0.2 true cm

{\bf Part 2.1.  The analysis  in zone $D_1$}

\vskip 0.2 true cm

From \eqref{equ:q9-4}, Lemma \ref{lem1} (i) and (iii), it holds that for $j=0,1$,
\begin{equation}\label{e:q13}
\begin{aligned}
|\p_t\Psi_j(t,\tau,\xi)|
&\lesssim\left| |\xi|^{2-j} t^{\rho}
\begin{vmatrix}
(-1)^{1-\rho-j}J_{1-\rho-j}(|\xi|) & (-1)^{1-\rho}J_{1-\rho}(t|\xi|) \\
(-1)^{1-\rho-j}Y_{1-\rho-j}(|\xi|) & (-1)^{1-\rho}Y_{1-\rho}(t|\xi|)
\end{vmatrix} \right|\\
&\lesssim |\xi|^{2-j} t^{\rho}\left[|\xi|^{1-\rho-j}(t|\xi|)^{-\f12}+(t|\xi|)^{-\f12}|\xi|^{\rho-1+j}\right]\\
&\lesssim t^{-\f{\mu}{2}}|\xi|^{-\f{\mu}{2}}\cdot|\xi|\\
&\leq t^{-\f{\mu}{2}}|\xi|^{-\f{\mu}{2}}.
\end{aligned}
\end{equation}
Then for $|\xi| \leq 1\leq t|\xi|$ and $\mu\geq3$, as in \eqref{e:q2}, we have that for $k=1, 2, 3$,
\begin{equation}\label{e:q14}
\begin{aligned}
\left\|\p_t\hat{v}(t, \cdot)\right\|_{L^2(D_1)} & \leq\left\|(1+t|\xi|)^{-\f32}(\hat{v}_{0}+\hat{v}_{1})\right\|_{L^{2}} \\
& \lesssim t^{\f32}\|(v_0+v_1)(ty)\|_{H^{-\f32}}\\
& \lesssim t^{\f32-\f{3}{1+\varepsilon_1}}\|v_0+v_1\|_{(1+\varepsilon_1, 2)}
\end{aligned}
\end{equation}
and
\begin{equation}\label{e:q15}
\left\||\xi_k|\hat{v}(t, \cdot)\right\|_{L^2(D_1)}\leq\left\|\hat{v}(t, \cdot)\right\|_{L^2(D_1)}
\lesssim  t^{\f32-\f{3}{1+\varepsilon_1}}\|v_0+v_1\|_{(1+\varepsilon_1, 2)}.
\end{equation}
Combining  \eqref{e:q14} and \eqref{e:q15} yields
\begin{equation}\label{e:q16}
\|\widehat{\p v}\|_{L^2(D_1)}\leq \left\|\p_t\hat{v}(t, \cdot)\right\|_{L^2(D_1)}+\sum_{j=1}^{3}\left\||\xi_i|\hat{v}(t, \cdot)\right\|_{L^2(D_1)} \lesssim t^{\f32-\f{3}{1+\varepsilon_1}}\|v_0+v_1\|_{(1+\varepsilon_1, 2)}.
\end{equation}

\vskip 0.2 true cm

{\bf Part 2.2.  The analysis  in zone $D_2$}

\vskip 0.2 true cm

By  \eqref{equ:q9} and Lemma \ref{lem1}(ii), we obtain that for $j=0,1$,
\begin{equation}\label{e:q17}
\begin{aligned}
|\p_t\Psi_j(t,\tau,\xi)|
\lesssim |\xi|^{2-j} t^{\rho}\left[|\xi|^{-\f12}(t|\xi|)^{-\f12}+(t|\xi|)^{-\f12}|\xi|^{-\f12}\right]\lesssim t^{-\f{\mu}{2}}|\xi|^{1-j}.
\end{aligned}
\end{equation}
It follows from \eqref{e:q7} and the transformation of variables $(x, \eta)=(ty, t\xi)$  that for $k=1, 2, 3$,
\begin{equation}\label{e:q18}
\begin{aligned}
\left\|\p_t\hat{v}\right\|_{L^2(D_2)} & \lesssim  t^{-\f{\mu}{2}}\left\||\xi|\hat{v}_{0}\right\|_{L^{2}}+t^{-\f{\mu}{2}}\left\||\xi|^{-1}|\xi|\hat{v}_{1}\right\|_{L^{2}} \\
& \lesssim  t^{-\f{\mu}{2}}\left\|v_{0}\right\|_{Z, 1, 2}+t^{-\f{\mu}{2}+1}\left\|(1+t|\xi|)^{-1}|\xi|\hat{v}_{1}\right\|_{L^{2}} \\
& \lesssim  t^{-\f{\mu}{2}}\left\|v_{0}\right\|_{Z, 1, 2}+t^{-\f{\mu}{2}+\f52}\left\|\nabla v_{1}(ty)\right\|_{H^{-1}} \\
& \lesssim  t^{-\f{\mu}{2}}\left\|v_{0}\right\|_{Z, 1, 2}+t^{-\f{\mu}{2}}\left\| v_{1}\right\|_{Z, 1, (\f65, 2)}
\end{aligned}
\end{equation}
and
\begin{equation}\label{e:q19}
\begin{aligned}
\left\||\xi_k|\hat{v}(t, \cdot)\right\|_{L^2(D_2)}&\lesssim t^{-\frac{\mu}{2}}\left\||\xi_k|\hat{v}_{0}\right\|_{L^{2}}+t^{-\f{\mu}{2}} \||\xi|^{-1}|\xi_k|\hat{v}_{1}\|_{L^{2}(D_2)}\\
& \lesssim  t^{-\f{\mu}{2}}\left\|v_{0}\right\|_{Z, 1, 2}+t^{-\f{\mu}{2}+1}\left\|(1+t|\xi|)^{-1}|\xi|\hat{v}_{1}\right\|_{L^{2}} \\
& \lesssim  t^{-\f{\mu}{2}}\left\|v_{0}\right\|_{Z, 1, 2}+t^{-\f{\mu}{2}}\left\| v_{1}\right\|_{Z, 1, (\f65, 2)}.
\end{aligned}
\end{equation}
Hence,
\begin{equation}\label{e:q20}
\|\widehat{\p v}\|_{L^2(D_2)} \leq \left\|\p_t\hat{v}(t, \cdot)\right\|_{L^2(D_2)}+\sum_{k=1}^{3}\left\||\xi_k|\hat{v}(t, \cdot)\right\|_{L^2(D_2)}\lesssim  t^{-\f{\mu}{2}}\left\|v_{0}\right\|_{Z, 1, 2}+t^{-\f{\mu}{2}}\left\| v_{1}\right\|_{Z, 1, (\f65, 2)}.
\end{equation}

\vskip 0.2 true cm

{\bf Part 2.3.  The analysis  in zone $D_3$}

\vskip 0.2 true cm

For $t|\xi|\le 1$ and integer $\rho\in(-\infty,-1]$, it follows from \eqref{equ:q9-4}, Lemma \ref{lem1} (i) and (iii) that for $j=0,1$,

\begin{equation}\label{e:q21}
\begin{aligned}
|\p_t\Psi_j(t,\tau,\xi)|
&\lesssim \left| |\xi|^{2-j} t^{\rho}
\begin{vmatrix}
(-1)^{1-\rho-j}J_{1-\rho-j}(|\xi|) & (-1)^{1-\rho}J_{1-\rho}(t|\xi|) \\
(-1)^{1-\rho-j}Y_{1-\rho-j}(|\xi|) & (-1)^{1-\rho}Y_{1-\rho}(t|\xi|)
\end{vmatrix} \right|\\
&\lesssim |\xi|^{2-j} t^{\rho}\left[|\xi|^{1-\rho-j}(t|\xi|)^{\rho-1}+(t|\xi|)^{1-\rho}|\xi|^{\rho-1+j}\right]\\
&\lesssim t^{2\rho-1}|\xi|^{2-2j}+t|\xi|^2\\
&\lesssim t^{-1}.
\end{aligned}
\end{equation}
As treated in \eqref{e:q11}, we deduce that for $k=1, 2, 3$,
\begin{equation}\label{e:q22}
\begin{aligned}
\left\|\p_t\hat{v}\right\|_{L^2(D_3)} & \leq t^{-1}\left\|(1+t|\xi|)^{\f32}\times(1+t|\xi|)^{-\f32}(\hat{v}_{0}+\hat{v}_{1})\right\|_{L^{2}} \\
&\lesssim t^{-1}\left\|(1+t|\xi|)^{-\f32}(\hat{v}_{0}+\hat{v}_{1})\right\|_{L^{2}}\\
& \lesssim t^{\f32} \|(v_0+v_1)(ty)\|_{H^{-\f32}(\R^3)}\\
&\lesssim t^{\f12-\f{3}{1+\varepsilon_1}}\|v_0+v_1\|_{(1+\varepsilon_1, 2)}
\end{aligned}
\end{equation}
and
$$
\left\||\xi_k|\hat{v}(t, \cdot)\right\|_{L^2(D_3)}\leq t^{-1}\left\|\hat{v}(t, \cdot)\right\|_{L^2(D_3)} \lesssim t^{\f12-\f{3}{1+\varepsilon_1}}\|v_0+v_1\|_{(1+\varepsilon_1, 2)}.
$$
This, together with \eqref{e:q22}, yields
\begin{equation}\label{e:q23}
\|\widehat{\p v}\|_{L^2(D_3)}\lesssim t^{\f12-\f{3}{1+\varepsilon_1}}\|v_0+v_1\|_{(1+\varepsilon_1, 2)}.
\end{equation}
From Parts 2.1-2.3, we conclude that
\begin{equation}\label{e:q24}
\begin{aligned}
\|\p v(t, \cdot)\|_{L^2(\R^3)}&=\|\widehat{\p v}\|_{L^2(\R^3)} \leq\|\widehat{\p v}\|_{L^2(D_1)}
+\|\widehat{\p v}\|_{L^2(D_2)}+\|\widehat{\p v}\|_{L^2(D_3)}\\
& \lesssim t^{-\f{\mu}{2}}\left\|v_0\right\|_{Z, 1, 2}+t^{-\f{\mu}{2}}\left\|v_1\right\|_{Z, 1,(\f65, 2)}+t^{\f32-\f{3}{1+\varepsilon_1}}\left\|v_0+v_1\right\|_{Z, 1,(1+\varepsilon_1, 2)}.
\end{aligned}
\end{equation}

\vskip 0.2 true cm

{\bf Part 3.  The treatment  of $\|L_0v(t, \cdot)\|_{L^2(\R^3)}$ }

\vskip 0.2 true cm

We next derive the following estimate
\begin{equation}\label{e:q25}
\|L_0v(t, \cdot)\|_{L^2(\R^3)} \lesssim t^{-\f{\mu}{2}}\left\|v_0\right\|_{Z, 1, 2}+t^{-\f{\mu}{2}}\left\|v_1\right\|_{Z, 1,(\f65, 2)}+t^{\f32-\f{3}{1+\varepsilon_1}}\left\|v_0+v_1\right\|_{Z, 1,(1+\varepsilon_1, 2)}.
\end{equation}
Due to
\begin{equation}\label{e:q26}
\begin{aligned}
\|L_0v(t, \cdot)\|_{L^2(\R^3)}& =\|t\p_t\hat{v}+\widehat{x_1\p_1v}+\widehat{x_2\p_2v}+\widehat{x_3\p_3v}\|_{L^2(\R^3)}\\
& =\|t\p_t\hat{v}+\p_{\xi_1}(\xi_1\hat{v})+\p_{\xi_2}(\xi_2\hat{v})+\p_{\xi_3}(\xi_3\hat{v})\|_{L^2(\R^3)}\\
&\leq 3\|\hat{v}\|_{L^2(\R^3)}+\|t\p_t\hat{v}\|_{L^2(\R^3)}+\|\xi_1\p_{\xi_1}\hat{v}+\xi_2\p_{\xi_2}\hat{v}+\xi_3\p_{\xi_3}\hat{v}\|_{L^2(\R^3)},
\end{aligned}
\end{equation}
then in order to show \eqref{e:q25}, it suffices to estimate $\|t\p_t\hat{v}\|_{L^2(\R^3)}$ and $\|\xi_1\p_{\xi_1}\hat{v}+\xi_2\p_{\xi_2}\hat{v}+\xi_3\p_{\xi_3}\hat{v}\|_{L^2(\R^3)}$ in \eqref{e:q26} since the estimate of $\|\hat{v}\|_{L^2(\R^3)}$ has already been obtained in \eqref{equ:q18}.

For $\|t\p_t\hat{v}\|_{L^2(\R^3)}$, as in \eqref{e:q2}, combining \eqref{e:q13} and \eqref{e:q17} yields that for $\mu\geq3$,
\begin{equation}\label{e:q27}
\begin{aligned}
\left\|t\p_t\hat{v}(t, \cdot)\right\|_{L^2(D_1)}  &\lesssim\left\|(1+t|\xi|)^{-\f32}t|\xi|(\hat{v}_{0}+\hat{v}_{1})\right\|_{L^{2}}\\
& \lesssim t^{\f32-\f{3}{1+\varepsilon_1}}\sum_{j=1}^{3}\|t\p_j(v_0+v_1)\|_{(1+\varepsilon_1, 2)}\\
& = t^{\f32-\f{3}{1+\varepsilon_1}}\sum_{j=1}^{3}\|t\p_j(v_0+v_1)+x_j\p_t(v_0+v_1)\|_{(1+\varepsilon_1, 2)}\\
& = t^{\f32-\f{3}{1+\varepsilon_1}}\sum_{j=1}^{3}\|L_j(v_0+v_1)\|_{(1+\varepsilon_1, 2)}\\
& \lesssim t^{\f32-\f{3}{1+\varepsilon_1}}\|v_0+v_1\|_{Z,1, (1+\varepsilon_1, 2)}
\end{aligned}
\end{equation}
and
\begin{equation}\label{e:q28}
\begin{aligned}
\left\|t\p_t\hat{v}(t, \cdot)\right\|_{L^2(D_2)}
& \lesssim t^{-\frac{\mu}{2}}\left\|t|\xi|\hat{v}_{0}\right\|_{L^{2}}+t^{1-\f{\mu}{2}} \||\xi|^{-1}|\xi|\hat{v}_{1}\|_{L^{2}(D_2)}\\
& \lesssim t^{-\frac{\mu}{2}}\sum_{j=1}^{3}\|L_jv_0\|_{L^2}+t^{-\frac{\mu}{2}+1} \|(1+t|\xi|)^{-1}t|\xi|\hat{v}_{1}\|_{L^{2}}\\
& \lesssim t^{-\frac{\mu}{2}}\left\|v_{0}\right\|_{Z, 1, 2}+ t^{-\frac{\mu}{2}}\sum_{j=1}^{3}\|t\p_jv_1\|_{(\f65, 2)} \\
&=t^{-\frac{\mu}{2}}\left\|v_{0}\right\|_{Z, 1, 2}+ t^{-\frac{\mu}{2}}\sum_{j=1}^{3}\|L_jv_1\|_{(\f65, 2)}\\
& \lesssim  t^{-\frac{\mu}{2}}\left\|v_{0}\right\|_{Z, 1, 2}+ t^{-\frac{\mu}{2}}\left\|v_1\right\|_{Z, 1,(\f65, 2)}.
\end{aligned}
\end{equation}
By \eqref{e:q22}, we arrive at
\begin{equation}\label{e:q29}
\left\|t\p_t\hat{v}(t, \cdot)\right\|_{L^2(D_3)}=t\left\|\p_t\hat{v}(t, \cdot)\right\|_{L^2(D_3)} \lesssim t^{\f32-\f{3}{1+\varepsilon_1}}\|v_0+v_1\|_{(1+\varepsilon_1, 2)}.
\end{equation}
This, together with \eqref{e:q27} and \eqref{e:q28}, gives
\begin{equation}\label{e:q30}
\|t\p_tv(t, \cdot)\|_{L^2(\R^3)} \lesssim t^{-\f{\mu}{2}}\left\|v_0\right\|_{Z, 1, 2}+t^{-\f{\mu}{2}}\left\|v_1\right\|_{Z, 1,(\f65, 2)}+t^{\f32-\f{3}{1+\varepsilon_1}}\left\|v_0+v_1\right\|_{Z, 1,(1+\varepsilon_1, 2)}.
\end{equation}

We next estimate $\|\xi_1\p_{\xi_1}\hat{v}+\xi_2\p_{\xi_2}\hat{v}+\xi_3\p_{\xi_3}\hat{v}\|_{L^2(\R^3)}$
to show
\begin{equation}\label{e:q31}
\begin{aligned}
&\|\xi_1\p_{\xi_1}\hat{v}+\xi_2\p_{\xi_2}\hat{v}+\xi_3\p_{\xi_3}\hat{v}\|_{L^2(\R^3)} \\&\lesssim t^{-\f{\mu}{2}}\left\|v_0\right\|_{Z, 1, 2}+t^{-\f{\mu}{2}}\left\|v_1\right\|_{Z, 1,(\f65, 2)}+t^{\f32-\f{3}{1+\varepsilon_1}}\left\|v_0+v_1\right\|_{Z, 1,(1+\varepsilon_1, 2)}.
\end{aligned}
\end{equation}
To this end, we
divide the space $\R_{\xi}^3$ into three zones as in \eqref{equ:q19}.

\vskip 0.2 true cm

{\bf Part 3.1.  The analysis  in zones $D_1$ and  $D_3$}
\vskip 0.2 true cm
By \eqref{equ:q9-3} and Lemma \ref{lem1-1}, it holds   that for $k=1,2,3$,
\begin{equation}\label{e:q32}
\begin{aligned}
|\partial_{\xi_k} \Psi_0(t, 1, \xi)|
& \lesssim  \big| \partial_{\xi_k} \left( |\xi| t^\rho \left( J_{\rho-1}(|\xi|) Y_\rho(t|\xi|)
- J_\rho(t|\xi|) Y_{\rho-1}(|\xi|) \right) \right) \big| \\
& \lesssim  \frac{|\xi_k|}{|\xi|} t^\rho \left| J_{1-\rho}(|\xi|) Y_{-\rho}(t|\xi|) \right|
+ \frac{|\xi_k|}{|\xi|} t^\rho \left| J_{-\rho}(t|\xi|) Y_{1-\rho}(|\xi|) \right| \\
& \quad + t^\rho |\xi_k| \left| J_{2-\rho}(|\xi|) Y_{-\rho}(t|\xi|) \right|
+ t^\rho \frac{|\xi_k|}{|\xi|} \left| J_{1-\rho}(|\xi|) Y_{-\rho}(t|\xi|) \right| \\
& \quad + t^{\rho+1} |\xi_k| \left| Y_{1-\rho}(t|\xi|) J_{1-\rho}(|\xi|) \right|
+ t^\rho \frac{|\xi_k|}{|\xi|} \left| Y_{-\rho}(t|\xi|) J_{1-\rho}(|\xi|) \right| \\
& \quad + t^{\rho+1} |\xi_k| \left| J_{1-\rho}(t|\xi|) Y_{1-\rho}(|\xi|) \right|
+ t^\rho \frac{|\xi_k|}{|\xi|} \left| J_{-\rho}(t|\xi|) Y_{1-\rho}(|\xi|) \right| \\
& \quad + t^\rho |\xi_k| \left| Y_{2-\rho}(|\xi|) J_{-\rho}(t|\xi|) \right|
+ t^\rho \frac{|\xi_k|}{|\xi|} \left| Y_{1-\rho}(|\xi|) J_{-\rho}(t|\xi|) \right|
\end{aligned}
\end{equation}
and
\begin{equation}\label{e:q33}
\begin{aligned}
|\partial_{\xi_k} \Psi_1(t, 1, \xi)|
& \lesssim  \big| \partial_{\xi_k} \left( t^\rho \left( J_{\rho}(|\xi|) Y_\rho(t|\xi|)
- J_\rho(t|\xi|) Y_{\rho}(|\xi|) \right) \right) \big| \\
& \lesssim  \frac{|\xi_k|}{|\xi|} t^\rho \big| J_{1-\rho}(|\xi|) Y_{-\rho}(t|\xi|) \big|
+ \frac{|\xi_k|}{|\xi|^2} t^\rho \big| J_{-\rho}(|\xi|) Y_{-\rho}(t|\xi|) \big| \\
& \quad + t^{\rho+1}\f{|\xi_k|}{|\xi|} \left| Y_{1-\rho}(t|\xi|)J_{-\rho}(|\xi|) \right|
+ t^\rho \frac{|\xi_k|}{|\xi|^2} \left| Y_{-\rho}(t|\xi|)J_{-\rho}(|\xi|) \right|\\
& \quad + t^{\rho+1}\f{|\xi_k|}{|\xi|} \left| J_{1-\rho}(t|\xi|) Y_{-\rho}(|\xi|) \right|
+ t^\rho \frac{|\xi_k|}{|\xi|^2} \left| J_{-\rho}(t|\xi|) Y_{-\rho}(|\xi|) \right| \\
& \quad + t^\rho \f{|\xi_k|}{|\xi|} \left| Y_{1-\rho}(|\xi|) J_{-\rho}(t|\xi|) \right|
+ t^\rho \frac{|\xi_k|}{|\xi|^2} \left| Y_{-\rho}(|\xi|) J_{-\rho}(t|\xi|) \right|.
\end{aligned}
\end{equation}
Analogously to the treatment of \eqref{e:q2}, by \eqref{e:q32}-\eqref{e:q33}, \eqref{formula},  Lemma \ref{lem1-1} (i) and (iii),
one has that
\begin{equation}\label{e:q34}
\begin{aligned}
&\big\|\xi_1\p_{\xi_1}\hat{v}+\xi_2\p_{\xi_2}\hat{v}+\xi_3\p_{\xi_3}\hat{v}\big\|_{L^2(D_1)}\\
&=\big\|  \sum_{k=1}^{3}\left( \xi_k (\partial_{\xi_k} \Psi_0) \hat{v}_0
+ \xi_k (\partial_{\xi_k} \Psi_1) \hat{v}_1
+ \xi_k \Psi_0 \partial_{\xi_k} \hat{v}_0
+ \xi_k \Psi_1 \partial_{\xi_k} \hat{v}_1\right)\big\|_{L^2(D_1)}\\
&\lesssim \big\|(t|\xi|)^{-\frac{\mu}{2}}(\sum_{k=1}^3t|\xi_k|)(\hat{v}_{0}+\hat{v}_{1})\big\|_{L^{2}}+ \big\|(t|\xi|)^{-\f{\mu}{2}}(\sum_{k=1}^3\xi_k\p_{\xi_k})(\hat{v}_{0}+\hat{v}_{1})\big\|_{L^{2}}\\
& \lesssim \sum_{k=1}^3 \big\|(1+t|\xi|)^{-\f{3}{2}}t|\xi_k|(\hat{v}_{0}+\hat{v}_{1})\big\|_{L^{2}}+
\big\|(1+t|\xi|)^{-\f{3}{2}}\sum_{k=1}^3\widehat{(x_k\p_k)(v_0+v_1)}\big\|_{L^{2}}\\
&\quad+ \big\|(1+t|\xi|)^{-\f{3}{2}}\widehat{3(v_0+v_1)}\big\|_{L^{2}}\\
&\lesssim  t^{\f32-\f{3}{1+\varepsilon_1}}\left\|v_0+v_1\right\|_{Z, 1,(1+\varepsilon_1, 2)}
\end{aligned}
\end{equation}
and
\begin{equation}\label{e:q35}
\begin{aligned}
&\big\|\xi_1\p_{\xi_1}\hat{v}+\xi_2\p_{\xi_2}\hat{v}+\xi_3\p_{\xi_3}\hat{v}\big\|_{L^2(D_3)}\\
&\lesssim \big\|(1+t|\xi|)^{\f32}(1+t|\xi|)^{-\f32}(\hat{v}_{0}+\hat{v}_{1})\big\|_{L^{2}}+ \big\|(1+t|\xi|)^{\f32}(1+t|\xi|)^{-\f32}(\sum_{k=1}^3\xi_k\p_{\xi_k})(\hat{v}_{0}+\hat{v}_{1})\big\|_{L^{2}}\\
&\lesssim \big\|(1+t|\xi|)^{-\f32}(\hat{v}_{0}+\hat{v}_{1})\big\|_{L^{2}}+ \big\|(1+t|\xi|)^{-\f32}(\sum_{k=1}^3\xi_k\p_{\xi_k})(\hat{v}_{0}+\hat{v}_{1})\big\|_{L^{2}}\\
& \lesssim \left\|(1+t|\xi|)^{-\f{3}{2}}(\hat{v}_{0}+\hat{v}_{1})\right\|_{L^{2}}+
\big\|(1+t|\xi|)^{-\f{3}{2}}\sum_{k=1}^3\widehat{(x_k\p_k)(v_0+v_1)}\big\|_{L^{2}}\\
&\lesssim  t^{\f32-\f{3}{1+\varepsilon_1}}\left\|v_0+v_1\right\|_{Z, 1,(1+\varepsilon_1, 2)}.
\end{aligned}
\end{equation}
\vskip 0.2 true cm

{\bf Part 3.2.  The analysis  in zone $D_2$}

\vskip 0.2 true cm

It follows from \eqref{equ:q6} and Lemma \ref{lem1-1}  that for $k=1,2,3$,
\begin{equation}\label{e:q36}
\begin{aligned}
\left| \partial_{\xi_k} \Psi_0(t, 1, \xi) \right|
&\lesssim \left| \partial_{\xi_k} \left( |\xi| t^\rho \left( H_\rho^+(t|\xi|) H_{\rho-1}^-(|\xi|)
- H_{\rho-1}^+(|\xi|) H_\rho^-(t|\xi|) \right) \right) \right| \\
&\lesssim \frac{|\xi_k|}{|\xi|} t^\rho \left| H_\rho^+(t|\xi|) H_{\rho-1}^-(|\xi|) \right|
+ \frac{|\xi_k|}{|\xi|} t^\rho \left| H_{\rho-1}^+(|\xi|) H_\rho^-(t|\xi|) \right| \\
&\quad + t^{\rho+1} |\xi_k| \left| H_{\rho-1}^+(t|\xi|) H_{\rho-1}^-(|\xi|) \right|
+ t^\rho \frac{|\xi_k|}{|\xi|} \left| H_\rho^+(t|\xi|) H_{\rho-1}^-(|\xi|) \right| \\
&\quad + t^\rho |\xi_k| \left| H_{\rho-2}^-(|\xi|) H_\rho^+(t|\xi|) \right|
+ t^\rho \frac{|\xi_k|}{|\xi|} \left| H_{\rho-1}^-(|\xi|) H_\rho^+(t|\xi|) \right| \\
&\quad + t^\rho |\xi_k| \left| H_{\rho-2}^+(|\xi|) H_\rho^-(t|\xi|) \right|
+ t^\rho \frac{|\xi_k|}{|\xi|} \left| H_{\rho-1}^+(|\xi|) H_\rho^-(t|\xi|) \right| \\
&\quad + t^{\rho+1} |\xi_k| \left| H_{\rho-1}^-(t|\xi|) H_{\rho-1}^+(|\xi|) \right|
+ t^\rho \frac{|\xi_k|}{|\xi|} \left| H_\rho^-(t|\xi|) H_{\rho-1}^+(|\xi|) \right|
\end{aligned}
\end{equation}
and
\begin{equation}\label{e:q37}
\begin{aligned}
\left| \partial_{\xi_k} \Psi_1(t, 1, \xi) \right|
&\lesssim  t^{\rho+1} \f{|\xi_k|}{|\xi|} \left| H_{\rho-1}^+(t|\xi|) H_{\rho}^-(|\xi|) \right|
+ t^\rho \frac{|\xi_k|}{|\xi|^2} \left| H_\rho^+(t|\xi|) H_{\rho}^-(|\xi|) \right| \\
&\quad + t^\rho \f{|\xi_k|}{|\xi|} \left| H_{\rho-1}^-(|\xi|) H_\rho^+(t|\xi|) \right|
+ t^\rho \frac{|\xi_k|}{|\xi|^2} \left| H_{\rho}^-(|\xi|) H_\rho^+(t|\xi|) \right| \\
&\quad + t^\rho \f{|\xi_k|}{|\xi|} \left| H_{\rho-1}^+(|\xi|) H_\rho^-(t|\xi|) \right|
+ t^\rho \frac{|\xi_k|}{|\xi|^2} \left| H_{\rho}^+(|\xi|) H_\rho^-(t|\xi|) \right| \\
&\quad + t^{\rho+1} \f{|\xi_k|}{|\xi|} \left| H_{\rho-1}^-(t|\xi|) H_{\rho}^+(|\xi|) \right|
+ t^\rho \frac{|\xi_k|}{|\xi|^2} \left| H_\rho^-(t|\xi|) H_{\rho}^+(|\xi|) \right|.
\end{aligned}
\end{equation}
As in \eqref{e:q7}, using \eqref{e:q36}-\eqref{e:q37} and  Lemma \ref{lem1-1}(ii), we can obtain that for $\mu\geq3$,
\begin{equation}\label{e:q38}
\begin{aligned}
&\big\|\xi_1\p_{\xi_1}\hat{v}+\xi_2\p_{\xi_2}\hat{v}+\xi_3\p_{\xi_3}\hat{v}\big\|_{L^2(D_2)}
\\
&\lesssim \big\|t^{-\frac{\mu}{2}}|\xi|\hat{v}_{0}\big\|_{L^{2}}+ \big\|t^{-\frac{\mu}{2}}(1+|\xi|^2)^{-\f12}|\xi|\hat{v}_{1}\big\|_{L^{2}}+ \big\|t^{-\f{\mu}{2}}(1+|\xi|^2)^{-\f12}(\sum_{k=1}^3\xi_k\p_{\xi_k})\hat{v}_{1}\big\|_{L^{2}}\\
&\quad +\big\|t^{-\f{\mu}{2}}(\sum_{k=1}^3\xi_k\p_{\xi_k})\hat{v}_{0}\big\|_{L^{2}}\\
& \lesssim t^{-\frac{\mu}{2}}\big\|\hat{v}_{0}\big\|_{Z, 1, 2}+ t^{-\frac{\mu}{2}}\big\|\hat{v}_{1}\big\|_{Z, 1, (\f{6}{5}, 2)}+ t^{-\f{\mu}{2}}
\big\|(1+|\xi|^2)^{-\f{1}{2}}\sum_{k=1}^3\widehat{(x_k\p_k)v_1}\big\|_{L^{2}}\\
&\quad+ t^{-\f{\mu}{2}} \big\|(1+|\xi|^2)^{-\f{1}{2}}3\widehat{v_1}\big\|_{L^{2}}+t^{-\f{\mu}{2}} \big\|\sum_{k=1}^3\widehat{(x_k\p_k)v_0}\big\|_{L^{2}}+ t^{-\f{\mu}{2}} \left\|3\widehat{v_0}\right\|_{L^{2}}\\
&\lesssim  t^{-\f{\mu}{2}}\left\|v_0\right\|_{Z, 1, 2}+t^{-\f{\mu}{2}}\left\|v_1\right\|_{Z, 1,(\f65, 2)}.
\end{aligned}
\end{equation}

Therefore, collecting \eqref{e:q34}, \eqref{e:q35} and \eqref{e:q38} yields \eqref{e:q31}. Consequently, \eqref{e:q25} holds.

\vskip 0.2 true cm

{\bf Part 4.  The treatment  of $\|L_jv(t, \cdot)\|_{L^2(\R^3)}(j=1, 2, 3)$ }

\vskip 0.2 true cm

We now assert
\begin{equation}\label{e:q39}
\|L_jv(t, \cdot)\|_{L^2(\R^3)} \lesssim t^{-\f{\mu}{2}}\left\|v_0\right\|_{Z, 1, 2}+t^{-\f{\mu}{2}}\left\|v_1\right\|_{Z, 1,(\f65, 2)}+t^{\f32-\f{3}{1+\varepsilon_1}}\left\|v_0+v_1\right\|_{Z, 1,(1+\varepsilon_1, 2)}.
\end{equation}

To prove \eqref{e:q39}, it  suffices only to estimate $\|L_3v(t, \cdot)\|_{L^2(\R^3)}$ since the terms $\|L_jv(t, \cdot)\|_{L^2(\R^3)}$ for $(j=1,2)$ can be treated in the same way. The support condition of the initial data, together with $t\geq1$, implies that $|x|\lesssim t$
for $(t, x)\in \operatorname{supp}v$. Thus,
\begin{equation}\label{e:q40}
\begin{aligned}
\|L_3v(t, \cdot)\|_{L^2(\R^3)}\leq \|t\p_3v\|_{L^2(\R^3)}+\|x_3\p_tv\|_{L^2(\R^3)}
\lesssim \|t|\xi_3|\hat{v}\|_{L^2(\R^3)}+\|t\p_t\hat{v}\|_{L^2(\R^3)}.
\end{aligned}
\end{equation}
Note that $\|t\p_t\hat{v}\|_{L^2(\R^3)}$ can be treated as in \eqref{e:q30},
it remains to estimate $\|t|\xi_3|\hat{v}\|_{L^2(\R^3)}$.

It follows from  \eqref{e:q4}, \eqref{e:q1}, \eqref{e:q5} and  \eqref{e:q10} that for $\mu\geq3$,
\begin{equation}\label{e:q41}
\begin{aligned}
\left\|t|\xi_3|\hat{v}(t, \cdot)\right\|_{L^2(D_1)} & \lesssim\left\|(1+t|\xi|)^{-\f32}t|\xi|(\hat{v}_{0}+\hat{v}_{1})\right\|_{L^{2}(D_1)} \\
& \lesssim t^{\f32-\f{3}{1+\varepsilon_1}}\left\|v_0+v_1\right\|_{Z, 1,(1+\varepsilon_1, 2)},
\end{aligned}
\end{equation}
\begin{equation}\label{e:q42}
\begin{aligned}
\left\|t|\xi_3|\hat{v}(t, \cdot)\right\|_{L^2(D_2)}
& \lesssim t^{-\frac{\mu}{2}}\left\|t|\xi|v_{0}\right\|_{L^{2}}+t^{-\f{\mu}{2}} \|(1+|\xi|^2)^{-\f{1}{2}}t|\xi|\hat{v}_{1}\|_{L^{2}}\\
&\lesssim t^{-\f{\mu}{2}}\left\|v_0\right\|_{Z, 1, 2}+t^{-\f{\mu}{2}}\left\|v_1\right\|_{Z, 1,(\f65, 2)}
\end{aligned}
\end{equation}
and
\begin{equation}\label{e:q43}
\begin{aligned}
\left\|t|\xi_3|\hat{v}(t, \cdot)\right\|_{L^2(D_3)} & \leq\big\|(1+t|\xi|)^{\f32}\times(1+t|\xi|)^{-\f32}t|\xi|(\hat{v}_{0}+\hat{v}_{1})\big\|_{L^{2}(D_3)} \\
&\lesssim\big\|(1+t|\xi|)^{-\f32}t|\xi|(\hat{v}_{0}+\hat{v}_{1})\big\|_{L^{2}}\\
&\lesssim t^{\f32-\f{3}{1+\varepsilon_1}}\|v_0+v_1\|_{Z, 1, (1+\varepsilon_1, 2)}.
\end{aligned}
\end{equation}
Thus, combining \eqref{e:q41}, \eqref{e:q42} and \eqref{e:q43} yields
\begin{equation}\label{e:q44}
\|t|\xi_3|\hat{v}(t, \cdot)\|_{L^2(\R^3)} \lesssim t^{-\f{\mu}{2}}\left\|v_0\right\|_{Z, 1, 2}
+t^{-\f{\mu}{2}}\left\|v_1\right\|_{Z, 1,(\f65, 2)}+t^{\f32-\f{3}{1+\varepsilon_1}}\left\|v_0+v_1\right\|_{Z, 1,(1+\varepsilon_1, 2)}.
\end{equation}
Then \eqref{e:q39} follows immediately from \eqref{e:q30} and \eqref{e:q44}.
\vskip 0.2 true cm

{\bf Part 5.  The treatment of $\|\Omega_{kj}v(t, \cdot)\|_{L^2(\R^3)}$ $(1\le k<j\le3)$ }

\vskip 0.2 true cm
We now prove
\begin{equation}\label{e:q45}
\|\Omega_{kj}v(t, \cdot)\|_{L^2(\R^3)} \lesssim t^{-\f{\mu}{2}}\left\|v_0\right\|_{Z, 1, 2}+t^{-\f{\mu}{2}}\left\|v_1\right\|_{Z, 1,(\f65, 2)}+t^{\f32-\f{3}{1+\varepsilon_1}}\left\|v_0+v_1\right\|_{Z, 1,(1+\varepsilon_1, 2)}.
\end{equation}

Following the same argument as in \eqref{e:q40}, one can arrive at
\begin{equation}\label{e:q46''}
\begin{aligned}
\|\Omega_{kj}v(t, \cdot)\|_{L^2(\R^3)}&\leq \|\f{x_k}{t}t\p_jv\|_{L^2(\R^3)}+\|\f{x_j}{t}t\p_kv\|_{L^2(\R^2)}\\
& \lesssim \|t|\xi_j|\hat{v}\|_{L^2(\R^3)}+\|t|\xi_k|\hat{v}\|_{L^2(\R^3)}.\\
\end{aligned}
\end{equation}
For each $j=1,2,3$, \eqref{e:q44} implies
$$
\|t|\xi_j|\hat{v}(t, \cdot)\|_{L^2(\R^3)} \lesssim t^{-\f{\mu}{2}}\left\|v_0\right\|_{Z, 1, 2}+t^{-\f{\mu}{2}}\left\|v_1\right\|_{Z, 1,(\f65, 2)}+t^{\f32-\f{3}{1+\varepsilon_1}}\left\|v_0+v_1\right\|_{Z, 1,(1+\varepsilon_1, 2)}.
$$
This proves \eqref{e:q45}.

Therefore, the proof of Lemma \ref{lem2} is completed by collecting the above results in Parts 1-5.
\end{proof}

It remains to consider the case where $\mu$ is not an odd integer. Next, we  establish the following lemma.

\begin{lemma}\label{lem2-1}
Let $v$ be the solution to \eqref{equ:q16} and $\varepsilon_1\in(0, 1)$. Then
for $\f{14}{5}\leq\mu<3$,
\begin{equation}\label{e:q46'}
\|v(t, \cdot)\|_{Z, 1, 2} \lesssim t^{-\f{\mu}{2}}\left\|v_0\right\|_{Z, 1, 2}
+t^{-\f{\mu}{2}}\left\|v_1\right\|_{Z, 1,(\f{6}{5}, 2)}
+t^{-\f{\mu}{2}}\left\|v_0+v_1\right\|_{Z, 1,(\f{6}{3+\mu}, 2)},
\end{equation}
and for $\mu\geq3$,
\begin{equation}\label{e:q46}
\|v(t, \cdot)\|_{Z, 1, 2} \lesssim t^{-\f{\mu}{2}}\left\|v_0\right\|_{Z, 1, 2}
+t^{-\f{\mu}{2}}\left\|v_1\right\|_{Z, 1,(\f{6}{5}, 2)}
+t^{\f{3}{2}-\f{3}{1+\varepsilon_1}}\left\|v_0+v_1\right\|_{Z, 1,(1+\varepsilon_1, 2)}.
\end{equation}
\end{lemma}
\begin{proof}
Note that for $|\xi|\geq 1$, the related proof for \eqref{e:q46'}-\eqref{e:q46} is identical to that in Lemma \ref{lem2}
where $\mu\ge3$ is an odd integer. Then we omit the details here. Hence, it suffices to consider the remaining frequency zones
\begin{equation}\label{e:q47}
A_1=\{\xi: |\xi| \leq 1 \leq t|\xi|\}, \quad A_2=\{\xi: |\xi|\leq t|\xi| \leq 1\}.
\end{equation}
By the definition of $\|v\|_{Z, 1, 2}$,
we divide the proof into the following five parts.
\vskip 0.2 true cm

{\bf Part 1.  The treatment  of $\|v(t, \cdot)\|_{L^2(\R^3)}$ }

\vskip 0.2 true cm
We shall establish the following estimates
\begin{equation}\label{e:q48'}
\|v(t, \cdot)\|_{L^2(\R^3)} \lesssim t^{-\f{\mu}{2}}\left\|v_0\right\|_{L^2(\R^3)}
+t^{-\f{\mu}{2}}\left\|v_1\right\|_{(\f{6}{5}, 2)}
+t^{-\f{\mu}{2}}\left\|v_0+v_1\right\|_{(\f{6}{3+\mu}, 2)}\quad \text{for}\, \f{14}{5}\leq\mu<3
\end{equation}
and
\begin{equation}\label{e:q48}
\|v(t, \cdot)\|_{L^2(\R^3)} \lesssim t^{-\f{\mu}{2}}\left\|v_0\right\|_{L^2(\R^3)}
+t^{-\f{\mu}{2}}\left\|v_1\right\|_{(\f{6}{5}, 2)}
+t^{\f{3}{2}-\f{3}{1+\varepsilon_1}}\left\|v_0+v_1\right\|_{(1+\varepsilon_1, 2)} \quad \text{for}\, \mu\geq3.
\end{equation}

When $\xi\in A_1$ and $\rho=-\f{\mu-1}{2}\in(-\infty,-\f{9}{10}]$ is not an integer, it follows from \eqref{equ:q9-5'}
and Lemma \ref{lem1} that for $j=0,1$,
\begin{equation}\label{e:q49}
\begin{aligned}
|\Psi_j(t,\tau,\xi)|
&\lesssim\left| |\xi|^{1-j} t^{\rho}
\begin{vmatrix}
J_{1-\rho-j}(|\xi|) & J_{-\rho}(t|\xi|) \\
(-1)^{1-j}J_{\rho-1+j}(|\xi|) & J_{\rho}(t|\xi|)
\end{vmatrix} \right|\\
&\lesssim |\xi|^{1-j} t^{\rho}|\xi|^{\rho-1+j}(t|\xi|)^{-\f12}\\
&= t^{-\f{\mu}{2}}|\xi|^{-\f{\mu}{2}}.
\end{aligned}
\end{equation}
For $\f{14}{5}\leq\mu<3$, it follows from
$x=ty$ and $\eta=t\xi$ that
\begin{equation}\label{e:q2'}
\begin{aligned}
\left\|\hat{v}(t, \cdot)\right\|_{L^2(A_1)} & \lesssim\left\|(1+t|\xi|)^{-\f{\mu}{2}}(\hat{v}_{0}+\hat{v}_{1})\right\|_{L^{2}(A_1)} \\
& \lesssim t^{\f32}\left(\int_{R^3} (\frac{1}{1+|\eta|})^\mu\big(\int_{R^3} e^{-i y \cdot \eta} (v_0+v_1)(ty) \mathrm{d} y\big)^2 \mathrm{d} \eta\right)^{\frac{1}{2}}\\
& \lesssim t^{\f32}\|(v_0+v_1)(ty)\|_{H^{-\f{\mu}{2}}(\R^3)}.
\end{aligned}
\end{equation}
Due to $\|V_1\|_{H^{-\f{\mu}{2}}}=\sup _{\substack{v \in H^{\f{\mu}{2}} \\ v \neq 0}} \frac{\left|\int_{\mathbb{R}^3} V_1(y) v(y) \mathrm{d} y\right|}{\|v\|_{H^{\f{\mu}{2}}}}$ with $V_1(y)=(v_0+v_1)(ty)$, by
 H\"{o}lder's inequality and Sobolev imbedding theorem, one then obtain that for $\f{14}{5}\leq\mu<3$,
\begin{equation}\label{e:q3'}
\|V_1 v\|_{L^1(\R^3)} \lesssim\|V_1\|_{(\f{6}{3+\mu}, 2)}\|v\|_{\f{6}{3-\mu}} \lesssim\|V_1\|_{(\f{6}{3+\mu}, 2)}\|v\|_{H^{\f{\mu}{2}}}\lesssim t^{-\f{3+\mu}{2}}\|v_0+v_1\|_{(\f{6}{3+\mu}, 2)}\|v\|_{H^{\f{\mu}{2}}}.
\end{equation}
This, together with  \eqref{e:q2'}, implies that for $\f{14}{5}\leq\mu<3$,
\begin{equation}\label{e:q4'}
\left\|\hat{v}(t, \cdot)\right\|_{L^2(A_1)}\lesssim t^{-\f{\mu}{2}}\|v_0+v_1\|_{(\f{6}{3+\mu}, 2)}.
\end{equation}
Note that for $\mu\geq 3$, the Sobolev embedding theorem in \eqref{e:q3'} is not applicable, but this difficulty can be overcome by the estimate $(1+t|\xi|)^{-\frac{\mu}{2}} \le (1+t|\xi|)^{-\frac{3}{2}}$. Then it follows from \eqref{e:q49} and \eqref{e:q2} that for $\mu\geq 3$,
\begin{equation}\label{e:q50}
\begin{aligned}
\left\|\hat{v}(t, \cdot)\right\|_{L^2(A_1)} \lesssim\big\|(1+t|\xi|)^{-\f32}(\hat{v}_{0}+\hat{v}_{1})\big\|_{L^{2}} \lesssim t^{\f32-\f{3}{1+\varepsilon_1}}\|v_0+v_1\|_{(1+\varepsilon_1, 2)}.
\end{aligned}
\end{equation}
When $\xi\in A_2$ and $\rho=-\f{\mu-1}{2}$ is not an integer, it follows from \eqref{equ:q9-5'} and \eqref{equ:q15} that for $j=0,1$,
\begin{equation}\label{e:q51}
\begin{aligned}
|\Psi_j(t,\tau,\xi)|
\lesssim |\xi|^{1-j} t^{\rho}\left[|\xi|^{1-\rho-j}(t|\xi|)^{\rho}+|\xi|^{\rho-1+j}(t|\xi|)^{-\rho}\right]
\lesssim 1.
\end{aligned}
\end{equation}
Then for $|\xi|\leq t|\xi| \leq 1$ and $\f{14}{5}\leq\mu<3$, as in \eqref{e:q2'}, one has
\begin{equation}\label{e:q52'}
\begin{aligned}
\left\|\hat{v}(t, \cdot)\right\|_{L^2(A_2)} & \leq\big\|(1+t|\xi|)^{\f{\mu}{2}}\times(1+t|\xi|)^{-\f{\mu}{2}}(\hat{v}_{0}+\hat{v}_{1})\big\|_{L^{2}(A_2)} \\
&\lesssim\big\|(1+t|\xi|)^{-\f{\mu}{2}}(\hat{v}_{0}+\hat{v}_{1})\big\|_{L^{2}}\\
&\lesssim t^{-\f{\mu}{2}}\|v_0+v_1\|_{(\f{6}{3+\mu}, 2)}.
\end{aligned}
\end{equation}
While for $|\xi|\leq t|\xi| \leq 1$ and $\mu\geq3$, as in \eqref{e:q2}, we arrive at
\begin{equation}\label{e:q52}
\begin{aligned}
\big\|\hat{v}(t, \cdot)\big\|_{L^2(A_2)} & \leq\big\|(1+t|\xi|)^{\f32}\times(1+t|\xi|)^{-\f32}(\hat{v}_{0}+\hat{v}_{1})\big\|_{L^{2}(A_2)} \\
&\lesssim\big\|(1+t|\xi|)^{-\f32}(\hat{v}_{0}+\hat{v}_{1})\big\|_{L^{2}}\\
&\lesssim t^{\f32-\f{3}{1+\varepsilon_1}}\|v_0+v_1\|_{(1+\varepsilon_1, 2)}.
\end{aligned}
\end{equation}

Collecting \eqref{e:q4'}-\eqref{e:q50}, \eqref{e:q52}-\eqref{e:q52'} and \eqref{e:q9}  gives \eqref{e:q48'}-\eqref{e:q48}.

\vskip 0.2 true cm

{\bf Part 2.  The treatment  of $\|\p v(t, \cdot)\|_{L^2(\R^3)}$ }

\vskip 0.2 true cm

We assert
\begin{equation}\label{e:q53'}
\|\p v(t, \cdot)\|_{L^2(\R^3)} \lesssim t^{-\f{\mu}{2}}\left\|v_0\right\|_{Z, 1, 2}+t^{-\f{\mu}{2}}\left\|v_1\right\|_{Z, 1,(\f65, 2)}+t^{-\f{\mu}{2}}\left\|v_0+v_1\right\|_{Z, 1,(\f{6}{3+\mu}, 2)}\quad \text{for}\, \f{14}{5}\leq\mu<3
\end{equation}
and
\begin{equation}\label{e:q53}
\|\p v(t, \cdot)\|_{L^2(\R^3)} \lesssim t^{-\f{\mu}{2}}\left\|v_0\right\|_{Z, 1, 2}+t^{-\f{\mu}{2}}\left\|v_1\right\|_{Z, 1,(\f65, 2)}+t^{\f32-\f{3}{1+\varepsilon_1}}\left\|v_0+v_1\right\|_{Z, 1,(1+\varepsilon_1, 2)}\quad \text{for}\, \mu\geq3.
\end{equation}
To prove \eqref{e:q53'}-\eqref{e:q53}, we only need to consider  the two zones as in \eqref{e:q47}.

For the zone $A_1$, it follows from \eqref{equ:q9-5'} and  Lemma \ref{lem1} (i) that for $j=0,1$,
\begin{equation}\label{e:q54}
\begin{aligned}
|\p_t\Psi_j(t,\tau,\xi)|
&\lesssim\left| |\xi|^{2-j} t^{\rho}
\begin{vmatrix}
J_{1-\rho-j}(|\xi|) & J_{1-\rho}(t|\xi|) \\
(-1)^{2-j}J_{\rho-1+j}(|\xi|) & J_{\rho-1}(t|\xi|)
\end{vmatrix} \right|\\
&\lesssim |\xi|^{2-j} t^{\rho}\left[|\xi|^{1-\rho-j}(t|\xi|)^{-\f12}+(t|\xi|)^{-\f12}|\xi|^{\rho-1+j}\right]\\
&\lesssim t^{-\f{\mu}{2}}|\xi|^{-\f{\mu}{2}}\cdot|\xi|\\
&\leq t^{-\f{\mu}{2}}|\xi|^{-\f{\mu}{2}}.
\end{aligned}
\end{equation}
By  $\f{14}{5}\leq\mu<3$, as in \eqref{e:q2'} and \eqref{e:q4'}, one can obtain that for $k=1, 2, 3$,
\begin{equation}\label{e:q55}
\left\||\xi_k|\hat{v}(t, \cdot)\right\|_{L^2(A_1)}\leq\left\|\hat{v}(t, \cdot)\right\|_{L^2(A_1)}
\lesssim  t^{-\f{\mu}{2}}\left\|v_0+v_1\right\|_{(\f{6}{3+\mu}, 2)}
\end{equation}
and
\begin{equation}\label{e:q56}
\begin{aligned}
\left\|\p_t\hat{v}(t, \cdot)\right\|_{L^2(A_1)} & \lesssim \big\|(1+t|\xi|)^{-\f{\mu}{2}}(\hat{v}_{0}+\hat{v}_{1})\big\|_{L^{2}} \lesssim t^{-\f{\mu}{2}}\left\|v_0+v_1\right\|_{(\f{6}{3+\mu}, 2)}.
\end{aligned}
\end{equation}
By  $\mu\geq3$, as in \eqref{e:q50} and \eqref{e:q2}, we have that for $k=1, 2, 3$,
\begin{equation}\label{e:q55}
\left\||\xi_k|\hat{v}(t, \cdot)\right\|_{L^2(A_1)}\leq\left\|\hat{v}(t, \cdot)\right\|_{L^2(A_1)}
\lesssim  t^{\f32-\f{3}{1+\varepsilon_1}}\|v_0+v_1\|_{(1+\varepsilon_1, 2)}
\end{equation}
and
\begin{equation}\label{e:q56}
\begin{aligned}
\left\|\p_t\hat{v}(t, \cdot)\right\|_{L^2(A_1)} & \lesssim \big\|(1+t|\xi|)^{-\f32}(\hat{v}_{0}+\hat{v}_{1})\big\|_{L^{2}} \lesssim t^{\f32-\f{3}{1+\varepsilon_1}}\|v_0+v_1\|_{(1+\varepsilon_1, 2)}.
\end{aligned}
\end{equation}

For the zone $A_2$,  one has from \eqref{equ:q9-4} and Lemma \ref{lem1} (i) that for $j=0, 1$,
\begin{equation}\label{e:q58'}
\begin{aligned}
|\p_t\Psi_j(t,\tau,\xi)|
\lesssim |\xi|^{2-j} t^{\rho}\left[|\xi|^{1-\rho-j}(t|\xi|)^{\rho-1}+(t|\xi|)^{1-\rho}|\xi|^{\rho-1+j}\right]
\leq t^{-1}.
\end{aligned}
\end{equation}
As in \eqref{e:q11} and \eqref{e:q52}, we obtain that for $\mu\geq3$ and $k=1, 2, 3$,
\begin{equation}\label{e:q58}
\begin{aligned}
\left\|\p_t\hat{v}\right\|_{L^2(A_2)} \lesssim t^{-1}\left\|(1+t|\xi|)^{-\f32}(\hat{v}_{0}+\hat{v}_{1})\right\|_{L^{2}} \lesssim t^{\f12-\f{3}{1+\varepsilon_1}}\|v_0+v_1\|_{(1+\varepsilon_1, 2)}
\end{aligned}
\end{equation}
and
\begin{equation}\label{e:q57}
\left\||\xi_k|\hat{v}(t, \cdot)\right\|_{L^2(A_2)}\leq t^{-1}\left\|\hat{v}(t, \cdot)\right\|_{L^2(A_2)} \lesssim t^{\f12-\f{3}{1+\varepsilon_1}}\|v_0+v_1\|_{(1+\varepsilon_1, 2)}.
\end{equation}

For $\f{14}{5}\leq\mu<3$ and $k=1, 2, 3$, as in \eqref{e:q2'}, one arrives at
\begin{equation}\label{e:q58'}
\begin{aligned}
\left\|\p_t\hat{v}\right\|_{L^2(A_2)} \lesssim t^{-1}\left\|(1+t|\xi|)^{-\f{\mu}{2}}(\hat{v}_{0}+\hat{v}_{1})\right\|_{L^{2}} \lesssim t^{-\f{\mu}{2}-1}\|v_0+v_1\|_{(\f{6}{3+\mu}, 2)}
\end{aligned}
\end{equation}
and
\begin{equation}\label{e:q57'}
\left\||\xi_k|\hat{v}(t, \cdot)\right\|_{L^2(A_2)}\leq t^{-1}\left\|\hat{v}(t, \cdot)\right\|_{L^2(A_2)} \lesssim t^{-\f{\mu}{2}-1}\|v_0+v_1\|_{(\f{6}{3+\mu}, 2)}.
\end{equation}
Therefore, combining  \eqref{e:q55}-\eqref{e:q57'} with \eqref{e:q20} yields \eqref{e:q53'}-\eqref{e:q53}.

\vskip 0.2 true cm

{\bf Part 3.  The treatment  of $\|L_0v(t, \cdot)\|_{L^2(\R^3)}$ }

\vskip 0.2 true cm

We start to show
\begin{equation}\label{e:q59'}
\|L_0v(t, \cdot)\|_{L^2(\R^3)} \lesssim t^{-\f{\mu}{2}}\left\|v_0\right\|_{Z, 1, 2}+t^{-\f{\mu}{2}}\left\|v_1\right\|_{Z, 1,(\f65, 2)}+t^{-\f{\mu}{2}}\left\|v_0+v_1\right\|_{Z, 1,(\f{6}{3+\mu}, 2)}\quad \text{for}\, \f{14}{5}\leq\mu<3
\end{equation}
and
\begin{equation}\label{e:q59}
\|L_0v(t, \cdot)\|_{L^2(\R^3)} \lesssim t^{-\f{\mu}{2}}\left\|v_0\right\|_{Z, 1, 2}+t^{-\f{\mu}{2}}\left\|v_1\right\|_{Z, 1,(\f65, 2)}+t^{\f32-\f{3}{1+\varepsilon_1}}\left\|v_0+v_1\right\|_{Z, 1,(1+\varepsilon_1, 2)}\quad \text{for}\, \mu\geq3.
\end{equation}

As in \eqref{e:q26}, in order to prove \eqref{e:q59'}-\eqref{e:q59}, it suffices to treat
\[
\|t\partial_t\hat v\|_{L^2(\R^3)}
\quad\text{and}\quad
\|\sum_{k=1}^3\xi_k\p_{\xi_k}\hat{v}\|_{L^2(\R^3)}.
\]
We next treat these two terms by the same decomposition as in \eqref{e:q47}.

For $\|t\p_t\hat{v}\|_{L^2(\R^3)}$, as in \eqref{e:q27}, it follows from \eqref{e:q54} and \eqref{e:q58} that for $\mu\geq3$,
\begin{equation}\label{e:q60}
\begin{aligned}
\left\|t\p_t\hat{v}(t, \cdot)\right\|_{L^2(A_1)}  &\lesssim\left\|(1+t|\xi|)^{-\f32}t|\xi|(\hat{v}_{0}+\hat{v}_{1})\right\|_{L^{2}}
\lesssim t^{\f32-\f{3}{1+\varepsilon_1}}\|v_0+v_1\|_{Z,1, (1+\varepsilon_1, 2)}
\end{aligned}
\end{equation}
and
$$
\left\|t\p_t\hat{v}(t, \cdot)\right\|_{L^2(A_2)}=t\left\|\p_t\hat{v}(t, \cdot)\right\|_{L^2(A_2)} \lesssim t^{\f32-\f{3}{1+\varepsilon_1}}\|v_0+v_1\|_{(1+\varepsilon_1, 2)}.
$$
This, together with \eqref{e:q60} and \eqref{e:q28}, derives that for $\mu\geq3$,
\begin{equation}\label{e:q61}
\|t\p_tv(t, \cdot)\|_{L^2(\R^3)} \lesssim t^{-\f{\mu}{2}}\left\|v_0\right\|_{Z, 1, 2}+t^{-\f{\mu}{2}}\left\|v_1\right\|_{Z, 1,(\f65, 2)}+t^{\f32-\f{3}{1+\varepsilon_1}}\left\|v_0+v_1\right\|_{Z, 1,(1+\varepsilon_1, 2)}.
\end{equation}
Similarly, as in \eqref{e:q56} and \eqref{e:q58'}, one has that for $\f{14}{5}\leq\mu<3$,
\begin{equation}\label{e:q61'}
\|t\p_tv(t, \cdot)\|_{L^2(A_1\bigcup A_2)} \lesssim\left\|(1+t|\xi|)^{-\f{\mu}{2}}t|\xi|(\hat{v}_{0}+\hat{v}_{1})\right\|_{L^{2}} \lesssim t^{-\f{\mu}{2}}\left\|v_0+v_1\right\|_{Z, 1,(\f{6}{3+\mu}, 2)}.
\end{equation}
Combining \eqref{e:q61'} and \eqref{e:q28} yields
\begin{equation}\label{e:q61''}
\|t\p_tv(t, \cdot)\|_{L^2(\R^3)} \lesssim t^{-\f{\mu}{2}}\left\|v_0\right\|_{Z, 1, 2}+t^{-\f{\mu}{2}}\left\|v_1\right\|_{Z, 1,(\f65, 2)}+t^{-\f{\mu}{2}}\left\|v_0+v_1\right\|_{Z, 1,(\f{6}{3+\mu}, 2)}.
\end{equation}

Next, we turn to deal with $\|\ds\sum_{k=1}^3\xi_k\p_{\xi_k}\hat{v}\|_{L^2(\R^3)}$.
It follows from \eqref{equ:q9-5'} and Lemma \ref{lem1-1}  that for $k=1,2,3$,
\begin{equation}\label{e:q62}
\begin{aligned}
& |\p_{\xi_k}\Psi_0(t, 1, \xi)|
\lesssim \frac{|\xi_k|}{|\xi|} t^\rho J_{-\rho+1}(|\xi|) J_\rho(t|\xi|)+\frac{|\xi_k|}{|\xi|} t^\rho J_{\rho-1}(|\xi|) J_{-\rho}(t|\xi|)
+|\xi_k| t^\rho J_{-\rho}(|\xi|) J_\rho(t|\xi|)\\
&\quad +t^{\rho}\f{|\xi_k|}{|\xi|} J_{-\rho+1}(|\xi|) J_{\rho}(t|\xi|) +t^{\rho+1}|\xi_k| J_{\rho-1}(t|\xi|) J_{-\rho+1}(|\xi|)
+t^\rho|\xi_k| J_{\rho-2}(|\xi|) J_{-\rho}(t|\xi|)\\
&\quad +t^{\rho+1}|\xi_k| J_{-\rho-1}(t|\xi|) J_{\rho-1}(|\xi|)+t^{\rho}\f{|\xi_k|}{|\xi|} J_{-\rho}(t|\xi|) J_{\rho-1}(|\xi|)
\end{aligned}
\end{equation}
and
\begin{equation}\label{e:q63}
\begin{aligned}
&|\p_{\xi_k}\Psi_1(t, 1, \xi)|
\lesssim \frac{|\xi_k|}{|\xi|} t^\rho J_{-\rho-1}(|\xi|) J_\rho(t|\xi|)+\frac{|\xi_k|}{|\xi|} t^{\rho+1} J_{\rho-1}(t|\xi|) J_{-\rho}(|\xi|)+\f{|\xi_k|}{|\xi|} t^\rho J_{\rho-1}(|\xi|) J_{-\rho}(t|\xi|)\\
&\quad +t^{\rho+1}\f{|\xi_k|}{|\xi|} J_{-\rho-1}(t|\xi|) J_{\rho}(|\xi|) +t^{\rho}\f{|\xi_k|}{|\xi|^2} J_{\rho}(t|\xi|) J_{-\rho}(|\xi|)+t^\rho\f{|\xi_k|}{|\xi|^2} J_{-\rho}(|\xi|) J_{\rho}(t|\xi|)\\
&\quad +t^{\rho}\f{|\xi_k|}{|\xi|^2} J_{-\rho}(t|\xi|) J_{\rho}(|\xi|)+t^{\rho}\f{|\xi_k|}{|\xi|^2} J_{-\rho}(t|\xi|) J_{\rho}(|\xi|).
\end{aligned}
\end{equation}
As in  \eqref{e:q2'}, by using \eqref{equ:q4}, \eqref{e:q62}-\eqref{e:q63} and  Lemma \ref{lem1}, we obtain that for $\f{14}{5}\leq\mu<3$,
\begin{equation}\label{e:q64'}
\begin{aligned}
\big\|\sum_{k=1}^3\xi_k\p_{\xi_k}\hat{v}\big\|_{L^2(A_1)}&=\big\|  \sum_{k=1}^{3}\left( \xi_k (\partial_{\xi_k} \Psi_0) \hat{v}_0
+ \xi_k (\partial_{\xi_k} \Psi_1) \hat{v}_1
+ \xi_k \Psi_0 \partial_{\xi_k} \hat{v}_0
+ \xi_k \Psi_1 \partial_{\xi_k} \hat{v}_1\right)\big\|_{L^2(A_1)}\\
& \lesssim \sum_{k=1}^3 \big\|(1+t|\xi|)^{-\f{\mu}{2}}t|\xi_k|(\hat{v}_{0}+\hat{v}_{1})\big\|_{L^{2}}+
\big\|(1+t|\xi|)^{-\f{\mu}{2}}\sum_{k=1}^3\widehat{(x_k\p_k)(v_0+v_1)}\big\|_{L^{2}}\\
&\quad+ \big\|(1+t|\xi|)^{-\f{\mu}{2}}\widehat{3(v_0+v_1)}\big\|_{L^{2}}\\
&\lesssim  t^{-\f{\mu}{2}}\left\|v_0+v_1\right\|_{Z, 1,(\f{6}{3+\mu}, 2)}.
\end{aligned}
\end{equation}
Similarly,
\begin{equation}\label{e:q65'}
\begin{aligned}
\big\|\sum_{k=1}^3\xi_k\p_{\xi_k}\hat{v}\big\|_{L^2(A_2)}
&\lesssim \big\|(1+t|\xi|)^{-\f{\mu}{2}}(\hat{v}_{0}+\hat{v}_{1})\big\|_{L^{2}}+
\big\|(1+t|\xi|)^{-\f{\mu}{2}}\sum_{k=1}^3\widehat{(x_k\p_k)(v_0+v_1)}\big\|_{L^{2}}\\
&\lesssim  t^{-\f{\mu}{2}}\left\|v_0+v_1\right\|_{Z, 1,(\f{6}{3+\mu}, 2)}.
\end{aligned}
\end{equation}
When $\mu\geq3$, we proceed as in \eqref{e:q34}-\eqref{e:q35}. By $(1+t|\xi|)^{-\frac{\mu}{2}} \le (1+t|\xi|)^{-\frac{3}{2}}$,
one then has
\begin{equation}\label{e:q64}
\begin{aligned}
\big\|\sum_{k=1}^3\xi_k\p_{\xi_k}\hat{v}\big\|_{L^2(A_1)}& \lesssim \sum_{k=1}^3 \big\|(1+t|\xi|)^{-\f{3}{2}}t|\xi_k|(\hat{v}_{0}+\hat{v}_{1})\big\|_{L^{2}}+
\big\|(1+t|\xi|)^{-\f{3}{2}}\sum_{k=1}^3\widehat{(x_k\p_k)(v_0+v_1)}\big\|_{L^{2}}\\
&\quad+ \big\|(1+t|\xi|)^{-\f{3}{2}}\widehat{3(v_0+v_1)}\big\|_{L^{2}}\\
&\lesssim  t^{\f32-\f{3}{1+\varepsilon_1}}\left\|v_0+v_1\right\|_{Z, 1,(1+\varepsilon_1, 2)}
\end{aligned}
\end{equation}
and
\begin{equation}\label{e:q65}
\begin{aligned}
\big\|\sum_{k=1}^3\xi_k\p_{\xi_k}\hat{v}\big\|_{L^2(A_2)}
&\lesssim \big\|(1+t|\xi|)^{-\f{3}{2}}(\hat{v}_{0}+\hat{v}_{1})\big\|_{L^{2}}+
\big\|(1+t|\xi|)^{-\f{3}{2}}\sum_{k=1}^3\widehat{(x_k\p_k)(v_0+v_1)}\big\|_{L^{2}}\\
&\lesssim  t^{\f32-\f{3}{1+\varepsilon_1}}\left\|v_0+v_1\right\|_{Z, 1,(1+\varepsilon_1, 2)}.
\end{aligned}
\end{equation}
Collecting \eqref{e:q64'}-\eqref{e:q65},  \eqref{e:q38} and \eqref{e:q61},
\eqref{e:q59'}-\eqref{e:q59} can be obtained.

\vskip 0.2 true cm

{\bf Part 4.  The treatment  of $\|L_jv(t, \cdot)\|_{L^2(\R^3)}(j=1, 2, 3)$ }

\vskip 0.2 true cm

We now claim
\begin{equation}\label{e:q66'}
\|L_jv(t, \cdot)\|_{L^2(\R^3)} \lesssim t^{-\f{\mu}{2}}\left\|v_0\right\|_{Z, 1, 2}+t^{-\f{\mu}{2}}\left\|v_1\right\|_{Z, 1,(\f65, 2)}+t^{-\f{\mu}{2}}\left\|v_0+v_1\right\|_{Z, 1,(\f{6}{3+\mu}, 2)}\quad \text{for}\, \f{14}{5}\leq\mu<3,
\end{equation}
and
\begin{equation}\label{e:q66}
\|L_jv(t, \cdot)\|_{L^2(\R^3)} \lesssim t^{-\f{\mu}{2}}\left\|v_0\right\|_{Z, 1, 2}+t^{-\f{\mu}{2}}\left\|v_1\right\|_{Z, 1,(\f65, 2)}+t^{\f32-\f{3}{1+\varepsilon_1}}\left\|v_0+v_1\right\|_{Z, 1,(1+\varepsilon_1, 2)}\quad \text{for}\, \mu\geq3.
\end{equation}

Following the same procedure as in \eqref{e:q40} and \eqref{e:q47}, the proof of \eqref{e:q66'}-\eqref{e:q66} can be reduced to estimate  $\|t|\xi_3|\hat{v}\|_{L^2(A_1\bigcup A_2)}$ since the decay estimates for $\|t\p_t\hat{v}\|_{L^2(\R^3)}$ have already been
derived in \eqref{e:q61} and \eqref{e:q61''}. For $\f{14}{5}\leq\mu<3$,
   \eqref{e:q49} and  \eqref{e:q51} imply
\begin{equation}\label{e:q67'}
\begin{aligned}
\left\|t|\xi_3|\hat{v}(t, \cdot)\right\|_{L^2(A_1)}\lesssim\left\|(1+t|\xi|)^{-\f{\mu}{2}}t|\xi|(\hat{v}_{0}+\hat{v}_{1})\right\|_{L^{2}(A_1)}
\lesssim t^{-\f{\mu}{2}}\left\|v_0+v_1\right\|_{Z, 1,(\f{6}{3+\mu}, 2)}
\end{aligned}
\end{equation}
and
\begin{equation}\label{e:q68'}
\begin{aligned}
\left\|t|\xi_3|\hat{v}(t, \cdot)\right\|_{L^2(A_2)} & \leq\left\|(1+t|\xi|)^{\f{\mu}{2}}(1+t|\xi|)^{-\f{\mu}{2}}t|\xi|(\hat{v}_{0}+\hat{v}_{1})\right\|_{L^{2}} \lesssim t^{-\f{\mu}{2}}\|v_0+v_1\|_{Z, 1, (\f{6}{3+\mu}, 2)}.
\end{aligned}
\end{equation}
Similarly, we have that for $\mu\geq3$,
\begin{equation}\label{e:q67}
\begin{aligned}
\left\|t|\xi_3|\hat{v}(t, \cdot)\right\|_{L^2(A_1)}  \lesssim\left\|(1+t|\xi|)^{-\f32}t|\xi|(\hat{v}_{0}+\hat{v}_{1})\right\|_{L^{2}(A_1)}
\lesssim t^{\f32-\f{3}{1+\varepsilon_1}}\left\|v_0+v_1\right\|_{Z, 1,(1+\varepsilon_1, 2)}
\end{aligned}
\end{equation}
and
\begin{equation}\label{e:q68}
\begin{aligned}
\left\|t|\xi_3|\hat{v}(t, \cdot)\right\|_{L^2(A_2)} & \leq\left\|(1+t|\xi|)^{\f32}(1+t|\xi|)^{-\f32}t|\xi|(\hat{v}_{0}+\hat{v}_{1})\right\|_{L^{2}} \lesssim t^{\f32-\f{3}{1+\varepsilon_1}}\|v_0+v_1\|_{Z, 1, (1+\varepsilon_1, 2)}.
\end{aligned}
\end{equation}

Therefore, combining \eqref{e:q67'}-\eqref{e:q68} and \eqref{e:q42} gives
\begin{equation}\label{e:q69}
\|t|\xi_3|\hat{v}\|_{L^2(\R^3)} \lesssim\begin{cases} t^{-\f{\mu}{2}}\left\|v_0\right\|_{Z, 1, 2}+t^{-\f{\mu}{2}}\left\|v_1\right\|_{Z, 1,(\f65, 2)}+t^{-\f{\mu}{2}}\left\|v_0+v_1\right\|_{Z, 1,(\f{6}{3+\mu}, 2)},& \f{14}{5}\leq\mu<3,\\ t^{-\f{\mu}{2}}\left\|v_0\right\|_{Z, 1, 2}+t^{-\f{\mu}{2}}\left\|v_1\right\|_{Z, 1,(\f65, 2)}+t^{\f32-\f{3}{1+\varepsilon_1}}\left\|v_0+v_1\right\|_{Z, 1,(1+\varepsilon_1, 2)},& \mu\geq3.
\end{cases}
\end{equation}
Together with \eqref{e:q61} and \eqref{e:q61''}, this completes the proof of \eqref{e:q66'}-\eqref{e:q66}.
\vskip 0.2 true cm

{\bf Part 5.  The treatment of $\|\Omega_{kj}v(t, \cdot)\|_{L^2(\R^3)}$$(1\le k<j\le3)$ }

\vskip 0.2 true cm
We now start to prove
\begin{equation}\label{e:q70'}
\|\Omega_{kj}v(t, \cdot)\|_{L^2(\R^3)} \lesssim t^{-\f{\mu}{2}}\left\|v_0\right\|_{Z, 1, 2}+t^{-\f{\mu}{2}}\left\|v_1\right\|_{Z, 1,(\f65, 2)}+t^{-\f{\mu}{2}}\left\|v_0+v_1\right\|_{Z, 1,(\f{6}{3+\mu}, 2)}\quad \text{for}\, \f{14}{5}\leq\mu<3
\end{equation}
and
\begin{equation}\label{e:q70}
\|\Omega_{kj}v(t, \cdot)\|_{L^2(\R^3)} \lesssim t^{-\f{\mu}{2}}\left\|v_0\right\|_{Z, 1, 2}+t^{-\f{\mu}{2}}\left\|v_1\right\|_{Z, 1,(\f65, 2)}+t^{\f32-\f{3}{1+\varepsilon_1}}\left\|v_0+v_1\right\|_{Z, 1,(1+\varepsilon_1, 2)}\quad \text{for}\, \mu\geq3.
\end{equation}

As treated in \eqref{e:q46}  and \eqref{e:q47}, it suffices to deal with the term $\|t|\xi_j|\hat{v}(t, \cdot)\|_{L^2(A_1\bigcup A_2)}$$(1\leq j\leq3)$.
Note that by \eqref{e:q67'}-\eqref{e:q68}, it holds that for $j=1,2,3$ and $\f{14}{5}\leq\mu<3$,
$$
\|t|\xi_j|\hat{v}(t, \cdot)\|_{L^2(A_1\bigcup A_2)} \lesssim t^{-\f{\mu}{2}}\left\|v_0+v_1\right\|_{Z, 1,(\f{6}{3+\mu}, 2)},
$$
and for $\mu\geq3$,
$$
\|t|\xi_j|\hat{v}(t, \cdot)\|_{L^2(A_1\bigcup A_2)} \lesssim t^{\f32-\f{3}{1+\varepsilon_1}}\left\|v_0+v_1\right\|_{Z, 1,(1+\varepsilon_1, 2)}.
$$
This, together with \eqref{e:q42}, yields \eqref{e:q70'} and \eqref{e:q70} immediately.
Therefore, by collecting the results of  Parts 1-5, the proof of Lemma \ref{lem2-1} is completed.
\end{proof}

Based on  Lemmas \ref{lem2}--\ref{lem2-1},  we now deal with the first-order derivatives of solution $v$
to \eqref{equ:q16}.
\begin{lemma}\label{lem3}
Let $v$ be the solution to \eqref{equ:q16} and $\varepsilon_1\in(0, 1)$. Then for $\f{14}{5}\leq\mu<3$,
\begin{equation}\label{e:q71'}
\|\p v(t, \cdot)\|_{Z, 1, 2} \lesssim t^{-\f{\mu}{2}}\left\|\nabla v_0\right\|_{Z, 1, 2}+t^{-\f{\mu}{2}}\left\| v_1\right\|_{Z, 1, 2}+t^{-\f{\mu}{2}}\left\|v_0+v_1\right\|_{Z, 1,(\f{6}{3+\mu},2)},
\end{equation}
and for $\mu\geq3$,
\begin{equation}\label{e:q71}
\|\p v(t, \cdot)\|_{Z, 1, 2} \lesssim t^{-\f{\mu}{2}}\left\|\nabla v_0\right\|_{Z, 1, 2}+t^{-\f{\mu}{2}}\left\| v_1\right\|_{Z, 1, 2}+t^{\f32-\f{3}{1+\varepsilon_1}}\left\|v_0+v_1\right\|_{Z, 1,(1+\varepsilon_1,2)}.
\end{equation}
\end{lemma}
\begin{proof}
Following the proof of Lemma \ref{lem2}, we decompose the frequency space $\mathbb{R}_{\xi}^{3}$ into three zones: 
$D_1$, $D_2$ and $D_3$.
In view of the definition of $\|\p v\|_{Z, 1, 2}$, we distinguish  the following five parts.
\vskip 0.2 true cm

{\bf Part 1.  The treatment  of $\|\p v(t, \cdot)\|_{L^2(\R^3)}$ }

\vskip 0.2 true cm

We next prove that for $\f{14}{5}\leq\mu<3$,
\begin{equation}\label{e:q72'}
\|\p v(t, \cdot)\|_{L^2} \lesssim t^{-\f{\mu}{2}}\left\|\nabla v_0\right\|_{L^ 2}+t^{-\f{\mu}{2}}\left\| v_1\right\|_{L^ 2}+t^{-\f{\mu}{2}}\left\|v_0+v_1\right\|_{(\f{6}{3+\mu},2)},
\end{equation}
and for $\mu\geq3$,
\begin{equation}\label{e:q72}
\|\p v(t, \cdot)\|_{L^2} \lesssim t^{-\f{\mu}{2}}\left\|\nabla v_0\right\|_{L^ 2}+t^{-\f{\mu}{2}}\left\| v_1\right\|_{L^ 2}+t^{\f32-\f{3}{1+\varepsilon_1}}\left\|v_0+v_1\right\|_{(1+\varepsilon_1,2)}.
\end{equation}
For the zone $D_k$ $(k=1,3)$, by  \eqref{e:q1}, \eqref{e:q49} and \eqref{e:q10} , \eqref{e:q51} respectively, we obtain
\begin{equation}\label{e:q73'}
\begin{aligned}
\left\||\xi|\hat{v}(t, \cdot)\right\|_{L^2(D_k)}\leq\left\|(1+t|\xi|)^{-\f{\mu}{2}}(\hat{v_0}+\hat{v_1})\right\|_{L^2(D_k)}\lesssim t^{-\f{\mu}{2}}\left\|v_0+v_1\right\|_{(\f{6}{3+\mu},2)}\quad \text{for}\, \f{14}{5}\leq\mu<3
\end{aligned}
\end{equation}
and
\begin{equation}\label{e:q73}
\begin{aligned}
\left\||\xi|\hat{v}(t, \cdot)\right\|_{L^2(D_k)}&\leq\left\|(1+t|\xi|)^{-\f{\mu}{2}}(\hat{v_0}+\hat{v_1})\right\|_{L^2(D_k)}\\
&\leq\left\|(1+t|\xi|)^{-\f32}(\hat{v_0}+\hat{v_1})\right\|_{L^2(D_k)}\lesssim t^{\f32-\f{3}{1+\varepsilon_1}}\left\|v_0+v_1\right\|_{(1+\varepsilon_1,2)}\quad \text{for}\, \mu \geq3.
\end{aligned}
\end{equation}
Meanwhile, for the zone $D_2$, it follows directly from \eqref{e:q5} that for $\mu\geq\f{14}{5}$,
$$
\begin{aligned}
\left\||\xi|\hat{v}(t, \cdot)\right\|_{L^2(D_2)}\lesssim t^{-\frac{\mu}{2}}\left\||\xi|\hat{v}_{0}\right\|_{L^{2}}+t^{-\f{\mu}{2}} \|\hat{v}_{1}\|_{L^{2}}\lesssim t^{-\frac{\mu}{2}}\left\|\nabla v_{0}\right\|_{L^{2}}+ t^{-\f{\mu}{2}}\|v_1\|_{L^{2}}.
\end{aligned}
$$
This, together with \eqref{e:q73'}-\eqref{e:q73}, yields
$$
\left\||\xi|\hat{v}(t, \cdot)\right\|_{L^2(\R^3)}\lesssim \begin{cases} t^{-\frac{\mu}{2}}\left\|\nabla v_{0}\right\|_{L^{2}}+ t^{-\f{\mu}{2}}\|v_1\|_{L^{2}} + t^{-\f{\mu}{2}}\left\|v_0+v_1\right\|_{(\f{6}{3+\mu},2)},& \f{14}{5}\leq\mu<3,\\
t^{-\frac{\mu}{2}}\left\|\nabla v_{0}\right\|_{L^{2}}+ t^{-\f{\mu}{2}}\|v_1\|_{L^{2}}+ t^{\f32-\f{3}{1+\varepsilon_1}}\left\|v_0+v_1\right\|_{(1+\varepsilon_1,2)},& \mu \geq3.
\end{cases}
$$
Analogously to the treatment of $\left\||\xi|\hat{v}(t, \cdot)\right\|_{L^2(\R^3)}$,  we have
$$
\left\|\p_t\hat{v}(t, \cdot)\right\|_{L^2(\R^3)}\lesssim \begin{cases} t^{-\frac{\mu}{2}}\left\|\nabla v_{0}\right\|_{L^{2}}+ t^{-\f{\mu}{2}}\|v_1\|_{L^{2}} + t^{-\f{\mu}{2}}\left\|v_0+v_1\right\|_{(\f{6}{3+\mu},2)},&  \f{14}{5}\leq\mu<3,\\
t^{-\frac{\mu}{2}}\left\|\nabla v_{0}\right\|_{L^{2}}+ t^{-\f{\mu}{2}}\|v_1\|_{L^{2}}+ t^{\f32-\f{3}{1+\varepsilon_1}}\left\|v_0+v_1\right\|_{(1+\varepsilon_1,2)},&  \mu \geq3.
\end{cases}
$$
Then \eqref{e:q72'} and \eqref{e:q72} follow from Parseval's theorem.

\vskip 0.2 true cm

{\bf Part 2.  The treatment  of $\|\p(\p v)(t, \cdot)\|_{L^2(\R^3)}$ }

\vskip 0.2 true cm

We now show
\begin{equation}\label{e:q74'}
\|\p(\p v)(t, \cdot)\|_{L^2(\R^3)} \lesssim t^{-\f{\mu}{2}}\left\|\nabla v_0\right\|_{Z, 1, 2}+t^{-\f{\mu}{2}}\left\| v_1\right\|_{Z, 1, 2}+t^{-\f{\mu}{2}}\left\|v_0+v_1\right\|_{Z, 1,(\f{6}{3+\mu},2)}\quad \text{for}\, \f{14}{5}\leq\mu<3
\end{equation}
and
\begin{equation}\label{e:q74}
\|\p(\p v)(t, \cdot)\|_{L^2(\R^3)}\lesssim t^{-\f{\mu}{2}}\left\|\nabla v_0\right\|_{Z, 1, 2}+t^{-\f{\mu}{2}}\left\| v_1\right\|_{Z, 1, 2}+t^{\f32-\f{3}{1+\varepsilon_1}}\left\|v_0+v_1\right\|_{Z, 1,(1+\varepsilon_1,2)} \quad \text{for}\, \mu\geq3.
\end{equation}
Note that
\begin{equation}\label{e:q75}
\begin{aligned}
\|\p(\p v)(t, \cdot)\|_{L^2(\R^3)} \lesssim \|\p_t^2\hat{v}\|_{L^2(\R^3)}+\left\||\xi|\p_t\hat{v}(t, \cdot)\right\|_{L^2(\R^3)}+ \sum_{k,j=1}^3\left\||\xi_j||\xi_k|\hat{v}(t, \cdot)\right\|_{L^2(\R^3)}.
\end{aligned}
\end{equation}
When $\mu \geq \f{14}{5}$, it follows from \eqref{equ:q9-1}, \eqref{equ:q9-5}, \eqref{l-3} and Lemma \ref{lem1} that for the zone $D_1$,
\begin{equation}\label{e:q76}
\begin{aligned}
& \left|\p_t^2\Psi_0(t, 1,\xi)\right| \lesssim t^\rho|\xi|^3(t|\xi|)^{-\f{1}{2}}|\xi|^{\rho-1}\leq (t|\xi|)^{-\frac{\mu}{2}}\cdot|\xi|, \\
& \left|\p_t^2\Psi_1(t, 1, \xi)\right|  \lesssim t^\rho|\xi|^2(t|\xi|)^{-\f{1}{2}}|\xi|^{\rho}\leq (t|\xi|)^{-\frac{\mu}{2}}\cdot|\xi|.
\end{aligned}
\end{equation}
Meanwhile, for $D_2$ and $D_3$, we can obtain, respectively,
\begin{equation}\label{e:q77}
\begin{aligned}
& \left|\p_t^2\Psi_0(t, 1,\xi)\right| \lesssim t^{-\frac{\mu}{2}}|\xi|\cdot|\xi|, \\
& \left|\p_t^2\Psi_1(t, 1, \xi)\right|  \lesssim t^{-\frac{\mu}{2}}\cdot|\xi|
\end{aligned}
\end{equation}
and
\begin{equation}\label{e:q78}
\begin{aligned}
& \left|\p_t^2\Psi_0(t, 1,\xi)\right| \lesssim t^{-1}\cdot|\xi|, \\
& \left|\p_t^2\Psi_1(t, 1, \xi)\right|  \lesssim t^{-1}\cdot|\xi|.
\end{aligned}
\end{equation}
Comparing   \eqref{e:q76}--\eqref{e:q78} with  \eqref{e:q13}, \eqref{e:q17}, \eqref{e:q21}, \eqref{e:q54} and \eqref{e:q58'}, the estimate for $\left\||\xi|\p_t\hat{v}(t, \cdot)\right\|_{L^2(\R^3)}$ can be taken as for $\|\p_t^2\hat{v}\|_{L^2(\R^3)}$. Thus, to prove  \eqref{e:q74'}-\eqref{e:q74}, it suffices  to estimate $\left\|\p_t^2\hat{v}(t, \cdot)\right\|_{L^2(\R^3)}$ and $\left\||\xi_j||\xi_k|\hat{v}(t, \cdot)\right\|_{L^2(\R^3)}$ $(k, j=1, 2, 3)$.

For $D_1$, in view of  \eqref{e:q76}, \eqref{e:q1} and \eqref{e:q49}, it holds that for  $k,j=1,2,3$ and $\f{14}{5}\leq\mu<3$,
$$
\begin{aligned}
\left\|\p_t^2\hat{v}(t, \cdot)\right\|_{L^2(D_1)}
\lesssim \big\|(1+t|\xi|)^{-\f{\mu}{2}}|\xi|(\hat{v}_{0}+\hat{v}_{1})\big\|_{L^{2}}
\lesssim t^{-\f{\mu}{2}}\left\|v_0+v_1\right\|_{(\f{6}{3+\mu},2)}
\end{aligned}
$$
and
$$
\begin{aligned}
\left\||\xi_j||\xi_k|\hat{v}(t, \cdot)\right\|_{L^2(D_1)}
\lesssim \big\|(1+t|\xi|)^{-\f{\mu}{2}}(\hat{v}_{0}+\hat{v}_{1})\big\|_{L^{2}}
\lesssim t^{-\f{\mu}{2}}\left\|v_0+v_1\right\|_{(\f{6}{3+\mu},2)}.
\end{aligned}
$$
Similarly, for $\mu\geq3$, one has
$$
\begin{aligned}
\left\|\p_t^2\hat{v}(t, \cdot)\right\|_{L^2(D_1)}
&\lesssim \big\|(1+t|\xi|)^{-\f{\mu}{2}}|\xi|(\hat{v}_{0}+\hat{v}_{1})\big\|_{L^{2}} \\
&\lesssim \big\|(1+t|\xi|)^{-\f{3}{2}}(\hat{v}_{0}+\hat{v}_{1})\big\|_{L^{2}} \\
&
\lesssim t^{\f32-\f{3}{1+\varepsilon_1}}\left\|v_0+v_1\right\|_{(1+\varepsilon_1,2)}
\end{aligned}
$$
and
$$
\begin{aligned}
\left\||\xi_j||\xi_k|\hat{v}(t, \cdot)\right\|_{L^2(D_1)}
\lesssim \big\|(1+t|\xi|)^{-\f{3}{2}}(\hat{v}_{0}+\hat{v}_{1})\big\|_{L^{2}}
\lesssim t^{\f32-\f{3}{1+\varepsilon_1}}\left\|v_0+v_1\right\|_{(1+\varepsilon_1,2)}.
\end{aligned}
$$
For $D_2$ and $\mu\geq\f{14}{5}$, applying \eqref{e:q77} and \eqref{e:q5} yields
$$
\begin{aligned}
\left\|\p_t^2\hat{v}(t, \cdot)\right\|_{L^2(D_2)}
\lesssim t^{-\frac{\mu}{2}}\left\||\xi||\xi|\hat{v}_{0}\right\|_{L^{2}}+t^{-\f{\mu}{2}} \||\xi|\hat{v}_{1}\|_{L^{2}}\lesssim t^{-\f{\mu}{2}}\left\|\nabla v_0\right\|_{Z, 1, 2}+t^{-\f{\mu}{2}}\left\|v_1\right\|_{Z,1,2}
\end{aligned}
$$
and
$$
\left\||\xi_j||\xi_k|\hat{v}(t, \cdot)\right\|_{L^2(D_2)}\lesssim t^{-\frac{\mu}{2}}\left\||\xi_j||\xi|v_{0}\right\|_{L^{2}}+t^{-\f{\mu}{2}} \||\xi_j|\hat{v}_{1}\|_{L^{2}}\lesssim t^{-\f{\mu}{2}}\left\|\nabla v_0\right\|_{Z, 1, 2}+t^{-\f{\mu}{2}}\left\|v_1\right\|_{Z,1,2}.
$$
Furthermore, in $D_3$, by \eqref{e:q78}, \eqref{e:q10} and \eqref{e:q51}, we have
$$
\left\|\p_t^2\hat{v}\right\|_{L^2(D_3)}\lesssim \begin{cases} t^{-2}\left\|(1+t|\xi|)^{-\f{\mu}{2}}(\hat{v}_{0}+\hat{v}_{1})\right\|_{L^{2}}
\lesssim t^{-\f{\mu}{2}-2}\left\|v_0+v_1\right\|_{(\f{6}{3+\mu},2)},& \f{14}{5}\leq\mu<3,\\ t^{-2}\left\|(1+t|\xi|)^{-\f32}(\hat{v}_{0}+\hat{v}_{1})\right\|_{L^{2}} \lesssim t^{\f32-\f{3}{1+\varepsilon_1}-2}\left\|v_0+v_1\right\|_{(1+\varepsilon_1,2)},& \mu\geq3
\end{cases}
$$
and
$$
\left\||\xi_j||\xi_k|\hat{v}\right\|_{L^2(D_3)}\lesssim\begin{cases} t^{-2}\left\|(1+t|\xi|)^{-\f{\mu}{2}}(\hat{v}_{0}+\hat{v}_{1})\right\|_{L^{2}} \lesssim t^{-\f{\mu}{2}-2}\left\|v_0+v_1\right\|_{(\f{6}{3+\mu},2)},& \f{14}{5}\leq\mu<3,\\ t^{-2}\left\|(1+t|\xi|)^{-\f32}(\hat{v}_{0}+\hat{v}_{1})\right\|_{L^{2}} \lesssim t^{-\f12-\f{3}{1+\varepsilon_1}}\left\|v_0+v_1\right\|_{(1+\varepsilon_1,2)},& \mu\geq3.
\end{cases}
$$
Combining the above estimates yields \eqref{e:q74'}-\eqref{e:q74}.
\vskip 0.2 true cm

{\bf Part 3.  The treatment of $\|L_0(\p v)(t, \cdot)\|_{L^2(\R^3)}$}

\vskip 0.2 true cm

We now establish
\begin{equation}\label{e:q79'}
\|L_0(\p v)(t, \cdot)\|_{L^2(\R^3)}  \lesssim t^{-\f{\mu}{2}}\left\|\nabla v_0\right\|_{Z, 1, 2}+t^{-\f{\mu}{2}}\left\| v_1\right\|_{Z, 1, 2}+t^{-\f{\mu}{2}}\left\|v_0+v_1\right\|_{Z, 1,(\f{6}{3+\mu},2)} \quad \text{for}\, \f{14}{5}\leq\mu<3,
\end{equation}
and
\begin{equation}\label{e:q79}
\|L_0(\p v)(t, \cdot)\|_{L^2(\R^3)} \lesssim t^{-\f{\mu}{2}}\left\|\nabla v_0\right\|_{Z, 1, 2}+t^{-\f{\mu}{2}}\left\| v_1\right\|_{Z, 1, 2}+t^{\f32-\f{3}{1+\varepsilon_1}}\left\|v_0+v_1\right\|_{Z, 1,(1+\varepsilon_1,2)}\quad \text{for}\, \mu\geq3.
\end{equation}
Analogously to \eqref{e:q75} and \eqref{e:q26}, it suffices to estimate $\||\xi|v\|_{L^2(\R^3)}$, $\|t|\xi|\p_t\hat{v}\|_{L^2(\R^3)}$ and $\||\xi|\ds\sum_{k=1}^3\xi_k\p_{\xi_k}\hat{v}\|_{L^2(\R^3)}$.

By \eqref{e:q72'} and \eqref{e:q72}, the first term  $\||\xi|v\|_{L^2(\R^3)}$ satisfies
\begin{equation}\label{e:q80}
\||\xi| v(t, \cdot)\|_{L^2(\R^3)}\leq\|\p v(t, \cdot)\|_{L^2(\R^3)} \lesssim \begin{cases} t^{-\f{\mu}{2}}\left\|\nabla v_0\right\|_{L^ 2}+t^{-\f{\mu}{2}}\left\| v_1\right\|_{L^ 2}+t^{-\f{\mu}{2}}\left\|v_0+v_1\right\|_{(\f{6}{3+\mu},2)},
\\ \qquad\qquad\qquad\qquad\qquad\qquad\qquad\qquad\qquad \f{14}{5}\leq\mu<3,\\ t^{-\f{\mu}{2}}\left\|\nabla v_0\right\|_{L^ 2}+t^{-\f{\mu}{2}}\left\| v_1\right\|_{L^ 2}+t^{\f32-\f{3}{1+\varepsilon_1}}\left\|v_0+v_1\right\|_{(1+\varepsilon_1,2)},\\ \qquad\qquad\qquad\qquad\qquad\qquad\qquad\qquad\qquad \mu\geq3.
\end{cases}
\end{equation}
Next we estimate  $\|t|\xi|\p_t\hat{v}\|_{L^2(\R^3)}$. In  $D_1$,   \eqref{e:q13} and \eqref{e:q54} yield
$$
\left\|t|\xi|\p_t\hat{v}(t, \cdot)\right\|_{L^2(D_1)}
\lesssim \begin{cases}\big\|(1+t|\xi|)^{-\f{\mu}{2}}(\hat{v}_{0}+\hat{v}_{1})\big\|_{L^2}
\lesssim t^{-\f{\mu}{2}}\left\|v_0+v_1\right\|_{(\f{6}{\mu+3},2)}, &  \f{14}{5}\leq\mu<3,  \\ \big\|(1+t|\xi|)^{-\f{3}{2}}(\hat{v}_{0}+\hat{v}_{1})\big\|_{L^2}
\lesssim t^{\f32-\f{3}{1+\varepsilon_1}}\left\|v_0+v_1\right\|_{(1+\varepsilon_1,2)},
&  \mu\geq3.\end{cases}
$$
In $D_2$,  \eqref{e:q17} implies that for $\mu\geq\f{14}{5}$,
$$
\left\|t|\xi|\p_t\hat{v}(t, \cdot)\right\|_{L^2(D_2)}
 \lesssim t^{-\frac{\mu}{2}}\left\|t|\xi||\xi|\hat{v}_{0}\right\|_{L^{2}}+t^{-\f{\mu}{2}} \|t|\xi|\hat{v}_{1}\|_{L^{2}} \lesssim t^{-\f{\mu}{2}}\left\|\nabla v_0\right\|_{Z,1,2}+t^{-\f{\mu}{2}}\left\|v_1\right\|_{Z,1,2}.
$$
In $D_3$, combining \eqref{e:q21} and \eqref{e:q58'} gives
$$
\left\|t|\xi|\p_t\hat{v}(t, \cdot)\right\|_{L^2(D_3)}
\lesssim \begin{cases}t^{-1}\big\|(1+t|\xi|)^{-\f{\mu}{2}}(\hat{v}_{0}+\hat{v}_{1})\big\|_{L^2}
\lesssim t^{-\f{\mu}{2}-1}\left\|v_0+v_1\right\|_{(\f{6}{3+\mu},2)}, &  \f{14}{5}\leq\mu<3,  \\ t^{-1}\big\|(1+t|\xi|)^{-\f{3}{2}}(\hat{v}_{0}+\hat{v}_{1})\big\|_{L^2}
\lesssim t^{\f12-\f{3}{1+\varepsilon_1}}\left\|v_0+v_1\right\|_{(1+\varepsilon_1,2)},
&  \mu\geq3.\end{cases}
$$
Summing the estimates over $D_1$, $D_2$ and $D_3$ leads to
\begin{equation}\label{e:q83}
\|t|\xi|\p_t\hat{v}\|_{L^2(\R^3)}\lesssim \begin{cases}t^{-\f{\mu}{2}}\left\|\nabla v_0\right\|_{Z, 1,  2}+t^{-\f{\mu}{2}}\left\| v_1\right\|_{Z, 1,  2}+t^{-\f{\mu}{2}}\left\|v_0+v_1\right\|_{Z, 1,  (\f{6}{3+\mu},2)}, &  \f{14}{5}\leq\mu<3,  \\ t^{-\f{\mu}{2}}\left\|\nabla v_0\right\|_{Z, 1,  2}+t^{-\f{\mu}{2}}\left\| v_1\right\|_{Z, 1,  2}+t^{\f32-\f{3}{1+\varepsilon_1}}\left\|v_0+v_1\right\|_{Z, 1,  (1+\varepsilon_1,2)},
&  \mu\geq3.\end{cases}
\end{equation}

We now treat $\||\xi|\ds\sum_{k=1}^3\xi_k\p_{\xi_k}\hat{v}\|_{L^2(\R^3)}$.
As treated in \eqref{e:q34}, \eqref{e:q35} and \eqref{e:q38}, one can obtain that
\begin{equation}\label{e:q84}
\begin{aligned}
&\||\xi|\ds\sum_{k=1}^3\xi_k\p_{\xi_k}\hat{v}\|_{L^2(D_1)}
\\
&=\big\||\xi|  \ds\sum_{k=1}^{3}\left( \xi_k (\partial_{\xi_k} \Psi_0) \hat{v}_0
+ \xi_k (\partial_{\xi_k} \Psi_1) \hat{v}_1
+ \xi_k \Psi_0 \partial_{\xi_k} \hat{v}_0
+ \xi_k \Psi_1 \partial_{\xi_k} \hat{v}_1\right)\big\|_{L^2(D_1)}\\
&\lesssim \big\|(t|\xi|)^{-\frac{\mu}{2}}(\ds\sum_{k=1}^3t|\xi_k|)(\hat{v}_{0}+\hat{v}_{1})\big\|_{L^{2}}+ \big\|(t|\xi|)^{-\f{\mu}{2}}(\ds\sum_{k=1}^3\xi_k\p_{\xi_k})(\hat{v}_{0}+\hat{v}_{1})\big\|_{L^{2}}\\
& \lesssim \ds\sum_{k=1}^3 \big\|(1+t|\xi|)^{-\f{3}{2}}t|\xi_k|(\hat{v}_{0}+\hat{v}_{1})\big\|_{L^{2}}+
\big\|(1+t|\xi|)^{-\f{3}{2}}\ds\sum_{k=1}^3\widehat{(x_k\p_k)(v_0+v_1)}\big\|_{L^{2}}\\
&\quad+ \left\|(1+t|\xi|)^{-\f{3}{2}}\widehat{3(v_0+v_1)}\right\|_{L^{2}}\\
&\lesssim  t^{\f32-\f{3}{1+\varepsilon_1}}\left\|v_0+v_1\right\|_{Z, 1,(1+\varepsilon_1, 2)}
\end{aligned}
\end{equation}
and
\begin{equation}\label{e:q85}
\begin{aligned}
&\||\xi|\ds\sum_{k=1}^3\xi_k\p_{\xi_k}\hat{v}\|_{L^2(D_2)}\\
&\lesssim \big\|t^{-\frac{\mu}{2}}|\xi|\cdot|\xi|\hat{v}_{0}\big\|_{L^{2}}+ \big\|t^{-\frac{\mu}{2}}|\xi|\hat{v}_{1}\big\|_{L^{2}}+ \big\|t^{-\f{\mu}{2}}(\ds\sum_{k=1}^3\xi_k\p_{\xi_k})\hat{v}_{1}\big\|_{L^{2}} +\big\|t^{-\f{\mu}{2}}(\ds\sum_{k=1}^3\xi_k\p_{\xi_k})\hat{v}_{0}\big\|_{L^{2}}\\
& \lesssim t^{-\frac{\mu}{2}}\big\|\nabla v_{0}\big\|_{Z, 1, 2}+ t^{-\frac{\mu}{2}}\big\|v_{1}\big\|_{Z, 1, 2}+ t^{-\f{\mu}{2}}
\big\|\ds\sum_{k=1}^3\widehat{(x_k\p_k)v_1}\big\|_{L^{2}}\\
&\quad+ t^{-\f{\mu}{2}} \big\|3|\xi|\widehat{v_1}\big\|_{L^{2}}+t^{-\f{\mu}{2}} \big\|\ds\sum_{k=1}^3\widehat{(x_k\p_k)v_0}\big\|_{L^{2}}+ t^{-\f{\mu}{2}} \big\|3|\xi|\widehat{v_0}\big\|_{L^{2}}\\
&\lesssim  t^{-\f{\mu}{2}}\left\|\nabla v_0\right\|_{Z, 1, 2}+t^{-\f{\mu}{2}}\left\|v_1\right\|_{Z, 1, 2}
\end{aligned}
\end{equation}
and
\begin{equation}\label{e:q86}
\begin{aligned}
&\||\xi|\ds\sum_{k=1}^3\xi_k\p_{\xi_k}\hat{v}\|_{L^2(D_3)}\\
&\lesssim \big\||\xi|(1+t|\xi|)^{-\f32}(\hat{v}_{0}+\hat{v}_{1})\big\|_{L^{2}}+ \big\||\xi|(1+t|\xi|)^{-\f32}(\ds\sum_{k=1}^3\xi_k\p_{\xi_k})(\hat{v}_{0}+\hat{v}_{1})\big\|_{L^{2}}\\
&\lesssim \big\|(1+t|\xi|)^{-\f32}(\hat{v}_{0}+\hat{v}_{1})\big\|_{L^{2}}+ \big\|(1+t|\xi|)^{-\f32}(\ds\sum_{k=1}^3\xi_k\p_{\xi_k})(\hat{v}_{0}+\hat{v}_{1})\big\|_{L^{2}}\\
& \lesssim \big\|(1+t|\xi|)^{-\f{3}{2}}(\hat{v}_{0}+\hat{v}_{1})\big\|_{L^{2}}+
\big\|(1+t|\xi|)^{-\f{3}{2}}\ds\sum_{k=1}^3\widehat{(x_k\p_k)(v_0+v_1)}\big\|_{L^{2}}\\
&\lesssim  t^{\f32-\f{3}{1+\varepsilon_1}}\big\|v_0+v_1\big\|_{Z, 1,(1+\varepsilon_1, 2)}.
\end{aligned}
\end{equation}
While for $\f{14}{5}\leq\mu<3$, analogously treated as in \eqref{e:q64'}-\eqref{e:q65'} and \eqref{e:q38}, we arrive at
\begin{equation}\label{e:q86'}
\begin{aligned}
&\||\xi|\ds\sum_{k=1}^3\xi_k\p_{\xi_k}\hat{v}\|_{L^2(\R^3)}
\lesssim t^{-\f{\mu}{2}}\left\|\nabla v_0\right\|_{Z, 1,  2}+t^{-\f{\mu}{2}}\left\| v_1\right\|_{Z, 1,  2}
+t^{-\f{\mu}{2}}\left\|v_0+v_1\right\|_{Z, 1,  (\f{6}{3+\mu},2)}.
\end{aligned}
\end{equation}
Combining \eqref{e:q80}-\eqref{e:q86'} yields \eqref{e:q79'}-\eqref{e:q79}.
\vskip 0.2 true cm

{\bf Part 4.  The treatment  of $\|L_j(\p v)(t, \cdot)\|_{L^2(\R^3)}$ and $\|\Omega_{kj}(\p v)(t, \cdot)\|_{L^2(\R^3)}$ $(1\le k<j\le3)$}

\vskip 0.2 true cm
We now prove
\begin{equation}\label{e:q87}
\|L_j(\p v)\|_{L^2(\R^3)} \lesssim \begin{cases}t^{-\f{\mu}{2}}\left\|\nabla v_0\right\|_{Z, 1,  2}+t^{-\f{\mu}{2}}\left\| v_1\right\|_{Z, 1,  2}+t^{-\f{\mu}{2}}\left\|v_0+v_1\right\|_{Z, 1,  (\f{6}{3+\mu},2)}, &  \f{14}{5}\leq\mu<3,  \\ t^{-\f{\mu}{2}}\left\|\nabla v_0\right\|_{Z, 1,  2}+t^{-\f{\mu}{2}}\left\| v_1\right\|_{Z, 1,  2}+t^{\f32-\f{3}{1+\varepsilon_1}}\left\|v_0+v_1\right\|_{Z, 1,  (1+\varepsilon_1,2)},
&  \mu\geq3\end{cases}
\end{equation}
and
\begin{equation}\label{e:q88}
\|\Omega_{kj}(\p v)\|_{L^2(\R^3)}  \lesssim \begin{cases}t^{-\f{\mu}{2}}\left\|\nabla v_0\right\|_{Z, 1,  2}+t^{-\f{\mu}{2}}\left\| v_1\right\|_{Z, 1,  2}+t^{-\f{\mu}{2}}\left\|v_0+v_1\right\|_{Z, 1,  (\f{6}{3+\mu},2)}, &  \f{14}{5}\leq\mu<3,  \\ t^{-\f{\mu}{2}}\left\|\nabla v_0\right\|_{Z, 1,  2}+t^{-\f{\mu}{2}}\left\| v_1\right\|_{Z, 1,  2}+t^{\f32-\f{3}{1+\varepsilon_1}}\left\|v_0+v_1\right\|_{Z, 1,  (1+\varepsilon_1,2)},
&  \mu\geq3.\end{cases}
\end{equation}

To prove \eqref{e:q87} and \eqref{e:q88}, it is enough only to estimate $\|L_3(\p v)(t, \cdot)\|_{L^2(\R^3)}$ and  $\|\Omega_{13}(\p v)(t, \cdot)\|_{L^2(\R^3)} $.
Note that $|x|\lesssim t$ holds for $(t, x)\in \operatorname{supp}v$. Then
\begin{equation}\label{e:q89}
\begin{aligned}
\|L_3(\p v)(t, \cdot)\|_{L^2(\R^3)}
& \leq \|t|\xi_1||\xi|\hat{v}\|_{L^2(\R^3)}+\|t|\xi_1|\p_t\hat{v}\|_{L^2(\R^3)}+\|\f{x_1}{t}t\p_t\p v\|_{L^2(\R^3)}\\
& \lesssim \|t|\xi_1||\xi|\hat{v}\|_{L^2(\R^3)}+\|t|\xi_1|\p_t\hat{v}\|_{L^2(\R^3)}+\|t\p_t\p v\|_{L^2(\R^3)}
\end{aligned}
\end{equation}
and
\begin{equation}\label{e:q90}
\begin{aligned}
\|\Omega_{13}(\p v)(t, \cdot)\|_{L^2(\R^3)}
\lesssim \|t|\xi_3||\xi|\hat{v}\|_{L^2(\R^3)}+\|t|\xi_1||\xi|\hat{v}\|_{L^2(\R^3)}+\left\|t|\xi_3|\p_t\hat{v}\right\|_{L^2(\R^3)}
+\left\|t|\xi_1|\p_t\hat{v}\right\|_{L^2(\R^3)}.
\end{aligned}
\end{equation}
By comparing the estimates in \eqref{e:q1}, \eqref{e:q5}, \eqref{e:q10}, \eqref{e:q49} and \eqref{e:q51}
with those in  \eqref{e:q13}, \eqref{e:q17}, \eqref{e:q21}, \eqref{e:q54} and \eqref{e:q58'}, respectively,
then $\|t|\xi_j||\xi|\hat{v}\|_{L^2(\R^3)}$ can be treated
as for $\|t|\xi_j|\p_t\hat{v}\|_{L^2(\R^3)}$.

For the term $\|t\p_t\p v\|_{L^2(\R^3)}$, the following estimate has been obtained in the  proof procedure of
$\|L_0(\p v)(t, \cdot)\|_{L^2(\R^3)}$ in Part 3
\begin{equation}\label{e:q91}
\|t\p_t\p v\|_{L^2(\R^3)} \lesssim \begin{cases}t^{-\f{\mu}{2}}\left\|\nabla v_0\right\|_{Z, 1,  2}+t^{-\f{\mu}{2}}\left\| v_1\right\|_{Z, 1,  2}+t^{-\f{\mu}{2}}\left\|v_0+v_1\right\|_{Z, 1,  (\f{6}{3+\mu},2)}, &  \f{14}{5}\leq\mu<3,  \\ t^{-\f{\mu}{2}}\left\|\nabla v_0\right\|_{Z, 1,  2}+t^{-\f{\mu}{2}}\left\| v_1\right\|_{Z, 1,  2}+t^{\f32-\f{3}{1+\varepsilon_1}}\left\|v_0+v_1\right\|_{Z, 1,  (1+\varepsilon_1,2)},
&  \mu\geq3.\end{cases}
\end{equation}
Finally, we estimate the term $\|t|\xi_j|\p_t\hat{v}\|_{L^2(\R^3)}$. It follows from \eqref{e:q83} that for $\mu\geq3$,
\begin{equation}\label{e:q92}
\begin{aligned}
&\left\|t|\xi_j|\p_t\hat{v}\right\|_{L^2(\R^3)}\\&\leq\left\|t|\xi|\p_t\hat{v}\right\|_{L^2(\R^3)}\lesssim\begin{cases} t^{-\f{\mu}{2}}\left\|\nabla v_0\right\|_{Z, 1,  2}+t^{-\f{\mu}{2}}\left\| v_1\right\|_{Z, 1,  2}+t^{-\f{\mu}{2}}\left\|v_0+v_1\right\|_{Z, 1,  (\f{6}{3+\mu},2)},& \f{14}{5}\leq\mu<3\\
t^{-\f{\mu}{2}}\left\|\nabla v_0\right\|_{Z, 1,  2}+t^{-\f{\mu}{2}}\left\| v_1\right\|_{Z, 1,  2}+t^{\f32-\f{3}{1+\varepsilon_1}}\left\|v_0+v_1\right\|_{Z, 1,  (1+\varepsilon_1,2)},& \mu\geq3.
\end{cases}
\end{aligned}
\end{equation}
Therefore, \eqref{e:q87} and \eqref{e:q88} follow from \eqref{e:q91} and \eqref{e:q92} immediately.
Collecting all the parts above, we complete the proof of Lemma \ref{lem3}.
\end{proof}

\section{Time-decay estimates of solutions to  3-D inhomogeneous equation $\square w+\f{\mu}{t}\,\p_tw=F$}\label{sec4}
In this section, we establish the time-decay estimates for the 3-D linear equation
\begin{equation}\label{l1}
\left\{ \enspace
\begin{aligned}
&\square w+\f{\mu}{t}\,\p_tw=F, &&
t\geq 1,\\
&w(1,x)=0, \quad \partial_{t} w(1,x)=0, &&x\in\R^3,
\end{aligned}
\right.
\end{equation}
where $\mu\geq\f{14}{5}$. To apply  Duhamel's  principle,  we first consider the homogeneous problem starting from $\tau\ge 1$:
\begin{equation}\label{l2}
\left\{ \enspace
\begin{aligned}
&\square v+\f{\mu}{t}\,\p_tv=0, &&
t\geq \tau \geq1,\\
&v(\tau,x)=0, \quad \partial_{t} v(\tau,x)=v_1( x), &&x\in\R^3,
\end{aligned}
\right.
\end{equation}
where $v_1(x)\in C_0^{\infty}(\R^3)$.
\begin{lemma}\label{lem4}
For any $\varepsilon_1\in(0, 1)$ and $\mu\geq\f{14}{5}$,  the solution to \eqref{l2} satisfies the following decay estimates:
\begin{equation}\label{l3}
\|v\|_{Z, 1, 2} \lesssim \begin{cases}
t^{-\f{\mu}{2}}\tau^{\f{\mu}{2}}(\left\|v_1\right\|_{Z, 1,(\f65, 2)}+\left\|v_1\right\|_{Z, 1,(\f{6}{3+\mu}, 2)}),& \f{14}{5}\leq\mu<3,\\
t^{\f{3}{2}-\f{3}{1+\varepsilon_1}}\tau^\f32(\left\|v_1\right\|_{Z, 1,(\f65, 2)}+\left\|v_1\right\|_{Z, 1,(1+\varepsilon_1, 2)}),& \mu\geq3
\end{cases}
\end{equation}
and
\begin{equation}\label{l4}
\|\p v\|_{Z, 1, 2} \lesssim \begin{cases}
t^{-\f{\mu}{2}}\tau^{\f{\mu}{2}}(\left\|v_1\right\|_{Z, 1,2}+\left\|v_1\right\|_{Z, 1,(\f{6}{3+\mu}, 2)}),& \f{14}{5}\leq\mu<3,\\
t^{\f{3}{2}-\f{3}{1+\varepsilon_1}}\tau^\f32(\left\|v_1\right\|_{Z, 1,2}+\left\|v_1\right\|_{Z, 1,(1+\varepsilon_1, 2)}),& \mu\geq3,
\end{cases}.
\end{equation}
\end{lemma}
\begin{proof}
Although the proof follows the same procedure as in  Lemmas \ref{lem2}-\ref{lem3},
we still provide the main steps  for \eqref{l3} and \eqref{l4} due to  the presence of the initial time parameter $\tau$.

We first prove
$$
\|v\|_{Z, 1, 2} \lesssim \begin{cases}
t^{-\f{\mu}{2}}\tau^{\f{\mu}{2}}(\left\|v_1\right\|_{Z, 1,(\f65, 2)}+\left\|v_1\right\|_{Z, 1,(\f{6}{3+\mu}, 2)}),& \f{14}{5}\leq\mu<3,\\
t^{\f{3}{2}-\f{3}{1+\varepsilon_1}}\tau^\f32(\left\|v_1\right\|_{Z, 1,(\f65, 2)}+\left\|v_1\right\|_{Z, 1,(1+\varepsilon_1, 2)}),& \mu\geq3.
\end{cases}
$$
To this end,  we decompose the frequency space $\R^3_{\xi}$ into three regions:
\begin{equation}\label{l5}
D'_1=\{\xi: \tau|\xi| \leq 1 \leq t|\xi|\}, \quad D'_2=\{\xi: t|\xi|\geq\tau|\xi|\ge 1\}, \quad D'_3=\{\xi: \tau|\xi|\leq t|\xi| \leq 1\}.
\end{equation}

For the regions $D'_1$ and $D'_3$, by \eqref{equ:q9-3}, \eqref{equ:q9-5'}, Lemma \ref{lem1} (i) and (iii), it holds
that
\begin{equation}\label{l6}
\left|\Psi_1(t, \tau, \xi)\right|  \lesssim t^{-\frac{\mu}{2}}|\xi|^{-\frac{\mu}{2}}\tau
\quad\text{and}\quad
\begin{aligned}
\left|\Psi_1(t, \tau, \xi)\right|  \lesssim t^{2\rho}\tau^{1-2\rho}+\tau\lesssim \tau,
\end{aligned}
\end{equation}
respectively. For the region $D'_2$, combining \eqref{equ:q6} with Lemma \ref{lem1} (ii) yields
\begin{equation}\label{l8}
 \left|\Psi_1(t, \tau, \xi)\right|  \lesssim t^{-\frac{\mu}{2}}\tau^{\frac{\mu}{2}}|\xi|^{-1}.
\end{equation}
From \eqref{l6}-\eqref{l8} and \eqref{equ:q4} with $\hat{v}_0 = 0$, we deduce that for  $t\geq \tau\geq1$,
\begin{equation}\label{l9'}
\begin{aligned}
\|\hat{v}\|_{L^2(\R^3)} & \lesssim\big\|t^{-\frac{\mu}{2}} |\xi|^{-\frac{\mu}{2}}\tau \hat{v}_1(\xi)\big\|_{L^2(D'_1)}+\left\|\tau \hat{v}_1(\xi)\right\|_{L^2(D'_3)} +\big\|t^{-\frac{\mu}{2}}\tau^{\frac{\mu}{2}} |\xi|^{-1} \hat{v}_1(\xi)\big\|_{L^2(D'_2)} \\
& \lesssim \tau\left\|(1+t|\xi|)^{-\f{\mu}{2}} \hat{v}_1\right\|_{L^2}+\tau\left\|(1+t|\xi|)^{-\f{\mu}{2}} \hat{v}_1\right\|_{L^2}+t^{-\frac{\mu}{2}} \tau^{\frac{\mu}{2}} \cdot \tau\left\|(1+\tau|\xi|)^{-1} \hat{v}_1\right\|_{L^2} \\
& \lesssim \tau t^{-\f{\mu}{2}}\left\|v_1\right\|_{\left(1+\varepsilon_1, 2\right)}+t^{-\f{\mu}{2}} \tau^{\f{\mu}{2}}\left\|v_1\right\|_{\left(\f65, 2\right)} \\
& \lesssim  t^{-\f{\mu}{2}} \tau^{\f{\mu}{2}}(\left\|v_1\right\|_{\left(1+\varepsilon_1, 2\right)}+\left\|v_1\right\|_{\left(\f65, 2\right)})\quad\text{for}\,\f{14}{5}\leq\mu<3
\end{aligned}
\end{equation}
and
\begin{equation}\label{l9}
\begin{aligned}
\|\hat{v}\|_{L^2(\R^3)} & \lesssim\big\|t^{-\frac{\mu}{2}} |\xi|^{-\frac{\mu}{2}}\tau \hat{v}_1(\xi)\big\|_{L^2(D'_1)}+\left\|\tau \hat{v}_1(\xi)\right\|_{L^2(D'_3)} +\big\|t^{-\frac{\mu}{2}}\tau^{\frac{\mu}{2}} |\xi|^{-1} \hat{v}_1(\xi)\big\|_{L^2(D'_2)} \\
& \lesssim \tau\left\|(1+t|\xi|)^{-\f32} \hat{v}_1\right\|_{L^2}+\tau\left\|(1+t|\xi|)^{-\f32} \hat{v}_1\right\|_{L^2}+t^{-\frac{\mu}{2}} \tau^{\frac{\mu}{2}} \cdot \tau\left\|(1+\tau|\xi|)^{-1} \hat{v}_1\right\|_{L^2} \\
& \lesssim \tau t^{\f32-\frac{3}{1+\varepsilon_1}}\left\|v_1\right\|_{\left(1+\varepsilon_1, 2\right)}+t^{-\f{\mu}{2}} \tau^{\f{\mu}{2}}\left\|v_1\right\|_{\left(\f65, 2\right)} \\
& \lesssim t^{\f32-\frac{3}{1+\varepsilon_1}}t^{-\f{\mu}{2}-\f32+\frac{3}{1+\varepsilon_1}} \tau^{\f{\mu}{2}}\left\|v_1\right\|_{\left(\f65, 2\right)}+ t^{\f32-\frac{3}{1+\varepsilon_1}} \tau\left\|v_1\right\|_{\left(1+\varepsilon_1, 2\right)}\\
& \lesssim t^{\f32-\frac{3}{1+\varepsilon_1}} \tau^{\f32}(\left\|v_1\right\|_{\left(\f65, 2\right)}+ \left\|v_1\right\|_{\left(1+\varepsilon_1, 2\right)})\quad\text{for}\,\mu\geq3.
\end{aligned}
\end{equation}
Similarly,
\begin{equation}\label{l10'}
\begin{aligned}
\|t|\xi_j|\hat{v}\|_{L^2} & \lesssim\big\|t^{-\frac{\mu}{2}} |\xi|^{-\frac{\mu}{2}}\tau t|\xi_j| \hat{v}_1(\xi)\big\|_{L^2(D'_1)}+\left\|\tau t|\xi_j| \hat{v}_1(\xi)\right\|_{L^2(D'_3)} +\big\|t^{-\frac{\mu}{2}}\tau^{\frac{\mu}{2}} |\xi|^{-1}t|\xi_j| \hat{v}_1(\xi)\big\|_{L^2(D'_2)} \\
& \lesssim t^{-\f{\mu}{2}} \tau\left\|v_1\right\|_{Z, 1, \left(\f{6}{3+\mu}, 2\right)}+t^{-\f{\mu}{2}} \tau^{\f{\mu}{2}}\left\|v_1\right\|_{Z, 1, \left(\f65, 2\right)} \\
& \lesssim t^{-\f{\mu}{2}} \tau^{\f{\mu}{2}}(\left\|v_1\right\|_{Z, 1, \left(\f65, 2\right)}+\left\|v_1\right\|_{Z, 1, \left(\f{6}{3+\mu}, 2\right)})\quad\text{for}\,\f{14}{5}\leq\mu<3
\end{aligned}
\end{equation}
and for $\mu\geq3$,
\begin{equation}\label{l10}
\begin{aligned}
\|t|\xi_j|\hat{v}\|_{L^2} & \lesssim\big\|t^{-\frac{\mu}{2}} |\xi|^{-\frac{\mu}{2}}\tau t|\xi_j| \hat{v}_1(\xi)\big\|_{L^2(D'_1)}+\left\|\tau t|\xi_j| \hat{v}_1(\xi)\right\|_{L^2(D'_3)} +\big\|t^{-\frac{\mu}{2}}\tau^{\frac{\mu}{2}} |\xi|^{-1}t|\xi_j| \hat{v}_1(\xi)\big\|_{L^2(D'_2)} \\
& \lesssim t^{\f32-\frac{3}{1+\varepsilon_1}} \tau\left\|v_1\right\|_{Z, 1, \left(1+\varepsilon_1, 2\right)}+\tau t^{\f32-\frac{3}{1+\varepsilon_1}}\left\|v_1\right\|_{Z, 1, \left(1+\varepsilon_1, 2\right)}+t^{-\f{\mu}{2}} \tau^{\f{\mu}{2}}\left\|v_1\right\|_{Z, 1, \left(\f65, 2\right)} \\
&  \lesssim t^{\f32-\frac{3}{1+\varepsilon_1}} \tau^{-\f32+\frac{3}{1+\varepsilon_1}}\left\|v_1\right\|_{Z 1, \left(\f65, 2\right)}+ t^{\f32-\frac{3}{1+\varepsilon_1}} \tau\left\|v_1\right\|_{Z, 1, \left(1+\varepsilon_1, 2\right)}\\
& \lesssim t^{\f32-\frac{3}{1+\varepsilon_1}} \tau^{\f32}(\left\|v_1\right\|_{Z, 1, \left(\f65, 2\right)}+\left\|v_1\right\|_{Z, 1, \left(1+\varepsilon_1, 2\right)}).
\end{aligned}
\end{equation}
Hence we arrive at
\begin{equation}\label{l11}
\||\xi_j|\hat{v}\|_{L^2(\R^3)}=t^{-1}\|t|\xi_j|\p_t\hat{v}\|_{L^2}
 \lesssim\begin{cases}
 t^{-\f{\mu}{2}-1} \tau^{\f{\mu}{2}}(\left\|v_1\right\|_{Z, 1, \left(\f65, 2\right)}+\left\|v_1\right\|_{Z, 1, \left(\f{6}{3+\mu}, 2\right)}),&\f{14}{5}\leq\mu<3,\\
 t^{\f12-\frac{3}{1+\varepsilon_1}} \tau^{\f32}(\left\|v_1\right\|_{Z, 1, \left(\f65, 2\right)}+\left\|v_1\right\|_{Z, 1, \left(1+\varepsilon_1, 2\right)}),&\mu\geq3.
 \end{cases}
\end{equation}
In addition, for the regions $D'_1$ and $D'_3$,  combining \eqref{equ:q9-4}, \eqref{l-3} with Lemma \ref{lem1} (i) and (iii), one can
obtain that for $\mu\geq\f{14}{5}$,
\begin{equation}\label{l12}
\left|\p_t\Psi_1(t, \tau, \xi)\right|  \lesssim t^{-\frac{\mu}{2}}|\xi|^{\frac{\mu}{2}}\tau^{\mu}
\end{equation}
and
\begin{equation}\label{l13}
\begin{aligned}
\left|\p_t\Psi_1(t, \tau, \xi)\right|  \lesssim t^{-1}\tau.
\end{aligned}
\end{equation}
For the region $D'_2$,  we have from \eqref{equ:q9} and Lemma \ref{lem1} (ii) that
\begin{equation}\label{l14}
 \left|\p_t\Psi_1(t, \tau, \xi)\right|  \lesssim t^{-\frac{\mu}{2}}\tau^{\frac{\mu}{2}}.
\end{equation}
From \eqref{l12}--\eqref{l14}, it follows that for $\f{14}{5}\leq\mu<3$,
\begin{equation}\label{l15'}
\begin{aligned}
\|t\p_t\hat{v}\|_{L^2(\R^3)} & \lesssim\big\|t^{1-\frac{\mu}{2}} |\xi|^{\frac{\mu}{2}}\tau^{\mu} \hat{v}_1(\xi)\big\|_{L^2(D'_1)}+\left\| \tau \hat{v}_1(\xi)\right\|_{L^2(D'_3)} +\big\|t^{-\frac{\mu}{2}}\tau^{\frac{\mu}{2}} t \hat{v}_1(\xi)\big\|_{L^2(D'_2)} \\
& \lesssim t^{-\frac{\mu}{2}}\tau^{\frac{\mu}{2}}\big\| (1+|\xi|^2)^{-\f12}t|\xi|  \hat{v}_1(\xi)\big\|_{L^2}+t^{-\frac{\mu}{2}}\tau^{\frac{\mu}{2}}\cdot\tau\big\| (1+\tau|\xi|)^{-1}t|\xi|  \hat{v}_1(\xi)\big\|_{L^2}\\ &\quad +\tau\left\| (1+t|\xi|)^{-\f{\mu}{2}} \hat{v}_1(\xi)\right\|_{L^2}\\
& \lesssim t^{-\f{\mu}{2}} \tau^{\f{\mu}{2}} \left\|t\p_1v_1\right\|_{\left(\f65, 2\right)}+t^{-\f{\mu}{2}} \tau^{\f{\mu}{2}} \left\|t\p_2v_1\right\|_{\left(\f65, 2\right)}+t^{-\f{\mu}{2}} \tau^{\f{\mu}{2}} \left\|t\p_3v_1\right\|_{\left(\f65, 2\right)}\\
&\quad+\tau t^{-\f{\mu}{2}}\left\|v_1\right\|_{\left(\f{6}{3+\mu}, 2\right)} \\
& \lesssim t^{-\f{\mu}{2}} \tau^{\f{\mu}{2}}(\left\|v_1\right\|_{Z, 1, \left(\f65, 2\right)}+\left\|v_1\right\|_{Z, 1, \left(\f{6}{3+\mu}, 2\right)}),
\end{aligned}
\end{equation}
and for $\mu\geq3$,
\begin{equation*}\label{l15}
\begin{aligned}
\|t\p_t\hat{v}\|_{L^2(\R^3)} & \lesssim\big\|t^{1-\frac{\mu}{2}} |\xi|^{\frac{\mu}{2}}\tau^{\mu} \hat{v}_1(\xi)\big\|_{L^2(D'_1)}+\left\| \tau \hat{v}_1(\xi)\right\|_{L^2(D'_3)} +\big\|t^{-\frac{\mu}{2}}\tau^{\frac{\mu}{2}} t \hat{v}_1(\xi)\big\|_{L^2(D'_2)} \\
& \lesssim t^{-\frac{\mu}{2}}\tau^{\frac{\mu}{2}}\big\| (1+|\xi|^2)^{-\f12}t|\xi|  \hat{v}_1(\xi)\big\|_{L^2}+\tau\left\| (1+t|\xi|)^{-\f32} \hat{v}_1(\xi)\right\|_{L^2}\\ &\quad+t^{-\frac{\mu}{2}}\tau^{\frac{\mu}{2}}\cdot\tau\big\| (1+\tau|\xi|)^{-1}t|\xi|  \hat{v}_1(\xi)\big\|_{L^2} \\
& \lesssim t^{-\f{\mu}{2}} \tau^{\f{\mu}{2}} \left\|t\p_1v_1\right\|_{\left(\f65, 2\right)}+t^{-\f{\mu}{2}} \tau^{\f{\mu}{2}} \left\|t\p_2v_1\right\|_{\left(\f65, 2\right)}+t^{-\f{\mu}{2}} \tau^{\f{\mu}{2}} \left\|t\p_3v_1\right\|_{\left(\f65, 2\right)}\\
&\quad+\tau t^{\f32-\frac{3}{1+\varepsilon_1}}\left\|v_1\right\|_{\left(1+\varepsilon_1, 2\right)} \\
\end{aligned}
\end{equation*}

\begin{equation}\label{l15}
\begin{aligned}
& \lesssim t^{\f32-\frac{3}{1+\varepsilon_1}} \tau^{-\f32+\frac{3}{1+\varepsilon_1}}\left\|v_1\right\|_{Z, 1, \left(\f65, 2\right)}+ t^{\f32-\frac{3}{1+\varepsilon_1}} \tau\left\|v_1\right\|_{\left(1+\varepsilon_1, 2\right)}\\
& \lesssim t^{\f32-\frac{3}{1+\varepsilon_1}} \tau^{\f32}(\left\|v_1\right\|_{Z, 1, \left(\f65, 2\right)}+\left\|v_1\right\|_{Z, 1, \left(1+\varepsilon_1, 2\right)}).
\end{aligned}
\end{equation}
Then
\begin{equation}\label{l16}
\|\p_t\hat{v}\|_{L^2(\R^3)}=t^{-1}\|t\p_t\hat{v}\|_{L^2(\R^3)}
 \lesssim\begin{cases}
 t^{-\f{\mu}{2}-1} \tau^{\f{\mu}{2}}(\left\|v_1\right\|_{Z, 1, \left(\f65, 2\right)}+\left\|v_1\right\|_{Z, 1, \left(\f{6}{3+\mu}, 2\right)}),&\f{14}{5}\leq\mu<3,\\
 t^{\f12-\frac{3}{1+\varepsilon_1}} \tau^{\f32}(\left\|v_1\right\|_{Z, 1, \left(\f65, 2\right)}+\left\|v_1\right\|_{Z, 1, \left(1+\varepsilon_1, 2\right)}),&\mu\geq3.
 \end{cases}
\end{equation}
Analogously to the proof of \eqref{e:q33}, \eqref{e:q37} and \eqref{e:q63}, we have that
for $j=1,2,3$ and $\tau|\xi|\leq 1$,
\begin{equation}\label{l17}
\begin{aligned}
|\partial_{\xi_j}& \Psi_1(t, \tau, \xi)|
 \lesssim  \frac{|\xi_j|}{|\xi|} t^\rho\tau^{2-\rho} \left| J_{1-\rho}(\tau|\xi|) Y_{-\rho}(t|\xi|) \right|
+ \frac{|\xi_j|}{|\xi|^2} t^\rho\tau^{1-\rho} \left| J_{-\rho}(\tau|\xi|) Y_{-\rho}(t|\xi|) \right| \\
& \quad + t^{\rho+1}\tau^{1-\rho}\f{|\xi_j|}{|\xi|} \left| Y_{1-\rho}(t|\xi|)J_{-\rho}(\tau|\xi|) \right|
+ t^\rho\tau^{1-\rho} \frac{|\xi_j|}{|\xi|^2} \left| Y_{-\rho}(t|\xi|)J_{-\rho}(\tau|\xi|) \right|\\
& \quad + t^{\rho+1}\tau^{1-\rho}\f{|\xi_j|}{|\xi|} \left| J_{1-\rho}(t|\xi|) Y_{-\rho}(\tau|\xi|) \right|
+ t^\rho\tau^{1-\rho} \frac{|\xi_j|}{|\xi|^2} \left| J_{-\rho}(t|\xi|) Y_{-\rho}(\tau|\xi|) \right| \\
& \quad + t^\rho\tau^{2-\rho} \f{|\xi_j|}{|\xi|} \left| Y_{1-\rho}(\tau|\xi|) J_{-\rho}(t|\xi|) \right|
+ t^\rho\tau^{1-\rho} \frac{|\xi_j|}{|\xi|^2} \left| Y_{-\rho}(\tau|\xi|) J_{-\rho}(t|\xi|) \right|, \\&\qquad\qquad\qquad\qquad\qquad\qquad\qquad\qquad\qquad\qquad\qquad\qquad\qquad\text{if $\rho$ is an integer},
\end{aligned}
\end{equation}
\begin{equation}\label{l18}
\begin{aligned}
|\partial_{\xi_j}&\Psi_1(t, \tau, \xi)|
\lesssim t^\rho \tau^{2-\rho}\frac{\xi_j}{|\xi|} |J_{-\rho-1}(\tau|\xi|) J_\rho(t|\xi|)|+t^{\rho+1} \tau^{1-\rho}\frac{\xi_j}{|\xi|} |J_{\rho-1}(t|\xi|) J_{-\rho}(\tau|\xi|)| \\
&\quad -t^\rho \tau^{2-\rho}\frac{\xi_j}{|\xi|} |J_{\rho-1}(\tau|\xi|) J_{-\rho}(t|\xi|)|
-t^{\rho+1} \tau^{1-\rho}\frac{\xi_j}{|\xi|} |J_{-\rho-1}(t|\xi|) J_\rho(\tau|\xi|)|\\
& \quad+t^\rho \tau^{1-\rho}\frac{\xi_j}{|\xi|^2} |J_{-\rho}(\tau|\xi|) J_\rho(t|\xi|)| +t^\rho \tau^{1-\rho}\frac{\xi_j\rho}{|\xi|^2} |J_\rho(t|\xi|) J_{-\rho}(\tau|\xi|)|\\
&\quad +t^\rho \tau^{1-\rho}\frac{\xi_j\rho}{|\xi|^2} |J_\rho(\tau|\xi|) J_{-\rho}(t|\xi|)|+t^\rho \tau^{1-\rho}\frac{\xi_j\rho}{|\xi|^2} |J_{-\rho}(t|\xi|) J_\rho(\tau|\xi|)|,\\&\qquad\qquad\qquad\qquad\qquad\qquad\qquad\qquad\qquad\qquad\qquad\qquad\quad\text{if $\rho$ is not an integer.}
\end{aligned}
\end{equation}
Meanwhile, for $\tau|\xi|\geq 1$,
\begin{equation}\label{l19}
\begin{aligned}
|\partial_{\xi_j}&\Psi_1(t, \tau, \xi)|
\lesssim  t^{\rho+1}\tau^{1-\rho} \f{|\xi_j|}{|\xi|} \left| H_{\rho-1}^+(t|\xi|) H_{\rho}^-(\tau|\xi|) \right|
+ t^\rho\tau^{1-\rho} \frac{|\xi_j|}{|\xi|^2} \left| H_\rho^+(t|\xi|) H_{\rho}^-(\tau|\xi|) \right| \\
&\quad + t^\rho\tau^{2-\rho} \f{|\xi_j|}{|\xi|} \left| H_{\rho-1}^-(\tau|\xi|) H_\rho^+(t|\xi|) \right|
+ t^\rho\tau^{1-\rho} \frac{|\xi_j|}{|\xi|^2} \left| H_{\rho}^-(\tau|\xi|) H_\rho^+(t|\xi|) \right| \\
&\quad + t^\rho\tau^{2-\rho} \f{|\xi_j|}{|\xi|} \left| H_{\rho-1}^+(\tau|\xi|) H_\rho^-(t|\xi|) \right|
+ t^\rho\tau^{1-\rho} \frac{|\xi_j|}{|\xi|^2} \left| H_{\rho}^+(\tau|\xi|) H_\rho^-(t|\xi|) \right| \\
&\quad + t^{\rho+1}\tau^{1-\rho} \f{|\xi_j|}{|\xi|} \left| H_{\rho-1}^-(t|\xi|) H_{\rho}^+(\tau|\xi|) \right|
+ t^\rho\tau^{1-\rho} \frac{|\xi_j|}{|\xi|^2} \left| H_\rho^-(t|\xi|) H_{\rho}^+(\tau|\xi|) \right|.
\end{aligned}
\end{equation}
Proceeding as in \eqref{e:q38} with $\hat{v}_0=0$, we deduce from \eqref{equ:q6}, \eqref{l19} and Lemma \ref{lem1}(ii) that for $\f{14}{5}\leq\mu<3$,
\begin{equation}\label{l20'}
\begin{aligned}
&\big\|\xi_1\p_{\xi_1}\hat{v}+\xi_2\p_{\xi_2}\hat{v}+\xi_3\p_{\xi_3}\hat{v}\big\|_{L^2(D'_2)}
\\
&\lesssim \big\|t^{1-\frac{\mu}{2}}\tau^{\f{\mu}{2}}\hat{v}_{1}\big\|_{L^{2}(\tau|\xi|\geq1)}+ \big\|t^{-\f{\mu}{2}}\tau^{\f{\mu}{2}}|\xi|^{-1}(\ds\sum_{k=1}^3\xi_k\p_{\xi_k})\hat{v}_{1}\big\|_{L^{2}(\tau|\xi|\geq1)}\\
& \lesssim t^{-\frac{\mu}{2}}\tau^{\frac{\mu}{2}}\cdot\tau\big\| (1+\tau|\xi|)^{-1}t|\xi|  \hat{v}_1(\xi)\big\|_{L^2}+t^{-\frac{\mu}{2}}\tau^{\frac{\mu}{2}}\cdot\tau\big\| (1+\tau|\xi|)^{-1}(\ds\sum_{k=1}^3\xi_k\p_{\xi_k})  \hat{v}_1(\xi)\big\|_{L^2}\\
&\lesssim  t^{-\f{\mu}{2}}\tau^{\f{\mu}{2}}\left\|v_1\right\|_{Z, 1,(\f65, 2)},
\end{aligned}
\end{equation}
and for $\mu\geq3$,
\begin{equation}\label{l20}
\begin{aligned}
&\big\|\xi_1\p_{\xi_1}\hat{v}+\xi_2\p_{\xi_2}\hat{v}+\xi_3\p_{\xi_3}\hat{v}\big\|_{L^2(D'_2)}
\\
&\lesssim \big\|t^{1-\frac{\mu}{2}}\tau^{\f{\mu}{2}}\hat{v}_{1}\big\|_{L^{2}(\tau|\xi|\geq1)}+ \big\|t^{-\f{\mu}{2}}\tau^{\f{\mu}{2}}|\xi|^{-1}(\sum_{k=1}^3\xi_k\p_{\xi_k})\hat{v}_{1}\big\|_{L^{2}(\tau|\xi|\geq1)}\\
&\lesssim  t^{-\f{\mu}{2}}\tau^{\f{\mu}{2}}\left\|v_1\right\|_{Z, 1,(\f65, 2)}\\
&\lesssim t^{\f32-\frac{3}{1+\varepsilon_1}}t^{-\f{\mu}{2}-\f32+\frac{3}{1+\varepsilon_1}} \tau^{\f{\mu}{2}}\left\|v_1\right\|_{Z, 1,(\f65, 2)}\\
&\lesssim  t^{\f32-\frac{3}{1+\varepsilon_1}} \tau^{\f32}\left\|v_1\right\|_{Z, 1,(\f65, 2)}.
\end{aligned}
\end{equation}
Similarly, from \eqref{l17}-\eqref{l18}, \eqref{equ:q9-3}, \eqref{equ:q9-5'}, Lemma \ref{lem1} (i) and (iii), it follows that
\begin{equation}\label{l21}
\big\|\xi_1\p_{\xi_1}\hat{v}+\xi_2\p_{\xi_2}\hat{v}+\xi_3\p_{\xi_3}\hat{v}\big\|_{L^2(D'_1\bigcup D'_3)}
\lesssim \begin{cases}
t^{-\f{\mu}{2}}\tau^{\f{\mu}{2}}\left\|v_1\right\|_{Z, 1,(\f{6}{3+\mu}, 2)},&\f{14}{5}\leq\mu<3,\\
t^{\f{3}{2}-\f{3}{1+\varepsilon_1}}\tau^\f32\left\|v_1\right\|_{Z, 1,(1+\varepsilon_1, 2)},&\mu\geq3.
\end{cases}
\end{equation}
Then \eqref{l9'}-\eqref{l9}, \eqref{l15'}-\eqref{l15} and \eqref{l20'}-\eqref{l21} yield that for $\f{14}{5}\leq\mu<3$,
$$
\begin{aligned}
\|\widehat{L_0v}\|_{L^2}&\lesssim \|\hat{v}\|_{L^2(\R^3)}+\|t\p_t\hat{v}\|_{L^2(\R^3)}+\big\|\sum_{k=1}^3\xi_k\p_{\xi_k}\hat{v}\big\|_{L^2(\R^3)}\\
&\lesssim t^{-\f{\mu}{2}}\tau^{\f{\mu}{2}}(\left\|v_1\right\|_{Z, 1,(\f65, 2)}+\left\|v_1\right\|_{Z, 1,(\f{6}{3+\mu}, 2)}),
\end{aligned}
$$
and for $\mu\geq3$,
$$
\begin{aligned}
\|\widehat{L_0v}\|_{L^2}
&\lesssim t^{\f{3}{2}-\f{3}{1+\varepsilon_1}}\tau^\f32(\left\|v_1\right\|_{Z, 1,(\f65, 2)}+\left\|v_1\right\|_{Z, 1,(1+\varepsilon_1, 2)}).
\end{aligned}
$$
Meanwhile, it follows from \eqref{l11} and \eqref{l16} that
$$
\|\widehat{\p v}\|_{L^2}\lesssim \sum_{k=1}^3\||\xi_k|\hat{v}\|_{L^2}+\|\p_t\hat{v}\|_{L^2}\lesssim\begin{cases} t^{-\f{\mu}{2}-1}\tau^{\f{\mu}{2}}(\left\|v_1\right\|_{Z, 1,(\f65, 2)}+\left\|v_1\right\|_{Z, 1,(\f{6}{3+\mu}, 2)}),&\f{14}{5}\leq\mu<3,\\
t^{\f{1}{2}-\f{3}{1+\varepsilon_1}}\tau^\f32(\left\|v_1\right\|_{Z, 1,(\f65, 2)}+\left\|v_1\right\|_{Z, 1,(1+\varepsilon_1, 2)}),&\mu\geq3.
\end{cases}
$$
In addition, as in \eqref{e:q40} and \eqref{e:q46''}, from \eqref{l10'}-\eqref{l10} and \eqref{l15'}-\eqref{l15}, we deduce that
$$
\begin{aligned}
\sum_{j=1}^3\|\widehat{L_jv}\|_{L^2(\R^3)}&\lesssim \sum_{j=1}^3\|t|\xi_j|\hat{v}\|_{L^2(\R^3)}+\|t\p_t\hat{v}\|_{L^2(\R^3)}\\
&\lesssim \begin{cases} t^{-\f{\mu}{2}}\tau^{\f{\mu}{2}}(\left\|v_1\right\|_{Z, 1,(\f65, 2)}+\left\|v_1\right\|_{Z, 1,(\f{6}{3+\mu}, 2)}),&\f{14}{5}\leq\mu<3,\\
t^{\f{3}{2}-\f{3}{1+\varepsilon_1}}\tau^\f32(\left\|v_1\right\|_{Z, 1,(\f65, 2)}+\left\|v_1\right\|_{Z, 1,(1+\varepsilon_1, 2)}),&\mu\geq3
\end{cases}
\end{aligned}
$$
and
$$
\begin{aligned}
\sum_{1\leq k<j\leq3}\|\widehat{\Omega_{kj}v}\|_{L^2(\R^3)}&\lesssim \sum_{j=1}^3\left\|t|\xi_j|\hat{v}\right\|_{L^2(\R^3)}\\
&\lesssim \begin{cases} t^{-\f{\mu}{2}}\tau^{\f{\mu}{2}}(\left\|v_1\right\|_{Z, 1,(\f65, 2)}+\left\|v_1\right\|_{Z, 1,(\f{6}{3+\mu}, 2)}),&\f{14}{5}\leq\mu<3,\\
t^{\f{3}{2}-\f{3}{1+\varepsilon_1}}\tau^\f32(\left\|v_1\right\|_{Z, 1,(\f65, 2)}+\left\|v_1\right\|_{Z, 1,(1+\varepsilon_1, 2)}),&\mu\geq3.
\end{cases}
\end{aligned}
$$
Collecting the above estimates gives \eqref{l3}.

Next we turn to  $\|\p v\|_{Z, 1, 2}$. As in the proof
of \eqref{e:q71}, it remains to consider $\|t|\xi|\p_t\hat{v}\|_{L^2}$, $\||\xi_j||\xi|\hat{v}\|_{L^2}(j=1, 2, 3)$ and $\||\xi|\ds\sum_{k=1}^3\xi_k\p_{\xi_k}\hat{v}\|_{L^2}$.

By \eqref{l12}-\eqref{l14}, it follows that for $\f{14}{5}\leq\mu<3$,
\begin{equation}\label{l22'}
\begin{aligned}
\|t|\xi|\p_t\hat{v}\|_{L^2(\R^3)} & \lesssim\big\|t^{1-\frac{\mu}{2}} |\xi|^{\frac{\mu}{2}}\tau^{\mu}|\xi| \hat{v}_1(\xi)\big\|_{L^2(D'_1)}+\left\| \tau|\xi| \hat{v}_1(\xi)\right\|_{L^2(D'_3)} +\big\|t^{-\frac{\mu}{2}}\tau^{\frac{\mu}{2}} t|\xi| \hat{v}_1(\xi)\big\|_{L^2(D'_2)} \\
& \lesssim t^{-\frac{\mu}{2}}\tau^{\frac{\mu}{2}}\big\| t|\xi|  \hat{v}_1(\xi)\big\|_{L^2}+\tau t^{-1}\big\| (1+t|\xi|)^{-\f{\mu}{2}} \hat{v}_1(\xi)\big\|_{L^2}\\
& \lesssim t^{-\frac{\mu}{2}}\tau^{\frac{\mu}{2}}\left\|v_1\right\|_{Z, 1, 2}+ t^{-\f{\mu}{2}-1} \tau \left\|v_1\right\|_{(\f{6}{3+\mu}, 2)}\\
& \lesssim t^{-\f{\mu}{2}}\tau^{\f{\mu}{2}}(\left\|v_1\right\|_{Z, 1,2}+\left\|v_1\right\|_{Z, 1,(\f{6}{3+\mu}, 2)}),
\end{aligned}
\end{equation}
and for $\mu\geq3$,
\begin{equation}\label{l22}
\begin{aligned}
\|t|\xi|\p_t\hat{v}\|_{L^2(\R^3)} & \lesssim\big\|t^{1-\frac{\mu}{2}} |\xi|^{\frac{\mu}{2}}\tau^{\mu}|\xi| \hat{v}_1(\xi)\big\|_{L^2(D'_1)}+\left\| \tau|\xi| \hat{v}_1(\xi)\right\|_{L^2(D'_3)} +\big\|t^{-\frac{\mu}{2}}\tau^{\frac{\mu}{2}} t|\xi| \hat{v}_1(\xi)\big\|_{L^2(D'_2)} \\
& \lesssim t^{-\frac{\mu}{2}}\tau^{\frac{\mu}{2}}\big\| t|\xi|  \hat{v}_1(\xi)\big\|_{L^2}+\tau t^{-1}\big\| (1+t|\xi|)^{-\f32} \hat{v}_1(\xi)\big\|_{L^2}\\
& \lesssim t^{\f32-\frac{3}{1+\varepsilon_1}} \tau^{-\f32+\frac{3}{1+\varepsilon_1}}\left\|v_1\right\|_{Z, 1, 2}+ t^{\f12-\frac{3}{1+\varepsilon_1}} \tau\left\|v_1\right\|_{(1+\varepsilon_1, 2)}\\
& \lesssim t^{\f32-\frac{3}{1+\varepsilon_1}} \tau^{\f32}(\left\|v_1\right\|_{Z, 1, 2}+\left\|v_1\right\|_{Z, 1, \left(1+\varepsilon_1, 2\right)}).
\end{aligned}
\end{equation}
It follows from \eqref{l6}-\eqref{l8} that for $j=1, 2, 3$ and $\f{14}{5}\leq\mu<3$,
\begin{equation}\label{l23'}
\begin{aligned}
\||\xi||\xi_j|\hat{v}\|_{L^2(\R^3)}& \lesssim\big\|t^{-\frac{\mu}{2}} |\xi|^{-\frac{\mu}{2}}\tau|\xi||\xi_j| \hat{v}_1(\xi)\big\|_{L^2(D'_1)}+\left\|\tau|\xi||\xi_j| \hat{v}_1(\xi)\right\|_{L^2(D'_3)}\\
&\quad +\big\|t^{-\frac{\mu}{2}}\tau^{\frac{\mu}{2}} |\xi|^{-1}|\xi||\xi_j| \hat{v}_1(\xi)\big\|_{L^2(D'_2)} \\
& \lesssim \left\|(1+t|\xi|)^{-\f{\mu}{2}}|\xi_j| \hat{v}_1\right\|_{L^2}+t^{-1}\left\|(1+t|\xi|)^{-\f{\mu}{2}} \hat{v}_1\right\|_{L^2}\\
&\quad +t^{-\frac{\mu}{2}} \tau^{\frac{\mu}{2}} \left\||\xi_j|\hat{v}_1\right\|_{L^2} \\
& \lesssim t^{-\f{\mu}{2}} \left\|v_1\right\|_{Z, 1, (\f{6}{3+\mu})}+ t^{-\f{\mu}{2}-1}\left\|v_1\right\|_{(\f{6}{3+\mu}, 2)}+t^{-\f{\mu}{2}} \tau^{\f{\mu}{2}}\left\|v_1\right\|_{Z, 1, 2} \\
& \lesssim t^{-\f{\mu}{2}}\tau^{\f{\mu}{2}}(\left\|v_1\right\|_{Z, 1,2}+\left\|v_1\right\|_{Z, 1,(\f{6}{3+\mu}, 2)}),
\end{aligned}
\end{equation}
and for $\mu\geq3$,
\begin{equation*}\label{l23}
\begin{aligned}
\||\xi||\xi_j|\hat{v}\|_{L^2(\R^3)}& \lesssim\big\|t^{-\frac{\mu}{2}} |\xi|^{-\frac{\mu}{2}}\tau|\xi||\xi_j| \hat{v}_1(\xi)\big\|_{L^2(D'_1)}+\left\|\tau|\xi||\xi_j| \hat{v}_1(\xi)\right\|_{L^2(D'_3)}\\
&\quad +\big\|t^{-\frac{\mu}{2}}\tau^{\frac{\mu}{2}} |\xi|^{-1}|\xi||\xi_j| \hat{v}_1(\xi)\big\|_{L^2(D'_2)} \\
\end{aligned}
\end{equation*}

\begin{equation}\label{l23}
\begin{aligned}
& \lesssim \big\|(1+t|\xi|)^{-\f32}|\xi_j| \hat{v}_1\big\|_{L^2}+t^{-1}\big\|(1+t|\xi|)^{-\f32} \hat{v}_1\big\|_{L^2}\\
&\quad +t^{-\frac{\mu}{2}} \tau^{\frac{\mu}{2}} \left\||\xi_j|\hat{v}_1\right\|_{L^2} \\
& \lesssim t^{\f32-\frac{3}{1+\varepsilon_1}} \left\|v_1\right\|_{Z, 1, 2}+ t^{\f12-\frac{3}{1+\varepsilon_1}}\left\|v_1\right\|_{\left(1+\varepsilon_1, 2\right)}+t^{-\f{\mu}{2}} \tau^{\f{\mu}{2}}\left\|v_1\right\|_{Z, 1, 2} \\
& \lesssim t^{\f32-\frac{3}{1+\varepsilon_1}} \tau^{\f32}(\left\|v_1\right\|_{Z, 1, 2}+\left\|v_1\right\|_{Z, 1, \left(1+\varepsilon_1, 2\right)}).
\end{aligned}
\end{equation}
Analogously treated as in \eqref{l20'} and \eqref{l20}, by \eqref{l17}-\eqref{l19}, we have that for $\f{14}{5}\leq\mu<3$,
\begin{equation}\label{l24'}
\begin{aligned}
&\big\||\xi|(\xi_1\p_{\xi_1}\hat{v}+\xi_2\p_{\xi_2}\hat{v}+\xi_3\p_{\xi_3})\hat{v}\big\|_{L^2(\R^3)}
\\
&=\big\||\xi|  \sum_{k=1}^{3}\left(
 \xi_k (\partial_{\xi_k} \Psi_1) \hat{v}_1
+ \xi_k \Psi_1 \partial_{\xi_k} \hat{v}_1\right)\big\|_{L^2(\R^3)}\\
&\lesssim \big\|t^{1-\frac{\mu}{2}}\tau^{\f{\mu}{2}}|\xi|\hat{v}_{1}\big\|_{L^{2}(D'_2)}+ \big\|t^{-\f{\mu}{2}}\tau^{\f{\mu}{2}}(\sum_{k=1}^3\xi_k\p_{\xi_k})\hat{v}_{1}\big\|_{L^{2}(D'_2)}\\
&\quad+ \big\|t^{1-\frac{\mu}{2}}\tau|\xi|^{-\f{\mu}{2}+2}\hat{v}_{1}\big\|_{L^{2}(D'_1)}+ \big\|t^{-\frac{\mu}{2}}|\xi|^{-\frac{\mu}{2}}\tau|\xi|(\sum_{k=1}^3\xi_k\p_{\xi_k})\hat{v}_{1}\big\|_{L^{2}(D'_1)}\\
&\quad+ \big\|\tau|\xi|\hat{v}_{1}\big\|_{L^{2}(D'_3)}+ \big\|\tau|\xi|(\sum_{k=1}^3\xi_k\p_{\xi_k})\hat{v}_{1}\big\|_{L^{2}(D'_3)}\\
& \lesssim t^{-\frac{\mu}{2}}\tau^{\frac{\mu}{2}}\big\| t|\xi|  \hat{v}_1(\xi)\big\|_{L^2}+t^{-\frac{\mu}{2}}\tau^{\frac{\mu}{2}}\big\| (\sum_{k=1}^3\xi_k\p_{\xi_k})  \hat{v}_1(\xi)\big\|_{L^2}\\
&\quad+ \big\|(1+t|\xi|)^{-\frac{\mu}{2}}t|\xi|\hat{v}_{1}\big\|_{L^{2}(D'_1)}+ \big\|(1+t|\xi|)^{-\frac{\mu}{2}}(\sum_{k=1}^3\xi_k\p_{\xi_k})\hat{v}_{1}\big\|_{L^{2}(D'_1)}\\
&\quad+ t^{-1}\tau\big\|(1+t|\xi|)^{-\frac{\mu}{2}}\hat{v}_{1}\big\|_{L^{2}(D'_3)}+t^{-1}\tau \big\|(1+t|\xi|)^{-\frac{\mu}{2}}(\sum_{k=1}^3\xi_k\p_{\xi_k})\hat{v}_{1}\big\|_{L^{2}(D'_3)}\\
&\lesssim t^{-\f{\mu}{2}}\tau^{\f{\mu}{2}}(\left\|v_1\right\|_{Z, 1,2}+\left\|v_1\right\|_{Z, 1,(\f{6}{3+\mu}, 2)}),
\end{aligned}
\end{equation}
and for $\mu\geq3$,
\begin{equation}\label{l24}
\begin{aligned}
&\big\||\xi|(\xi_1\p_{\xi_1}\hat{v}+\xi_2\p_{\xi_2}\hat{v}+\xi_3\p_{\xi_3})\hat{v}\big\|_{L^2(\R^3)}
\\
& \lesssim t^{-\frac{\mu}{2}}\tau^{\frac{\mu}{2}}\big\| t|\xi|  \hat{v}_1(\xi)\big\|_{L^2}+t^{-\frac{\mu}{2}}\tau^{\frac{\mu}{2}}\big\| (\sum_{k=1}^3\xi_k\p_{\xi_k})  \hat{v}_1(\xi)\big\|_{L^2}\\
&\quad+ \big\|(1+t|\xi|)^{-\frac{3}{2}}t|\xi|\hat{v}_{1}\big\|_{L^{2}(D'_1)}+ \big\|(1+t|\xi|)^{-\frac{3}{2}}(\sum_{k=1}^3\xi_k\p_{\xi_k})\hat{v}_{1}\big\|_{L^{2}(D'_1)}\\
&\quad+ t^{-1}\tau\big\|(1+t|\xi|)^{-\frac{3}{2}}\hat{v}_{1}\big\|_{L^{2}(D'_3)}+t^{-1}\tau \big\|(1+t|\xi|)^{-\frac{3}{2}}(\sum_{k=1}^3\xi_k\p_{\xi_k})\hat{v}_{1}\big\|_{L^{2}(D'_3)}\\
&\lesssim t^{\f32-\frac{3}{1+\varepsilon_1}} \tau^{\f32}(\left\|v_1\right\|_{Z, 1, 2}+\left\|v_1\right\|_{Z, 1, \left(1+\varepsilon_1, 2\right)}).
\end{aligned}
\end{equation}
Therefore,  \eqref{l22'}-\eqref{l24} imply \eqref{l4}, and Lemma \ref{lem4} follows.
\end{proof}

\section{Proof of Theorem~\ref{YH-1}}
Motivated by \cite{LX} or \cite[Section 5]{LY}, based on Lemmas \ref{lem2}-\ref{lem3} and Lemma \ref{lem4},
for any $T>1$ and $\varepsilon_1\in(0, 1)$, we introduce the function space $X(T)$
$$
X(T)=\{u(t,x): \|u\|_{X(T)}<\infty, \operatorname{supp}_x u(x, t)\subset\{x: |x|\lesssim t\}\},
$$
which is equipped with a norm
\begin{equation}\label{Q1}
\|u\|_{X(T)}:=\begin{cases}
\sup _{t \in[1, T]}\big(t^{\f{\mu}{2}}\|u\|_{Z, 1, 2}+t^{\f{\mu}{2}}\|\p u\|_{Z, 1, 2}\big),&\f{14}{5}\leq\mu<3,\\
\sup _{t \in[1, T]}\big(t^{\f32-\f{3}{1+\varepsilon_1}}\|u\|_{Z, 1, 2}+t^{\f32-\f{3}{1+\varepsilon_1}}\|\p u\|_{Z, 1, 2}\big),&\mu\geq3.
\end{cases}
\end{equation}

Based on Duhamel's principle and the  expression \eqref{equ:q4}, we define the nonlinear mapping
\begin{equation}\label{Q2}
\begin{aligned}
\mathcal{N} u(t, x)&=\ve\Psi_{0}(t, 1, D)u_{0}(x)+\ve\Psi_{1}(t, 1, D) u_{1}(x)
+\int_{1}^{t} \Psi_{1}(t, \tau, D)|u(\tau, x)|^{p} \mathrm{d} \tau
\end{aligned}
\end{equation}
and introduce the  closed subset of $X(T)$
$$
X(T, M)=\left\{u \in X(T):\|u\|_{X(T)} \leq M\ve\right\},
$$
where $M$ is a fixed positive constant to be chosen later (see \eqref{L6} below).
\begin{proposition}\label{prop-1}
Let $\varepsilon_1\in(0,1)$ and $T>1$. For any  $u, v \in X(T)$, the following estimates hold:
\begin{equation}\label{Q3}
\|\mathcal{N} u\|_{X(T)} \leq C\big(\ve\left\|\left(u_{0}, u_{1}\right)\right\|_{\mathcal{X}}+\|u\|_{X(T)}^{p}\big)
\end{equation}
and
\begin{equation}\label{Q4}
\|\mathcal{N} u-\mathcal{N} v\|_{X(T)} \leq C\|u-v\|_{X(T)}\big(\|u\|_{X(T)}^{p-1}+\|v\|_{X(T)}^{p-1}\big),
\end{equation}
where and below the constant $C>0$ is independent of $\ve>0$ and $T>1$, and
\begin{equation}\label{Q5}
\begin{aligned}
\left\|\left(u_{0}, u_{1}\right)\right\|_{\mathcal{X}}:=&\|u_0\|_{Z, 1, 2}+\|u_1\|_{Z, 1, (\f65, 2)}+\|\nabla u_0\|_{Z, 1, 2}
+\|u_1\|_{Z, 1, 2}\\
&+\|u_0+u_1\|_{Z, 1, (1+\varepsilon_1, 2)}+\|u_0+u_1\|_{Z, 1, (\f{6}{3+\mu}, 2)}.
\end{aligned}
\end{equation}
\end{proposition}

\begin{proof}
By \eqref{equ:q17}, \eqref{e:q46'}-\eqref{e:q46} and \eqref{l3}, for any $\varepsilon_1\in(0,1)$, we obtain that for $\mu\geq3$,
\begin{equation}\label{Q6}
\begin{aligned}
\|\mathcal{N} u(t, \cdot)\|_{Z, 1, 2} &\lesssim t^{-\f{\mu}{2}}\ve\left\|u_0\right\|_{Z, 1, 2}
+t^{-\f{\mu}{2}}\ve\left\|u_1\right\|_{Z, 1,(\f65, 2)}
+t^{\f32-\f{3}{1+\ve_1}}\ve\left\|u_0+u_1\right\|_{Z, 1,(1+\varepsilon_1, 2)}\\
&\quad+ t^{\f32-\f{3}{1+\ve_1}}\int_{1}^{t} \tau^{\f32}\left(\left\||u(\tau, \cdot)|^p\right\|_{Z, 1,(1+\varepsilon_1, 2)}+\left\||u(\tau, \cdot)|^p\right\|_{Z, 1,(\f65, 2)}\right)\mathrm{d}\tau,
\end{aligned}
\end{equation}
and for $\f{14}{5}\leq\mu<3$,
\begin{equation}\label{Q6'}
\begin{aligned}
\|\mathcal{N} u(t, \cdot)\|_{Z, 1, 2} &\lesssim t^{-\f{\mu}{2}}\ve\left\|u_0\right\|_{Z, 1, 2}
+t^{-\f{\mu}{2}}\ve\left\|u_1\right\|_{Z, 1,(\f65, 2)}
+t^{-\f{\mu}{2}}\ve\left\|u_0+u_1\right\|_{Z, 1,(\f{6}{3+\mu}, 2)}\\
&\quad+ t^{-\f{\mu}{2}}\int_{1}^{t} \tau^{\f{\mu}{2}}\left(\left\||u(\tau, \cdot)|^p\right\|_{Z, 1,(\f{6}{3+\mu}, 2)}+\left\||u(\tau, \cdot)|^p\right\|_{Z, 1,(\f65, 2)}\right)\mathrm{d}\tau.
\end{aligned}
\end{equation}
For the terms $\left\||u(\tau, \cdot)|^p\right\|_{Z, 1,(1+\varepsilon_1, 2)}$ and $\left\||u(\tau, \cdot)|^p\right\|_{Z, 1,(\f{6}{3+\mu}, 2)}$,
it follows from \cite[(5.1.21)]{LZ} with $\operatorname{supp} \chi(\tau, x)\subset\{(\tau, x)\in[1, T]\times \R^3: |x|\lesssim \tau\}$ that
\begin{equation}\label{Q8}
\left\||u(\tau, \cdot)|^p\right\|_{Z, 1,(\f{6}{3+\mu}, 2)} \lesssim  \left\|u\right\|_{(\f{6}{\mu}, \infty)}\left\|u\right\|^{p-2}_{\infty}\|u(\tau, \cdot)\|_{Z, 1,2}
\end{equation}
and
\begin{equation}\label{Q8}
\left\||u(\tau, \cdot)|^p\right\|_{Z, 1,(1+\varepsilon_1, 2)} \lesssim  \left\|u\right\|_{(\tilde{q}, \infty)}\left\|u\right\|^{p-2}_{\infty}\|u(\tau, \cdot)\|_{Z, 1,2},
\end{equation}
where $\f{6}{\mu}\in(2, +\infty)$ and $\tilde{q}=\tilde{q}(\varepsilon_1)\in(2, +\infty)$ is determined by
\begin{equation}\label{Q9}
\f{1}{1+\varepsilon_1}=\f{1}{2}+\f{1}{\tilde{q}}.
\end{equation}

Substituting \eqref{Q10} and \eqref{Q21} in Subsection \ref{sec2-1-2} into \eqref{Q8} yields
$$
\left\||u(\tau, \cdot)|^p\right\|_{Z, 1,(1+\varepsilon_1, 2)} \lesssim  (\|u(\tau, \cdot)\|_{Z, 1, 2}+\|\p u(\tau, \cdot)\|_{Z, 1, 2})^{p-1}\|u(\tau, \cdot)\|_{Z, 1,2}
$$
and
$$
\left\||u(\tau, \cdot)|^p\right\|_{Z, 1,(\f{6}{3+\mu}, 2)} \lesssim  (\|u(\tau, \cdot)\|_{Z, 1, 2}+\|\p u(\tau, \cdot)\|_{Z, 1, 2})^{p-1}\|u(\tau, \cdot)\|_{Z, 1,2}.
$$
By the definition of norm \eqref{Q1} in space $X(T)$, it holds that for $\mu\geq3$,
\begin{equation}\label{Q23'}
\left\||u(\tau, \cdot)|^p\right\|_{Z, 1,(1+\varepsilon_1, 2)}\lesssim \tau^{p(\f32-\f{3}{1+\ve_1})}\ |u\|_{X(T)}^{p},
\end{equation}
and for $\f{14}{5}\leq\mu<3$,
\begin{equation}\label{Q23''}
\left\||u(\tau, \cdot)|^p\right\|_{Z, 1,(\f{6}{3+\mu}, 2)}\lesssim \tau^{-\f{\mu p}{2}}\ |u\|_{X(T)}^{p}.
\end{equation}
For the term $\left\||u(\tau, \cdot)|^p\right\|_{Z, 1,(\f65, 2)}$,
Lemma 2.1 in \cite{LX} together with the condition  $p>\f53$ implies that
\begin{equation}\label{Q23}
\begin{aligned}
\||u(\tau, \cdot)|^p\|_{Z, 1, (\f65, 2)}&\leq \||u|^{p-1}\|_{(3, \infty)}\|u\|_{Z, 1, 2}\\
&\leq  \||u|^{p-1-\f23}\|_{\infty}\||u|^{\f23}\|_{(3, \infty)}  \|u\|_{Z, 1, 2}\\
&\leq  \|u\|^{p-1-\f23}_{\infty}\|u\|^{\f23}_{(2, \infty)}  \|u\|_{Z, 1, 2}.
\end{aligned}
\end{equation}
Then by \eqref{Q20}-\eqref{Q21}, one can conclude
\begin{equation}\label{Q-24}
\left\||u(\tau, \cdot)|^p\right\|_{Z, 1,(\f65, 2)}\lesssim  (\|u(\tau, \cdot)\|_{Z, 1, 2}+\|\p u(\tau, \cdot)\|_{Z, 1, 2})^{p-1}\|u(\tau, \cdot)\|_{Z, 1,2}\lesssim \tau^{p(\f32-\f{3}{1+\ve_1})}\|u\|_{X(T)}^{p}.
\end{equation}
Substituting \eqref{Q-24} and \eqref{Q23'}-\eqref{Q23''} into \eqref{Q6} and \eqref{Q6'}  yields that for $\mu\geq3$,
\begin{equation}\label{Q25}
\begin{aligned}
\|\mathcal{N} u(t, \cdot)\|_{Z, 1, 2} &\lesssim t^{-\f{\mu}{2}}\ve\left\|u_0\right\|_{Z, 1, 2}
+t^{-\f{\mu}{2}}\ve\left\|u_1\right\|_{Z, 1,(\f65, 2)}
+t^{\f32-\f{3}{1+\ve_1}}\ve\left\|u_0+u_1\right\|_{Z, 1,(1+\varepsilon_1, 2)}\\
&\quad+ t^{\f32-\f{3}{1+\ve_1}}\int_{1}^{t} \tau^{\f32+p(\f32-\f{3}{1+\ve_1})}\mathrm{d}\tau,
\end{aligned}
\end{equation}
and for $\f{14}{5}\leq\mu<3$,
\begin{equation}\label{Q25'}
\begin{aligned}
\|\mathcal{N} u(t, \cdot)\|_{Z, 1, 2} &\lesssim t^{-\f{\mu}{2}}\ve\left\|u_0\right\|_{Z, 1, 2}
+t^{-\f{\mu}{2}}\ve\left\|u_1\right\|_{Z, 1,(\f65, 2)}
+t^{-\f{\mu}{2}}\ve\left\|u_0+u_1\right\|_{Z, 1,(\f{6}{3+\mu}, 2)}\\
&\quad+ t^{-\f{\mu}{2}}\int_{1}^{t} \tau^{\f{\mu}{2}+(-\f{\mu}{2})p}\mathrm{d}\tau.
\end{aligned}
\end{equation}
Observe that $\f32-\f{3}{1+\ve_1}\rightarrow 0+$  as $\varepsilon_1\rightarrow0+$. Hence, for \begin{equation}\label{crit-1}
p>p_f(3)=\f53,
\end{equation}
we can always choose $\varepsilon_1>0$ sufficiently small such that
$$
\f32+p (\f32-\f{3}{1+\ve_1})<-1.
$$
On the other hand, one has $\f{\mu}{2}+(-\f{\mu}{2})p<-1$ if and only if
\begin{equation}\label{crit}
p>1+\f{2}{\mu}.
\end{equation}
Consequently,  for $p>\max\{\f53, 1+\f{2}{\mu}\}$, it follows that
\begin{equation}\label{Q26}
\begin{aligned}
t^{\f{3}{1+\ve_1}-\f32}\|\mathcal{N} u(t, \cdot)\|_{Z, 1, 2} &\lesssim \ve\left\|u_0\right\|_{Z, 1, 2}
+\ve\left\|u_1\right\|_{Z, 1,(\f65, 2)}+\ve\left\|u_0+u_1\right\|_{Z, 1,(1+\varepsilon_1, 2)}+ \|u\|_{X(T)}^p\\
& \lesssim \ve\left\|\left(u_{0}, u_{1}\right)\right\|_{\mathcal{X}}+ \|u\|_{X(T)}^p
\end{aligned}
\end{equation}
and
\begin{equation}\label{Q26'}
\begin{aligned}
t^{\f{\mu}{2}}\|\mathcal{N} u(t, \cdot)\|_{Z, 1, 2} &\lesssim \ve\left\|u_0\right\|_{Z, 1, 2}
+\ve\left\|u_1\right\|_{Z, 1,(\f65, 2)}+\ve\left\|u_0+u_1\right\|_{Z, 1,(\f{6}{3+\mu}, 2)}+ \|u\|_{X(T)}^p\\
& \lesssim \ve\left\|\left(u_{0}, u_{1}\right)\right\|_{\mathcal{X}}+ \|u\|_{X(T)}^p.
\end{aligned}
\end{equation}
Similarly, by \eqref{e:q71'}-\eqref{e:q71} and \eqref{l4}, we obtain that for $\mu\geq3$,
\begin{equation}\label{Q27}
\begin{aligned}
\|\p\mathcal{N}u(t, \cdot)\|_{Z, 1, 2} &\lesssim t^{-\f{\mu}{2}}\ve\left\|\nabla u_0\right\|_{Z, 1, 2}
+t^{-\f{\mu}{2}}\ve\left\| u_1\right\|_{Z, 1, 2}+t^{\f32-\f{3}{1+\ve_1}}\ve\left\|u_0+u_1\right\|_{Z, 1,(1+\ve_1, 2)}\\
&\quad+ t^{\f32-\f{3}{1+\ve_1}}\int_{1}^{t} \tau^{\f32}(\left\||u(\tau, \cdot)|^p\right\|_{Z, 1,2}
+\left\||u(\tau, \cdot)|^p\right\|_{Z, 1,(1+\ve_1, 2)})\mathrm{d}\tau,
\end{aligned}
\end{equation}
and for $\f{14}{5}\leq\mu<3$,
\begin{equation}\label{Q27'}
\begin{aligned}
\|\p\mathcal{N}u(t, \cdot)\|_{Z, 1, 2} &\lesssim t^{-\f{\mu}{2}}\ve\left\|\nabla u_0\right\|_{Z, 1, 2}
+t^{-\f{\mu}{2}}\ve\left\| u_1\right\|_{Z, 1, 2}+t^{-\f{\mu}{2}}\ve\left\|u_0+u_1\right\|_{Z, 1,(\f{6}{3+\mu}, 2)}\\
&\quad+ t^{-\f{\mu}{2}}\int_{1}^{t} \tau^{\f{\mu}{2}}(\left\||u(\tau, \cdot)|^p\right\|_{Z, 1,2}
+\left\||u(\tau, \cdot)|^p\right\|_{Z, 1,(\f{6}{3+\mu}, 2)})\mathrm{d}\tau.
\end{aligned}
\end{equation}
For the term $\left\||u(\tau, \cdot)|^p\right\|_{Z, 1, 2}$,
it follows from \eqref{Q21} and the condition  $p>1$ that
\begin{equation}\label{Q23}
\begin{aligned}
\||u(\tau, \cdot)|^p\|_{Z, 1, 2}&\leq \||u|^{p-1}\|_{\infty}\|u\|_{Z, 1, 2}\\
&\leq  \|u\|^{p-1}_{\infty}\|u\|_{Z, 1, 2}\\
&\lesssim  (\|u(\tau, \cdot)\|_{Z, 1, 2}+\|\p u(\tau, \cdot)\|_{Z, 1, 2})^{p-1}\|u(\tau, \cdot)\|_{Z, 1,2}.
\end{aligned}
\end{equation}
This, together with \eqref{Q23'}-\eqref{Q23''} and \eqref{Q27}-\eqref{Q27'}, gives that for $\mu\geq3$,
$$
\begin{aligned}
\|\p \mathcal{N}u(t, \cdot)\|_{Z, 1, 2} &\lesssim t^{-\f{\mu}{2}}\ve\left\|\nabla u_0\right\|_{Z, 1, 2}
+t^{-\f{\mu}{2}}\ve\left\|u_1\right\|_{Z, 1, 2}+t^{\f32-\f{3}{1+\ve_1}}\ve\left\|u_0+u_1\right\|_{Z, 1,2}\\
&\quad+ t^{\f32-\f{3}{1+\ve_1}}\|u\|_{X(T)}^p\int_{1}^{t} \tau^{\f32+p(\f32-\f{3}{1+\ve_1})}\mathrm{d}\tau
\end{aligned}
$$
and for $\f{14}{5}\leq\mu<3$,
$$
\begin{aligned}
\|\p \mathcal{N}u(t, \cdot)\|_{Z, 1, 2} &\lesssim t^{-\f{\mu}{2}}\ve\left\|\nabla u_0\right\|_{Z, 1, 2}
+t^{-\f{\mu}{2}}\ve\left\|u_1\right\|_{Z, 1, 2}+t^{-\f{\mu}{2}}\ve\left\|u_0+u_1\right\|_{Z, 1,2}\\
&\quad+ t^{-\f{\mu}{2}}\|u\|_{X(T)}^p\int_{1}^{t} \tau^{\f{\mu}{2}+(-\f{\mu}{2})p}\mathrm{d}\tau.
\end{aligned}
$$
For $p>\max\{\f53, 1+\f{2}{\mu}\}$, we can choose $\varepsilon_1>0$ sufficiently small  such that $\f32+p(\f32-\f{3}{1+\ve_1})<-1$ and $\f{\mu}{2}+(-\f{\mu}{2})p<-1$ hold.
Hence, for $p>\max\{\f53, 1+\f{2}{\mu}\}$, it holds that
\begin{equation}\label{Q24}
\begin{aligned}
t^{\f{3}{1+\ve_1}-\f32}\|\p \mathcal{N} u(t, \cdot)\|_{Z, 1,2} &\lesssim \ve\left\|\nabla u_0\right\|_{Z, 1,2}+\ve\left\|u_1\right\|_{Z, 1,2}
+\ve\left\|u_0+u_1\right\|_{Z, 1, (1+\ve_1, 2)} +\|u\|_{X(T)}^p\\
&  \lesssim \ve\left\|\left(u_{0}, u_{1}\right)\right\|_{\mathcal{X}}+ \|u\|_{X(T)}^p
\end{aligned}
\end{equation}
and
\begin{equation}\label{Q24'}
\begin{aligned}
t^{\f{\mu}{2}}\|\p \mathcal{N} u(t, \cdot)\|_{Z, 1,2} &\lesssim \ve\left\|\nabla u_0\right\|_{Z, 1,2}+\ve\left\|u_1\right\|_{Z, 1,2}
+\ve\left\|u_0+u_1\right\|_{Z, 1, (\f{6}{3+\mu}, 2)} +\|u\|_{X(T)}^p\\
&  \lesssim \ve\left\|\left(u_{0}, u_{1}\right)\right\|_{\mathcal{X}}+ \|u\|_{X(T)}^p.
\end{aligned}
\end{equation}
Collecting \eqref{Q26}-\eqref{Q26'} and \eqref{Q24}-\eqref{Q24'}, one can obtain  \eqref{Q3}.

We next prove the contractive estimate \eqref{Q4}. As  in \eqref{Q6} and \eqref{Q6'}, it follows that for $u, v\in X(T)$ and $p>\max\{\f53, 1+\f{2}{\mu}\}$,
\begin{equation}\label{L20}
\begin{aligned}
\|(\mathcal{N} u-\mathcal{N} v)(t, \cdot)\|_{Z, 1, 2}& \lesssim t^{\f32-\f{3}{1+\ve_1}}\int_{1}^{t} \tau^{\f32}\left\||u-v|(\tau, \cdot)(|u|+|v|)^{p-1}(\tau, \cdot)\right\|_{Z, 1,(1+\varepsilon_1, 2)}\mathrm{d}\tau\\
&\quad+t^{\f32-\f{3}{1+\ve_1}}\int_{1}^{t} \tau^{\f32}\left\||u-v|(\tau, \cdot)(|u|+|v|)^{p-1}(\tau, \cdot)\right\|_{Z, 1,(\f65, 2)}\mathrm{d}\tau\\
&\lesssim t^{\f32-\f{3}{1+\ve_1}} \|u-v\|_{X(T)}\big(\|u\|_{X(T)}^{p-1}+\|v\|_{X(T)}^{p-1}\big)\int_{1}^{t} \tau^{\f32+p(\f32-\f{3}{1+\ve_1})}\mathrm{d}\tau\\
&\lesssim t^{\f32-\f{3}{1+\ve_1}} \|u-v\|_{X(T)}\big(\|u\|_{X(T)}^{p-1}+\|v\|_{X(T)}^{p-1}\big)
\end{aligned}
\end{equation}
and
\begin{equation}\label{L20'}
\begin{aligned}
\|(\mathcal{N} u-\mathcal{N} v)(t, \cdot)\|_{Z, 1, 2}& \lesssim t^{-\f{\mu}{2}}\int_{1}^{t} \tau^{\f{\mu}{2}}\left\||u-v|(\tau, \cdot)(|u|+|v|)^{p-1}(\tau, \cdot)\right\|_{Z, 1,(\f{6}{3+\mu}, 2)}\mathrm{d}\tau\\
&\quad+t^{-\f{\mu}{2}}\int_{1}^{t} \tau^{\f{\mu}{2}}\left\||u-v|(\tau, \cdot)(|u|+|v|)^{p-1}(\tau, \cdot)\right\|_{Z, 1,(\f65, 2)}\mathrm{d}\tau\\
&\lesssim t^{-\f{\mu}{2}} \|u-v\|_{X(T)}\big(\|u\|_{X(T)}^{p-1}+\|v\|_{X(T)}^{p-1}\big)\int_{1}^{t} \tau^{\f{\mu}{2}+(-\f{\mu}{2})p}\mathrm{d}\tau\\
&\lesssim t^{-\f{\mu}{2}} \|u-v\|_{X(T)}\big(\|u\|_{X(T)}^{p-1}+\|v\|_{X(T)}^{p-1}\big).
\end{aligned}
\end{equation}
On the other hand, by using the same argument as in the derivation of \eqref{Q27} and \eqref{Q27'}, we have that for $u, v \in X(T)$ and $p>\max\{\f53, 1+\f{2}{\mu}\}$,
\begin{equation}\label{L21}
\begin{aligned}
\|\p(\mathcal{N} u-\mathcal{N} v)(t, \cdot)\|_{Z, 1, 2}&
\lesssim t^{\f32-\f{3}{1+\ve_1}}\int_{1}^{t} \tau^{\f32}\left\||u-v|(\tau, \cdot)(|u|+|v|)^{p-1}(\tau, \cdot)\right\|_{Z, 1,2}\mathrm{d}\tau\\
&\quad+t^{\f32-\f{3}{1+\ve_1}}\int_{1}^{t} \tau^{\f32}\left\||u-v|(\tau, \cdot)(|u|+|v|)^{p-1}(\tau, \cdot)\right\|_{Z, 1,(1+\ve_1, 2)}\mathrm{d}\tau\\
&\lesssim t^{\f32-\f{3}{1+\ve_1}} \|u-v\|_{X(T)}\big(\|u\|_{X(T)}^{p-1}+\|v\|_{X(T)}^{p-1}\big)\int_{1}^{t} \tau^{\f32+p(\f32-\f{3}{1+\ve_1})}\mathrm{d}\tau\\
&\lesssim t^{\f32-\f{3}{1+\ve_1}} \|u-v\|_{X(T)}\big(\|u\|_{X(T)}^{p-1}+\|v\|_{X(T)}^{p-1}\big)
\end{aligned}
\end{equation}
and
\begin{equation}\label{L21'}
\begin{aligned}
\|\p(\mathcal{N} u-\mathcal{N} v)(t, \cdot)\|_{Z, 1, 2}&
\lesssim t^{-\f{\mu}{2}}\int_{1}^{t} \tau^{\f{\mu}{2}}\left\||u-v|(\tau, \cdot)(|u|+|v|)^{p-1}(\tau, \cdot)\right\|_{Z, 1,2}\mathrm{d}\tau\\
&\quad+t^{-\f{\mu}{2}}\int_{1}^{t} \tau^{\f{\mu}{2}}\left\||u-v|(\tau, \cdot)(|u|+|v|)^{p-1}(\tau, \cdot)\right\|_{Z, 1,(\f{6}{3+\mu}, 2)}\mathrm{d}\tau\\
&\lesssim t^{-\f{\mu}{2}} \|u-v\|_{X(T)}\big(\|u\|_{X(T)}^{p-1}+\|v\|_{X(T)}^{p-1}\big)\int_{1}^{t} \tau^{\f{\mu}{2}+(-\f{\mu}{2})p}\mathrm{d}\tau\\
&\lesssim t^{-\f{\mu}{2}} \|u-v\|_{X(T)}\big(\|u\|_{X(T)}^{p-1}+\|v\|_{X(T)}^{p-1}\big).
\end{aligned}
\end{equation}
Therefore, \eqref{Q4} is proved, and Proposition \ref{prop-1} follows by \eqref{Q1}-\eqref{Q2} and \eqref{Q3}.
\end{proof}

Based on Proposition \ref{prop-1}, we now turn to prove Theorem \ref{YH-1}.
\vskip 0.1 true cm

\begin{proof}
[Proof of Theorem~\ref{YH-1}]
 Choosing $M=3C\left\|\left(u_{0}, u_{1}\right)\right\|_{\mathcal{X}}$ in $X(T, M)$, where $C$ is the positive constant appeared in
 Proposition \ref{prop-1}, we have that for any $u \in X(T, M)$,
\begin{equation}\label{L6}
\begin{aligned}
\|\mathcal{N} u\|_{X(T)}
& \leq C\ve\left\|\left(u_{0}, u_{1}\right)\right\|_{\mathcal{X}}+C\|u\|_{X(T)}^{p} \\
& \leq C \ve\left\|\left(u_{0}, u_{1}\right)\right\|_{\mathcal{X}}+C\left(3 C\ve\left\|\left(u_{0}, u_{1}\right)\right\|_{\mathcal{X}}\right)^{p} \\
& \leq\left(C+3^{p} C^{p+1}\left\|\left(u_{0}, u_{1}\right)\right\|_{\mathcal{X}}^{p-1}\ve^{p-1}\right)\ve\left\|\left(u_{0}, u_{1}\right)\right\|_{\mathcal{X}} \\
& \leq 3 C\ve\left\|\left(u_{0}, u_{1}\right)\right\|_{\mathcal{X}},
\end{aligned}
\end{equation}
where $\ve\le \big(\f{2}{3^pC^p}\big)^{\f{1}{p-1}}\f{1}{1+\left\|\left(u_{0}, u_{1}\right)\right\|_{\mathcal{X}}}$. Similarly, if $\ve\le\f{{(4C)}^{-\f{1}{p-1}}}{3C(1+\left\|\left(u_{0}, u_{1}\right)\right\|_{\mathcal{X}})}$, it follows that
\begin{equation}\label{L7}
\begin{aligned}
\|\mathcal{N} u-\mathcal{N} v\|_{X(T)} &\leq C\|u-v\|_{X(T)}\left(\|u\|_{X(T)}^{p-1}+\|v\|_{X(T)}^{p-1}\right)\\
&\leq C\|u-v\|_{X(T)}\cdot2\left(3 C\ve\left\|\left(u_{0}, u_{1}\right)\right\|_{\mathcal{X}}\right)^{p-1}\\
&\leq \f{1}{2}\|u-v\|_{X(T)}.
\end{aligned}
\end{equation}
Define $$\varepsilon_0=\min\{\f{\big(\f{2}{3^pC^p}\big)^{\f{1}{p-1}}}{1+\left\|\left(u_{0}, u_{1}\right)\right\|_{\mathcal{X}}},\f{{(4C)}^{-\f{1}{p-1}}}{3C(1+\left\|\left(u_{0}, u_{1}\right)\right\|_{\mathcal{X}})}\}.$$
Then for any $\ve\leq \varepsilon_0$,
we have that $\mathcal{N}$ is a contraction mapping on $X(T, M)$ into itself. Therefore, there exists a unique solution $u \in X(T, M)$ such that
$\mathcal{N}u=u$. Hence,  $u$ is a  solution  of \eqref{equ:eff1}. The independence of the constant $C>0$ with respect to $T$ implies that $u$ is a global solution of \eqref{equ:eff1}. This completes the proof of Theorem \ref{YH-1}.
\end{proof}

\vskip 0.1 true cm

{\bf Data availability statement}. All data that support the findings of this study
are included within the article (and any supplementary files).


\end{document}